\documentclass[11pt]{article}

\usepackage{fullpage}
\usepackage[utf8]{inputenc}
\usepackage[T1]{fontenc}
\usepackage{amsmath,amsfonts,amssymb,amsthm,bbm}
\usepackage{stackrel}
\usepackage{enumerate}
\usepackage{graphicx}
\usepackage{hyperref}
\usepackage{subfigure}
\usepackage{color}
\usepackage{stmaryrd}

\newtheorem{theorem}{Theorem}
\newtheorem{corollary}[theorem]{Corollary}
\newtheorem{proposition}[theorem]{Proposition}

\newtheorem{lemma}[theorem]{Lemma}
\newtheorem{definition}[theorem]{Definition}

\newtheorem{remark}[theorem]{Remark}
\newtheorem{assumption}{Assumption}

\numberwithin{theorem}{section}
\numberwithin{equation}{section}
\numberwithin{figure}{section}

\newcommand{\ve}{\varepsilon}
\newcommand{\ind}{\mathbbm{1}}

\newcommand{\ul}{\underline}
\newcommand{\ol}{\overline}

\DeclareMathOperator{\rad}{rad}

\newcommand{\calE}{\mathcal{E}}

\newcommand{\calH}{\mathcal{H}}
\newcommand{\calW}{\mathcal{W}}
\newcommand{\calN}{\mathcal{N}}
\newcommand{\calP}{\mathcal{P}}
\newcommand{\calD}{\mathcal{D}}
\newcommand{\calT}{\mathcal{T}}

\newcommand{\ZZ}{\mathbb{Z}}
\newcommand{\RR}{\mathbb{R}}
\newcommand{\CC}{\mathbb{C}}

\newcommand{\PP}{\mathbb{P}}

\newcommand{\TT}{\mathbb{T}} % triangular lattice
\newcommand{\Ball}{B} % ball
\newcommand{\Ann}{A} % annulus
\newcommand{\din}{\partial^{\textrm{in}}} % inner boundary
\newcommand{\dout}{\partial^{\textrm{out}}} % outer boundary
\newcommand{\cluster}{\mathcal{C}}
\newcommand{\lclus}{\cluster^{\textrm{max}}} % largest cluster
\newcommand{\Ch}{\mathcal{C}_H} % horizontal crossing event
\newcommand{\Cv}{\mathcal{C}_V} % vertical crossing event
\newcommand{\circuitevent}{\mathcal{O}} % event that an annulus has a black circuit
\newcommand{\colorseq}{\mathfrak{S}} % set of colour sequences
\newcommand{\arm}{\mathcal{A}} % arm event
\newcommand{\net}{\mathcal{N}} % event for existence of net

\newcommand{\Cov}{\textrm{Cov}}
\newcommand{\Var}{\textrm{Var}}
\newcommand{\next}{\psi}

\newcommand{\PPh}{\overline{\mathbb{P}}}
\newcommand{\EEh}{\overline{\mathbb{E}}}

\newcommand{\NET}{\textrm{NET}}
\newcommand{\NETB}{\textrm{NETB}}
\newcommand{\OCP}{\textrm{OCP}}
\newcommand{\VC}{\textrm{VC}}
\newcommand{\I}{\textrm{I}}
\newcommand{\NI}{\textrm{NI}}
\newcommand{\OP}{\textrm{OP}}

\newcommand{\lra}{\leftrightarrow}

\begin{document}

\title{Near-critical percolation with heavy-tailed impurities,\\ forest fires and frozen percolation}

\author{Jacob van den Berg\footnote{CWI, VU University Amsterdam, and NYU Abu Dhabi; E-mail: \texttt{J.van.den.Berg@cwi.nl}.}, Pierre Nolin\footnote{City University of Hong Kong; E-mail: \texttt{bpmnolin@cityu.edu.hk}. Partially supported by a GRF grant from the Research Grants Council of the Hong Kong SAR (project CityU11304718).}}

\date{}

\maketitle

\begin{abstract}
Consider critical site percolation on a ``nice'' planar lattice: each vertex is occupied with probability $p = p_c$, and vacant with probability $1 - p_c$. Now, suppose that additional vacancies (``holes'', or ``impurities'') are created, independently, with some small probability, i.e. the parameter $p_c$ is replaced by $p_c - \ve$, for some small $\ve > 0$. A celebrated result by Kesten \cite{Ke1987} says, informally speaking, that on scales below the characteristic length $L(p_c - \ve)$, the connection probabilities remain of the same order as before. We prove a substantial and subtle generalization to the case where the impurities are not only microscopic, but allowed to be ``mesoscopic''.

This generalization, which is also interesting in itself, was motivated by our study of models of forest fires (or epidemics). In these models, all vertices are initially vacant, and then become occupied at rate $1$. If an occupied vertex is hit by lightning, which occurs at a (typically very small) rate $\zeta$, its entire occupied cluster burns immediately, so that all its vertices become vacant.

Our results for percolation with impurities turn out to be crucial for analyzing the behavior of these forest fire models near and beyond the critical time (i.e. the time after which, in a forest without fires, an infinite cluster of trees emerges). In particular, we prove (so far, for the case when burnt trees do not recover) the existence of a sequence of ``exceptional scales'' (functions of $\zeta$). For forests on boxes with such side lengths, the impact of fires does not vanish in the limit as $\zeta \searrow 0$.

\bigskip

\end{abstract}

\tableofcontents

\section{Introduction and main results} \label{sec:intro}

Self-organized criticality is a fascinating phenomenon that may be used to explain the emergence of ``complexity'' (in particular, fractal shapes) in nature. It refers, roughly speaking, to the spontaneous (approximate) arising of a critical regime without any fine-tuning of a parameter. Numerous works have been devoted to it, mostly in statistical physics (see e.g. \cite{Ba1996, Je1998}, and the references therein), but also on the mathematical side.

In various models where this phenomenon occurs, the (near-) critical regime of independent percolation seems to play a crucial role, even though this is not obvious at all from the rules (dynamics) of the process. An example is a model for the displacement of oil by water in a random medium \cite{WW1983, DSV2009}. Another paradigmatic example, much less understood than the previous one, is a mathematical model of forest fires, or, more generally, of excitable media which also include certain epidemics (where infections from outside the population are rare, but spread out very fast) and neuronal or sensor / communication networks; such models were introduced by Drossel and Schwabl \cite{DrSc1992} in 1992. In the present paper we study versions of such processes, where we focus on a model where burnt trees cannot be ``replaced'' by new trees (or, in a sensor / communication network context, each node, i.e. sensor-transmitter in the network, can only once send a signal to neighboring nodes). We will refer to this version as ``forest fires without recovery'' (abbreviated as \textbf{FFWoR}).

Even though forest fire processes attracted a lot of attention, very little is known about their long-time behavior. They are notoriously difficult to study, due to the existence of competing effects on the connectivity of the forest: since the rate of lightning is tiny, large connected components of trees can arise, and when such components eventually burn, they create lasting ``scars'' on the lattice which seem to function as ``fire lanes'', hindering the appearance of new large components. It turns out that, apart from exceptional cases, these scars are essentially only formed near the so-called critical percolation time (but still, in some sense, at many different ``time scales''). Due to this non-monotonicity, standard tools from statistical mechanics for models on lattices cannot be used. Hence, new techniques and ideas are required to understand rigorously the effect of large-scale connections, which play a central role in the spread of fires.

\subsection{Frozen percolation and forest fire processes}

We now describe in more detail the processes studied in, or relevant for, this paper. First, for the study of forest fire models, it appears to be very convenient to compare (couple) them with the classical percolation model, introduced by Broadbent and Hammersley \cite{BH1957} in 1957. More precisely, we consider Bernoulli site percolation with parameter $p \in [0,1]$ on a connected, countably infinite, graph $G = (V, E)$. In this model, each vertex $v \in V$ is occupied (or open, denoted by $1$), with probability $p$, and vacant (or closed, denoted by $0$), with probability $1-p$, independently of the other vertices. There is a critical value $p_c = p_c^{\textrm{site}}(G) \in [0,1]$ for the parameter $p$, below which (almost surely) all occupied clusters are finite, and above which there may be an infinite occupied cluster. This model (and variations, such as bond percolation) has been widely studied, especially on ``nice'' planar lattices like the square and the triangular lattices, and on the hypercubic lattices $\ZZ^d$, $d \geq 3$, with nearest-neighbor edges (the vertices of this lattice are the points with integer coordinates, and two such points $v$, $v'$ are connected by an edge \emph{iff} they differ along exactly one coordinate, by $\pm 1$).

The following model, which we will call the \textbf{$N$-volume-frozen percolation} model, or, sometimes, simply parameter-$N$ model, was studied in \cite{BN2015, BKN2015}, motivated by work by Aldous \cite{Al2000} (who in turn was inspired by phenomena concerning sol-gel transitions \cite{St1943}). It has a (typically very large) parameter $N \geq 1$, and it is defined in terms of i.i.d. random variables $(\tau_v)_{v \in V}$ uniformly distributed on $[0,1]$. Each vertex $v$ is vacant at time $0$, and becomes occupied at time $\tau_v$ \emph{unless} some neighbor of $v$ already belongs to an occupied cluster with size (i.e. number of vertices) at least $N$ (in which case $v$ remains vacant). In other words, a cluster stops growing as soon as it has size $\geq N$: such a large cluster, together with its boundary, is said to be frozen, or ``giant''.

What can we say about the probability that a given vertex eventually (i.e. at time $1$) belongs to a giant cluster? Of course, this is a function of $N$, and we are interested in what happens as $N \to \infty$. Does the above-mentioned probability go to $0$? Or is it bounded away from $0$? What is, typically, the final size (at time $1$) of the cluster of a given vertex? Of course, it cannot be larger than $d (N-1) + 1$, where $d$ is the maximal degree of the graph, but is it typically smaller than $N$, and even of smaller order than $N$?

In the case where the graph is a binary tree, it was shown in \cite{BKN2012} (by extending ideas from \cite{Al2000}) that with high probability as $N \to \infty$, the final cluster of a given vertex is either giant (i.e. has size $\geq N$) or ``microscopic'' (of order $1$). For the square lattice (and other ``nice'' planar lattices), it was shown in \cite{BN2015} that there is a sequence of functions ($\sqrt{N} \asymp$)~$f_1(N) \ll f_2(N) \ll \ldots$ called exceptional scales such that the following holds: for each $i$, for the model in the box with side length $f_i(N)$, the probability that $0$ is eventually in a giant cluster is bounded away from $0$ as $N \to \infty$; however, for every function $f(N)$ with $f_i(N) \ll f(N) \ll f_{i+1}(N)$, the above-mentioned probability goes to $0$ as $N \to \infty$. In \cite{BKN2015}, it was shown (in the particular case of the triangular lattice) that this probability also tends to $0$ if $f(N) \gg f_i(N)$ for every $i \geq 1$, as well as for the process on the entire lattice.

As suggested above, the $N$-volume-frozen percolation model above can be interpreted as a simple model for gelation (sol-gel transition). It could (but see remarks below) also be interpreted as a model of forest fires (or epidemics) without recovery, where the ignitions (infections) are very rare ($N$ being very large), but once an ignition takes place, the fire spreads very fast, in effect wiping out instantaneously the entire occupied cluster. Initially, at each vertex there is one seed; once this seed has become a plant and this plant is burnt by fire, no other plant will grow at its location, and neither at neighboring locations. The role of the parameter $N$, i.e. that a cluster with size $\leq N$ cannot burn, is not very realistic for this interpretation: it makes the dynamic quite rigid. Also, the rule that nothing can grow any more on the sites along the external boundary of a burnt cluster of plants looks a bit artificial.

More realistic, as a model of forest fires without recovery, is the following, where time is now indexed by $[0, +\infty)$, and which has a (typically very small) parameter $\zeta > 0$. Again, at time $0$, all vertices are vacant (or, better, contain a seed). Independently of each other, they become occupied (the seeds become plants) at rate $1$. Each occupied vertex (plant) is ignited at rate $\zeta$, in which case its entire occupied cluster is instantaneously burnt (and remains so). Note that in this process, there are three possible states for a vertex: $0$ (vacant, or ``seed'': initially, all vertices are in this state), $1$ (occupied, or ``plant''), and $-1$ (burnt). We will denote this process by the earlier-mentioned abbreviation FFWoR. It is clear from the description above that, for each $\zeta > 0$, the probability that a given vertex is eventually in state $-1$ (i.e. burnt) is equal to $1$. The resulting configuration is depicted in Figure \ref{fig:final_config}.

\begin{figure}[t]
\begin{center}

\vspace{-1.2cm}
\includegraphics[width=.85\textwidth]{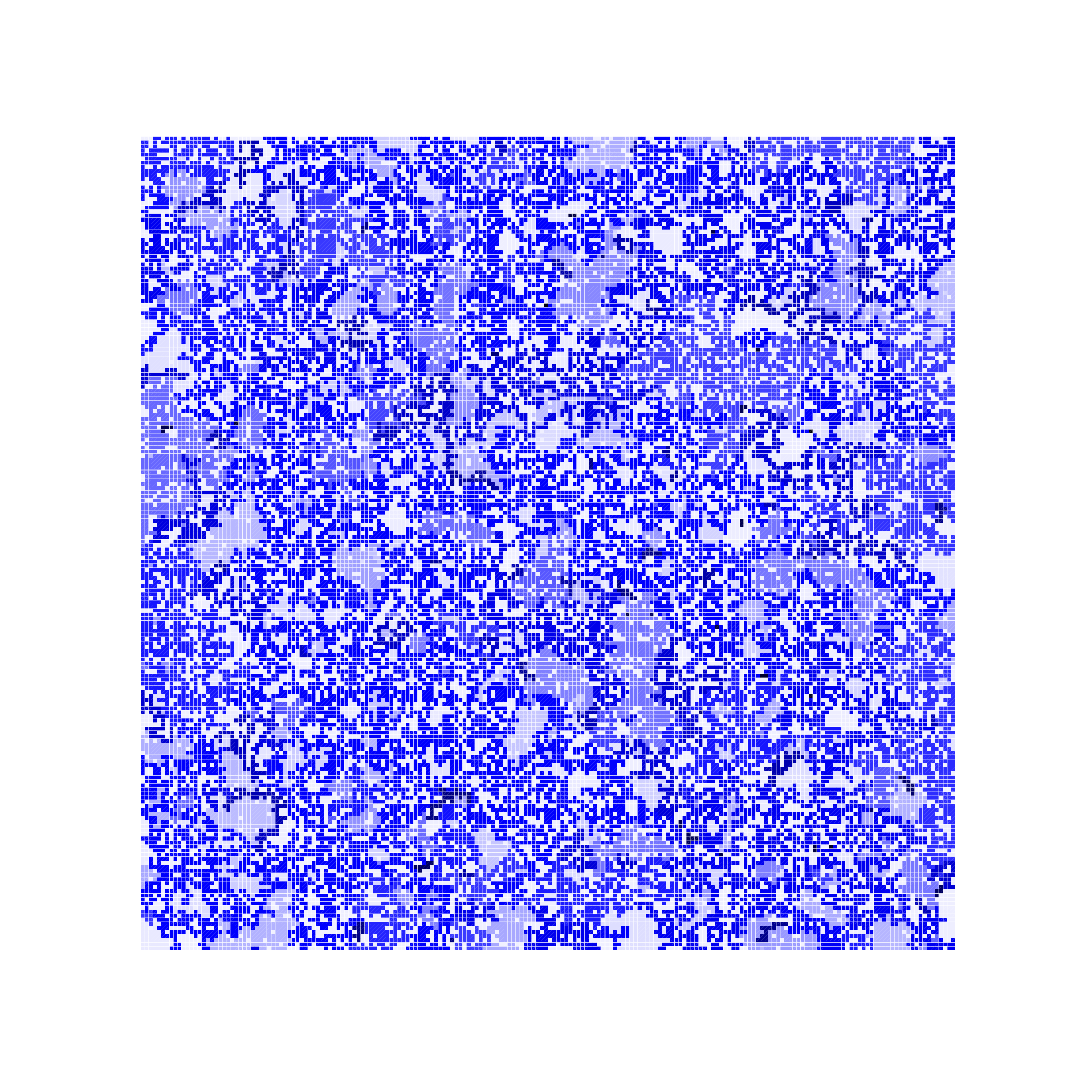}
\vspace{-1.8cm}
\caption{\label{fig:final_config} Final configuration (i.e. at time $t = +\infty$) for the forest fires without recovery process on $\ZZ^2$ with rate $\zeta = 0.01$, in a box with side length $200$. All sites are burnt, a lighter shade of blue corresponding to a later time of burning.}

\end{center}
\end{figure}

Analogs of (some of) the earlier questions are the following. Does, for each $t > 0$, the probability that a given vertex burns before time $t$ go to $0$ as $\zeta \searrow 0$? Or are there values of $t$ for which this probability is bounded away from $0$ as $\zeta \searrow 0$? At first sight, one might expect that this model can be analyzed in the same way as the ``parameter-$N$ model'', with (roughly) $N$ replaced by $\frac{1}{\zeta}$. Apart from the fact that this replacement is too naive, the arguments become considerably more complicated, due to quite delicate problems concerning what we call ``near-critical percolation with impurities'', as we heuristically indicate now.

\subsection{Heuristic derivation of exceptional scales} \label{sec:intro_heuristics}

We will compare the parameter-$N$ model in a box with side length $f(N)$ with the FFWoR model with parameter $\zeta$ in a box with side length $g(\zeta)$. The heuristic arguments (made rigorous in \cite{BN2015}) for the parameter-$N$ model on the square lattice are roughly as follows. If $f(N) = C \sqrt{N}$ with $C > 1$, then, clearly (since the total number of vertices is $> N$) at least one burning / freezing event will take place. Hence, a positive fraction ($\geq \frac{1}{C^2}$) of the vertices will freeze, which suggests (and this can be quite easily proved) that the probability that $0$ eventually freezes / burns is bounded away from $0$ (in fact, has limit $\frac{1}{C^2}$) as $N \to \infty$. Now, we try to find a function $f(N) \gg \sqrt{N}$ where this also holds, i.e. where the probability that $0$ burns / freezes is bounded away from $0$ as $N \to \infty$. To do this, let $\tau$ denote the first time that a giant cluster arises (recall that the time line in this model is the interval $[0,1]$). The biggest cluster at time $\tau$ has size roughly $\theta(\tau) f(N)^2$, where $\theta(\tau)$ is the probability (for Bernoulli site percolation on the whole lattice) that $0$ lies in an infinite occupied cluster at time $\tau$. We thus want
\begin{equation} \label{eq:intro_heuristics1}
\theta(\tau) f(N)^2 \asymp N.
\end{equation}
The freezing of this cluster disconnects the box into ``islands'' of diameter roughly of order $L(\tau)$ (the characteristic length for percolation with parameter $\tau$: see Section \ref{sec:percolation} for precise definitions of this and other notions). So if $\tau$ is such that $L(\tau)$ is of order $\sqrt{N}$, then $0$ will (after this freezing event) typically be in the interior of an island with diameter of order $\sqrt{N}$, and hence be in a similar situation as the previous case (i.e. the case where the box has side length $C \sqrt{N}$), so that the probability that $0$ freezes is bounded away from $0$. So we may choose $f(N)$ such that besides \eqref{eq:intro_heuristics1}, also the following equation holds:
\begin{equation} \label{eq:intro_heuristics2}
L(\tau) \asymp \sqrt{N}.
\end{equation}
A celebrated and classical result by Kesten \cite{Ke1987} says that $\theta(\tau) \asymp \pi_1(L(\tau))$, where $\pi_1(.)$ is the one-arm probability at $p_c$ (see \eqref{eq:def_pi} below). Combining \eqref{eq:intro_heuristics1} and \eqref{eq:intro_heuristics2} with this result gives $\pi_1(\sqrt{N}) f(N)^2 \asymp N$, hence
\begin{equation} \label{eq:intro_heuristics3}
f(N) \asymp \sqrt{\frac{N}{\pi_1(\sqrt{N})}}.
\end{equation}
As a conclusion, if $f(N)$ is indeed of this order, then, for the parameter-$N$ model in a box with side length $f(N)$, the probability that $0$ freezes is bounded away from $0$ as $N \to \infty$. We say that $\sqrt{N}$ is the first exceptional scale, and \eqref{eq:intro_heuristics3} above is the second. Iterating this procedure produces a sequence of exceptional scales.

We now turn to the FFWoR model with parameter $\zeta > 0$, and we will see that already the ``construction'' of the second exceptional scale involves new and delicate technical difficulties. First of all, similarly as in the parameter-$N$ model, it is not hard to see that for the process in a box with side length $\asymp \frac{1}{\sqrt{\zeta}}$, for each $t > t_c$, the probability that $0$ burns before time $t$ is bounded away from $0$ as $\zeta \searrow 0$ (here, time $t$ is related to the percolation parameter $p$ by $p = 1 - e^{-t}$; in particular, $p_c = 1 - e^{-t_c}$). The heuristic argument to find the next scale (call it $g(\zeta)$ for the moment) for which this happens is now as follows. Let $\tau$ be the first time that a big burning takes place, after which $0$ is separated from the boundary of the box. Analogously to the beginning of the argument for the parameter-$N$ model, the size of the biggest cluster at time $\tau > t_c$ is
\begin{equation} \label{eq:intro_heuristics4}
\theta(\tau) g(\zeta)^2.
\end{equation}
At time $t_c$ it is much smaller, but at time $\frac{t_c + \tau}{2}$ it has already a size of order \eqref{eq:intro_heuristics4}. So, to have a reasonable chance that the cluster burns ``near'' time $\tau$, we need
\begin{equation} \label{eq:intro_heuristics5}
\zeta (\tau - t_c) \theta(\tau) g(\zeta)^2 \asymp 1.
\end{equation}
We apply again the earlier-mentioned relation by Kesten, which in the current notation is 
\begin{equation} \label{eq:intro_heuristics6}
\theta(\tau) \asymp \pi_1(L(\tau)),
\end{equation}
as well as the following relation, also established by Kesten:
\begin{equation} \label{eq:intro_heuristics7}
(\tau - t_c) \pi_4(L(\tau)) L(\tau)^2 \asymp 1
\end{equation}
($\pi_4(.)$ is the probability at $p_c$ of observing four arms with alternating types from a given vertex, see \eqref{eq:def_pi}). Combining \eqref{eq:intro_heuristics5}, \eqref{eq:intro_heuristics6} and \eqref{eq:intro_heuristics7} gives
\begin{equation} \label{eq:intro_heuristics8}
\zeta \pi_1(L(\tau)) g(\zeta)^2 \asymp \pi_4(L(\tau)) L(\tau)^2.
\end{equation}
Analogously as for the parameter-$N$ case, we want to take $g(\zeta)$ such that $L(\tau) \asymp \frac{1}{\sqrt{\zeta}}$ (so that after the burning, $0$ finds itself roughly in the same situation as before, i.e. the probability that $0$ burns before time $t$ does not vanish as $\zeta \searrow 0$). Plugging this requirement into \eqref{eq:intro_heuristics8}, we get, as an analog of \eqref{eq:intro_heuristics3},
\begin{equation} \label{eq:intro_heuristics9}
g(\zeta) \asymp \frac{1}{\zeta} \sqrt{\frac{\pi_4 \big( \frac{1}{\sqrt{\zeta}} \big)}{\pi_1 \big( \frac{1}{\sqrt{\zeta}} \big)}}
\end{equation}
(and the next exceptional scales can be derived in a similar way).

\begin{figure}[t]
\begin{center}

\vspace{-1.2cm}
\includegraphics[width=.85\textwidth]{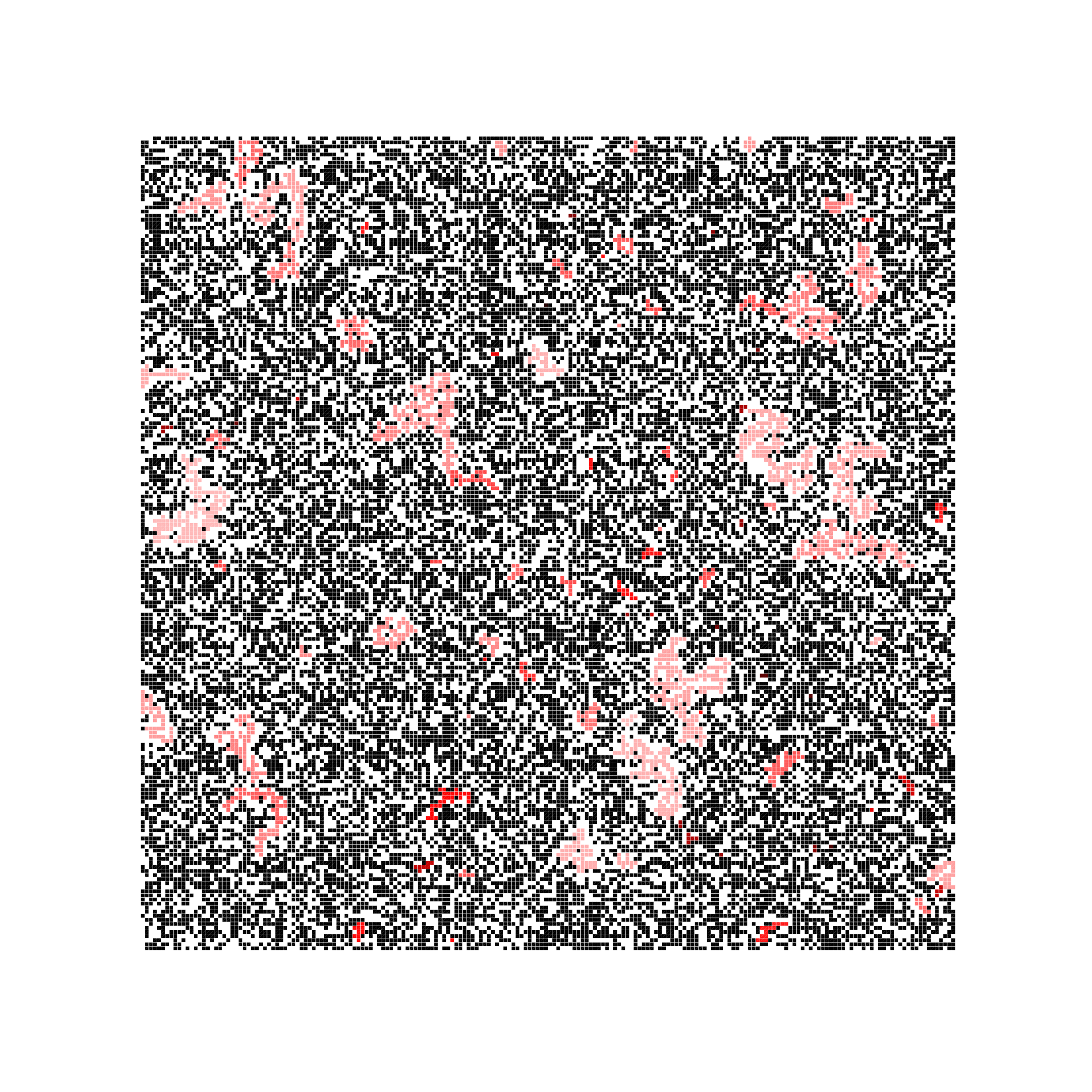}
\vspace{-1.8cm}
\caption{\label{fig:final_config2} In order to analyze the behavior around time $t_c$ of the FFWoR process, we first consider the process where ignitions stop at a slightly earlier time, $t_c - \ve$. The clusters burnt before that time are depicted in red, and may be viewed as ``impurities''. Here, $\ve = 0.1$.}

\end{center}
\end{figure}

Note that the heuristics (and the formula: \eqref{eq:intro_heuristics9} involves not only $\pi_1$, but also $\pi_4$) for the FFWoR process is more ``tricky'' than that for the parameter-$N$ model. Moreover, there is a much more serious complication. Although the reasoning leading to \eqref{eq:intro_heuristics9} might look sound, there is a delicate issue which was ``swept under the rug'' and which has no analog in the parameter-$N$ model. Indeed, we ignored the smaller burnings which took place already before time $\tau$ and created ``impurities'' in the lattice (see Figure \ref{fig:final_config2}, produced by using the same realizations of the birth and the ignition processes as for Figure \ref{fig:final_config}). For instance, the estimate \eqref{eq:intro_heuristics4} comes from ordinary percolation, but how do we know that in a model with impurities, this formula is still (more or less) correct? In fact, as we will see, the impurities are far from microscopic: we can consider them as ``heavy-tailed'', and we have to understand their cumulative effect. This effect turns out to be much more complicated (and interesting) than we anticipated in the short, speculative, last section (Section 8) of \cite{BKN2015}.

\subsection{Percolation with impurities and statement of results}

%The delicate issue mentioned above leads us to a more general study of near-critical percolation with impurities. The analysis of this more general form requires only little extra work compared to the special case corresponding to the Poisson($\zeta$) model (FFWoR), and it gives substantially more insight.

We will present and study a quite general form of near-critical percolation with impurities, which happens to be crucial for a thorough analysis of the behavior of forest fire processes near (and beyond) the critical time, and in particular for handling the delicate issue mentioned above. Roughly speaking, we have to show that, for a certain class of ``impurities'', the ``global'' connectivity properties of the percolation model are not (too) much worse than in the model without impurities (i.e. ordinary percolation). Figure \ref{fig:random_environment} gives an illustration of the type of environments that we have to analyze.

The models of impurities that we are led to study are parametrized by a positive integer denoted by $m$, and they can be described as follows. Each vertex $v \in V$, independently of the other vertices, is the center of an impurity (a square box with a random side length) with a probability that we denote by $\pi^{(m)}$. If there is an impurity centered at $v$, the probability that it has a side length $\geq r$ is written as $\rho^{(m)}([r, +\infty))$ (for all $r \geq 0$). After removing all the impurities from the lattice, we perform Bernoulli percolation with parameter $p$ on the remaining graph.

\begin{figure}[t]
\begin{center}

\vspace{-1.2cm}
\includegraphics[width=.85\textwidth]{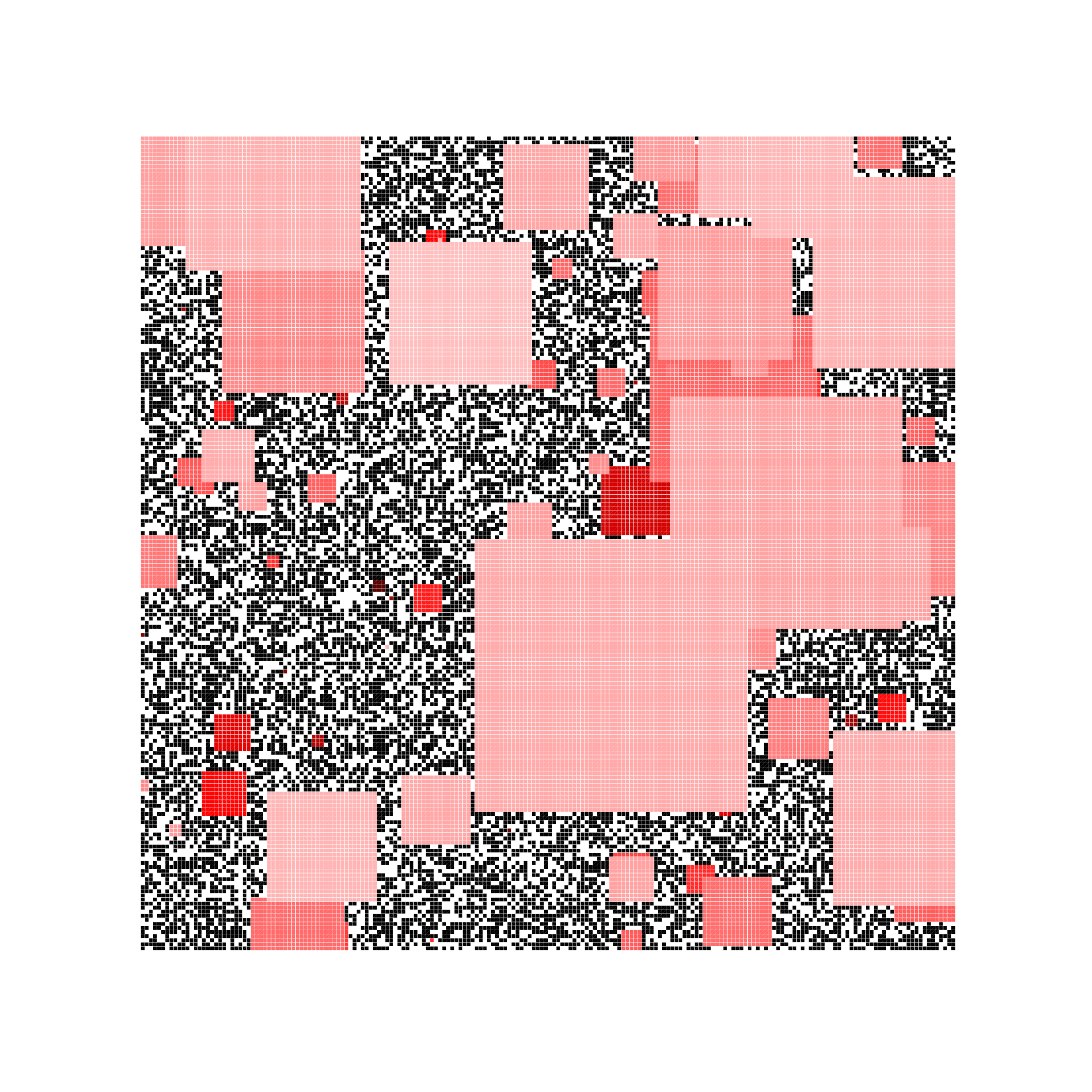}
\vspace{-1.8cm}
\caption{\label{fig:random_environment} A typical random environment produced by the fires up to time $t_c - \ve$ (here, $\ve = 0.1$), where each burnt cluster is replaced by an $L^{\infty}$ ball centered on the ignited vertex. Connections with the model with impurities are explained in Section \ref{sec:comparison_holes}.}

\end{center}
\end{figure}

Let us describe our choice of $\pi^{(m)}$ and $\rho^{(m)}$ more precisely. Let $c_1, c_2, c_3 > 0$ be constants, as well as $\alpha$ and $\beta$. Suppose that $\rho^{(m)}$ and $\pi^{(m)}$ are of the form
\begin{equation} \label{eq:form_rho_pi}
\rho^{(m)} \big( [r,+\infty) \big) = c_1 r^{\alpha-2} e^{-c_2 r / m} \:\: (r \geq 1) \quad \text{and} \quad \pi^{(m)} = c_3 m^{- \beta}.
\end{equation}
We then establish two kinds of results (which, so far, we can only prove for the triangular lattice, see Remark \ref{rem:triang_lattice} below).
\begin{itemize}
\item[1.] \emph{Stability properties for percolation with impurities.} Results of the first kind say that the resulting percolation model mentioned above (i.e. Bernoulli percolation with parameter $p$, on the graph obtained by removing impurities) satisfies, under certain conditions, connectivity properties comparable to these of the pure percolation model. In particular, under some hypotheses on the values of $\alpha$ and $\beta$, and on the relation between $m$ and $L(p)$, the four-arm probabilities remain comparable (\textbf{Theorem \ref{thm:four_arm_stability}}). We then use this to show that also one-arm probabilities, certain box-crossing probabilities, and, finally, the size of the largest cluster in a big box, remain comparable (see \textbf{Propositions \ref{prop:one_arm_stability}}, \textbf{\ref{prop:crossing}}, and \textbf{\ref{prop:largest_cluster}}).

\item[2.] \emph{Exceptional scales for forest fires with Poisson ignitions.} These results are then used to derive the second kind of results, involving applications to the model of forest fires without recovery. In particular, these results give a rigorous verification of the existence of exceptional scales mentioned (and heuristically derived) in Section \ref{sec:intro_heuristics}. This is done in Section \ref{sec:application_FF} (where the relation with the general model with impurities is proved), and Section \ref{sec:existence_excep_scales} (see \textbf{Theorems \ref{thm:case1}} and \textbf{\ref{thm:case2}}).
\end{itemize}

The four-arm stability result (Theorem \ref{thm:four_arm_stability}) turns out to be rather subtle. Indeed, we have to understand the effect of (possibly ``mesoscopic'') impurities on ``pivotal'' events, which relies on a delicate balance between ``helping'' vacant arms with the impurities, but ``hindering'' occupied arms. Our proof uses the inequality $\alpha_2 \geq \alpha_4 + 1$ between the two- and four-arm exponents for critical percolation, which can be checked (it is even an equality) from the actual values of these exponents. See Remark \ref{rem:Ke_GPS} and Section \ref{sec:stability_other_arm} for more background and details.

\begin{remark} ~
\begin{itemize}
\item As said earlier, we focus in this paper on the FFWoR model. So, when a vertex $v$ is ignited, its occupied cluster burns, but not the vacant sites along its boundary: these vertices will thus become occupied (and then burn) at later times. However, let us mention that our proofs of Theorems \ref{thm:case1} and \ref{thm:case2} also apply in the case when the occupied cluster is burnt together with its outer boundary, i.e. the seeds on the boundary ``die'' and never become a tree. In addition, we believe that, with extra work, our results can be extended to forest fires \emph{with} recovery, see the discussion in Section \ref{sec:forest_fires}.

\item The results in \cite{BN2015} were an important ingredient in our earlier-mentioned joint paper \cite{BKN2015} with Kiss, where it was proved that the parameter-$N$ model in the full plane exhibits a deconcentration property for the size of the final cluster of the origin. We believe that the results in our current paper should be instrumental to obtain similar deconcentration results for the full-plane FFWoR process.
\end{itemize}
\end{remark}

\subsection{Informal discussion about the process with impurities} \label{sec:holes_phase_diagram}

We comment a bit further on the percolation process with impurities, still assuming that $\rho^{(m)}$ and $\pi^{(m)}$ are of the form \eqref{eq:form_rho_pi}. Different behaviors arise according to the values of $\alpha$ and $\beta$, and we obtain the ``phase diagram'' depicted in Figure \ref{fig:phase_diagram}. After the present section, we will focus on Domain~I, i.e. $\alpha \in \big( \frac{3}{4},2 \big)$ and $\beta > \alpha$, which contains the relevant values for forest fires (this is Assumption~\ref{ass:alpha_beta} in Section \ref{sec:holes_def}): typically, $\alpha = \frac{55}{48} + \upsilon$ and $\beta = \alpha + \upsilon'$, for some arbitrarily small $\upsilon, \upsilon' > 0$. We want to emphasize that the somewhat informal discussion in this section is mostly about other domains, rather than the one we concentrate on, and it is not required for the understanding of the rest of the paper.

\begin{figure}
\begin{center}

\includegraphics[width=.6\textwidth]{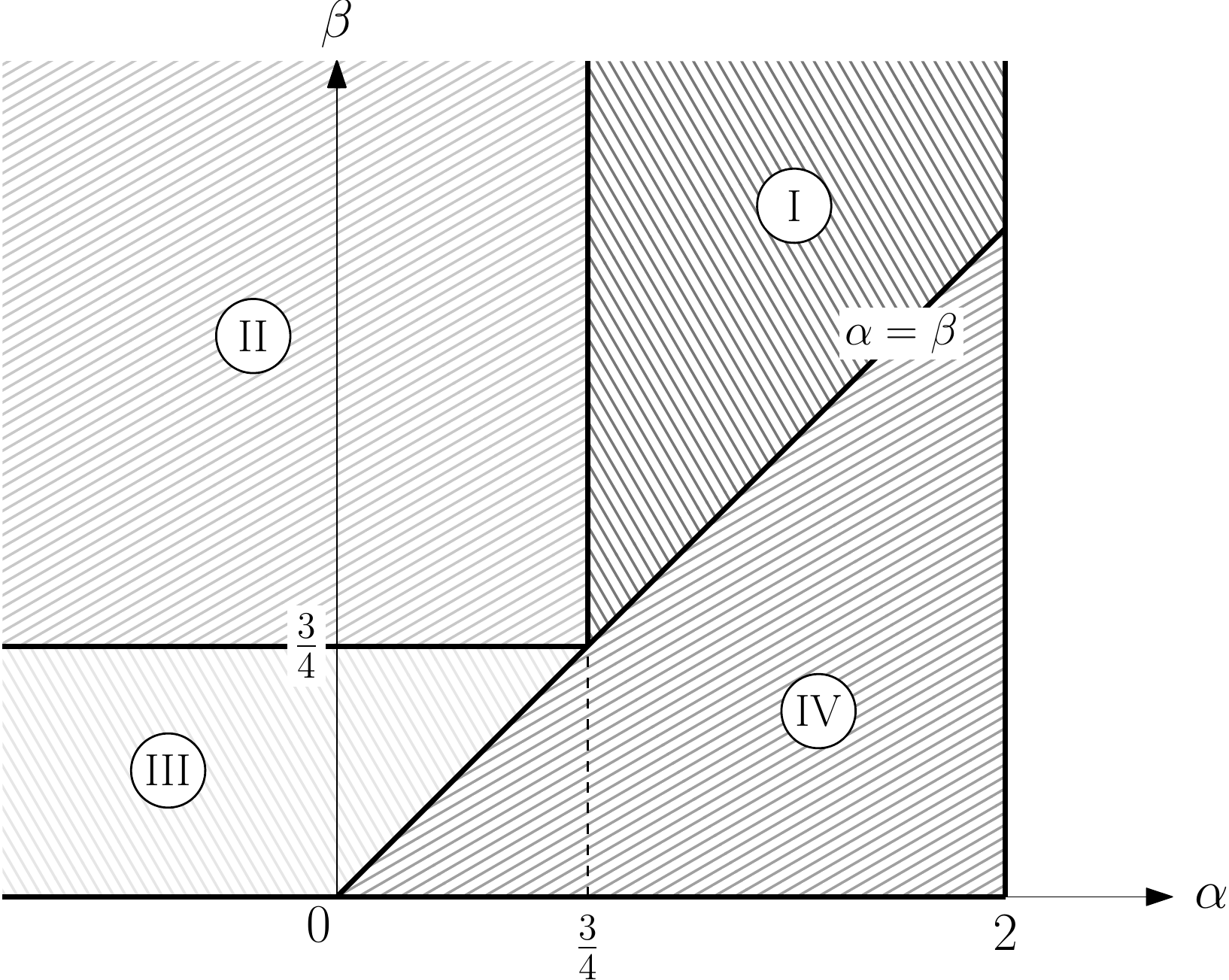}
\caption{\label{fig:phase_diagram} Near-critical percolation with impurities displays different behaviors according to the values of the exponents $\alpha$ and $\beta$. Here, $\frac{3}{4} = \frac{1}{\nu}$, where $\nu$ is the critical exponent for the characteristic length $L$.}

\end{center}
\end{figure}

First, note that as a special case, our framework contains classical near-critical percolation, studied in \cite{Ke1987, CLN2006, NW2009, GPS2013a}. Indeed, near-critical percolation with parameter $p < p_c$ can be constructed from the critical regime by performing single-site updates, i.e. letting the sites independently switch from occupied to vacant. It is obtained by taking $\rho^{(m)} = \delta_0$ (i.e. the Dirac mass at $0$, so that only impurities of radius $0$ are created) and $\pi^{(m)} \asymp \frac{1}{m^2 \pi_4(m)} = m^{- \frac{1}{\nu} + o(1)}$, where $\nu = \frac{4}{3}$ is the critical exponent associated with $L$. This means that if we start from the critical regime and update the sites with a probability $\pi^{(m)} \asymp m^{-\beta}$, the resulting configuration is subcritical for $\beta < \frac{1}{\nu} = \frac{3}{4}$, and it stays near-critical for $\beta > \frac{1}{\nu}$.

We now briefly discuss the various domains in Figure \ref{fig:phase_diagram}. Observe that a short computation (similar to the one in Lemma \ref{lem:no_crossing} below) shows that the ``density'' of impurities, i.e. the probability for each vertex to be contained in at least one impurity, is of order $m^{\alpha_+ - \beta}$.

\begin{itemize}
\item \underline{Domain I:} As mentioned in the previous section, we show that under appropriate hypotheses, percolation in the complement of the impurities stays comparable, in terms of connectedness, to ordinary percolation near criticality. We want to highlight that this behavior holds even when $\beta$ is very close to $\alpha$, which means that the exponent $\alpha-\beta$ for the density of impurities can be made arbitrarily close to $0$. This stands in contrast with the usual case of single-site updates, where this density has to stay below $m^{- \frac{3}{4} + o(1)}$. Roughly speaking, the behavior exhibited in Domain I comes from the particular way in which updates of sites are ``arranged'' spatially. Since they are grouped into balls, each of them has individually less effect. It emerges from explicit computations that the contribution of pivotal impurities is mainly produced by large impurities, with a radius of order $m$.

\item \underline{Domain II:} Similar properties as in Domain I hold in this case, which is essentially covered by our proofs (see Section \ref{sec:rem_domain2} below). However, the phenomenology in this domain is rather different since the main contribution of pivotal impurities is produced by microscopic impurities. Hence, the fact that the percolation configuration stays near-critical comes essentially from the same reasons as for single-site updates (this classical case corresponds to $\alpha = - \infty$ formally). Furthermore, for a given $\alpha < \frac{3}{4}$, the exponent $\alpha_+ - \beta$ in the density of impurities stays smaller than $\alpha_+ - \frac{3}{4}$, and so cannot be made arbitrarily close to $0$, contrary to Domain~I.

\item \underline{Domain III:} In this case, the configuration with impurities is clearly dominated by a configuration of Bernoulli percolation, obtained by using the same $\pi^{(m)}$, but with single-site updates (i.e. $\rho^{(m)} = \delta_0$). We know that in this case, the resulting configuration is subcritical for $\beta < \frac{3}{4}$.

\item \underline{Domain IV:} When $\alpha > \beta$, the process is completely ``degenerate''. For example, it is easy to see that for all $K > 0$, with high probability (as $m \to \infty$), there exists an impurity centered on some $v \in \Ball_{Km}$ that covers entirely $\Ball_{Km}$.
\end{itemize}

We conclude this discussion by mentioning some works with a somewhat similar flavor, although the techniques and questions that we are studying in this paper are quite different in nature. Percolation on fractal-like graphs has been studied in e.g. \cite{Sh1996, Ku1997, Sh2002, Sh2003, HW2008, KT2014} (see also the discussion in Section 2.1 of \cite{MSW2017}). There is also an extensive literature about a random walk / Brownian motion among randomly distributed obstacles: see for example the classical reference \cite{Sz1998}, the recent review \cite{ADS2017}, and the references therein.

\subsection{Organization of the paper}

Section \ref{sec:percolation} contains preliminaries about usual Bernoulli percolation. We first set notations, and then we collect classical results on the behavior of two-dimensional percolation through its phase transition, i.e. at and near its critical point.

In Sections \ref{sec:process_holes} to \ref{sec:other_stability}, we analyze the percolation process with heavy-tailed impurities. We introduce it and present some of its properties in Section \ref{sec:process_holes}. Section \ref{sec:four_arm_stability} is devoted to stating and proving a stability result for four-arm events in the near-critical regime. This property is instrumental to derive further stability results, which we do in Section \ref{sec:other_stability}, culminating with a volume estimate for the largest connected component in a box.

We then make the connection with forest fire processes in Section \ref{sec:application_FF}: we introduce the exceptional scales $(m_k(\zeta))_{k \geq 1}$, and we collect some of their properties. We also explain how forest fires can be coupled to the process with impurities. In Section \ref{sec:existence_excep_scales}, we present some applications of the results developed earlier: we show that the scales $m_k$ are indeed exceptional for forest fire processes without recovery. Finally, in Section \ref{sec:forest_fires}, we briefly discuss forest fire processes with recovery.

\section{Phase transition of two-dimensional percolation} \label{sec:percolation}

Our results rely heavily on a precise understanding of 2D percolation at and near criticality. We start by setting notations in Section \ref{sec:setting}. We then list in Section \ref{sec:near_critical_perc} all the classical properties which are needed later, before deriving some additional results in Section \ref{sec:add_results}.

\subsection{Setting and notations} \label{sec:setting}

In the present paper, we work with the triangular lattice $\TT = (V, E)$, with vertex set
$$V := \big\{ x + y e^{i \pi / 3} \in \CC \: : \: x, y \in \ZZ \big\}$$
and edge set $E := \{ \{v, v'\} \: : \: v, v' \in V \text{ with } |v - v'| = 1 \}$ (using the standard identification $\RR^2 \simeq \CC$). Two vertices $v, v' \in V$ are said to be neighbors if they are connected by an edge, and we denote it by $v \sim v'$. From now on, we always use the $L^{\infty}$ norm $\|.\| = \|.\|_{\infty}$. The inner (resp. outer) boundary of a subset $A \subseteq V$ is defined as $\din A := \{v \in A \: : \: v \sim v'$ for some $v' \in A^c\}$ (resp. $\dout A := \din (A^c)$), and its volume, denoted by $|A|$, is simply the number of vertices that it contains.

Recall that Bernoulli site percolation on $\TT$ with parameter $p \in [0,1]$ is obtained by declaring each vertex $v \in V$ either occupied or vacant, with respective probabilities $p$ and $1-p$, independently of the other vertices. We denote by $\PP_p$ the corresponding product probability measure on configurations of sites $(\omega_v)_{v \in V} \in \{0,1\}^V =: \Omega$.

A path of length $k$ ($k \geq 1$) is a sequence of vertices $v_0 \sim v_1 \sim \ldots \sim v_k$. Two vertices $v, v' \in V$ are connected (denoted by $v \lra v'$) if there exists a path of length $k$ from $v$ to $v'$, for some $k \geq 1$, containing only occupied sites (in particular, $v$ and $v'$ have to be occupied). More generally, two subsets $A, A' \subseteq V$ are connected if there exist $v \in A$ and $v' \in A'$ such that $v \lra v'$, which we denote by $A \lra A'$. Occupied vertices can be grouped into maximal connected components, or clusters. For a vertex $v \in V$, we denote by $\cluster(v)$ the occupied cluster of $v$, setting $\cluster(v) = \emptyset$ when $v$ is vacant. We write $v \lra \infty$ for the event that $|\cluster(v)| = \infty$, i.e. $v$ lies in an infinite occupied cluster, and we introduce $\theta(p) := \PP_p(0 \lra \infty)$. Site percolation on $\TT$ displays a phase transition at the percolation threshold $p_c = p_c^{\textrm{site}}(\TT)$, and it is now a classical result \cite{Ke1980} that $p_c = \frac{1}{2}$. Moreover, it is also known that $\theta(p_c) = 0$. Hence, for each $p \leq p_c = \frac{1}{2}$, there is almost surely no infinite cluster, while for $p > \frac{1}{2}$, there is almost surely a unique such cluster. The reader can consult the classical references \cite{Ke1982, Gr1999} for more background on percolation theory.

For a rectangle of the form $R = [x_1,x_2] \times [y_1,y_2]$ ($x_1 < x_2$, $y_1 < y_2$), an occupied path in $R$ ``connecting the left and right (resp. top and bottom) sides'' is called a horizontal (resp. vertical) crossing. Here, we use quotation marks because $R$ does not exactly ``fit'' the triangular lattice $\TT$, so the definition needs to be made more accurate: this can be done easily, see for instance Definition 1 in Section 3.3 of \cite{Ke1982} (the same remark applies to arm events, defined below). The event that such a crossing exists is denoted by $\Ch(R)$ (resp. $\Cv(R)$). We also write $\Ch^*(R)$ and $\Cv^*(R)$ for the corresponding events with paths of vacant vertices.

Let $\Ball_n := [-n,n]^2$ be the ball of radius $n \geq 0$ around $0$ for $\|.\|_{\infty}$. For $0 \leq n_1 < n_2$, we denote by $\Ann_{n_1,n_2} := \Ball_{n_2} \setminus \Ball_{n_1}$ the annulus with radii $n_1$ and $n_2$ centered at $0$. For $z \in \CC$, we write $\Ball_n(z) := z + \Ball_n$, and $\Ann_{n_1,n_2}(z) := z + \Ann_{n_1,n_2}$. Finally, we denote $\Ann_{n_1,\infty}(z) := (\Ball_{n_1}(z))^c$. For an annulus $A = \Ann_{n_1,n_2}(z)$ ($0 \leq n_1 < n_2 < \infty$, $z \in \CC$), we denote by $\circuitevent(A)$ (resp. $\circuitevent^*(A)$) the existence of an occupied (resp. vacant) circuit in $A$. For $k \geq 1$ and $\sigma \in \colorseq_k := \{o,v\}^k$ (where $o$ and $v$ stand for ``occupied'' and ``vacant'', resp.), we introduce the arm event $\arm_{\sigma}(A)$ that there exist $k$ disjoint paths $(\gamma_i)_{1 \leq i \leq k}$ in $A$, in counter-clockwise order, each with type prescribed by $\sigma_i$ (i.e. occupied or vacant path) and connecting $\dout \Ball_{n_1}(z)$ to $\din \Ball_{n_2}(z)$. We use the notation
\begin{equation} \label{eq:def_pi}
\pi_{\sigma}(n_1,n_2) := \PP_{p_c}\big( \arm_{\sigma}(\Ann_{n_1,n_2}) \big),
\end{equation}
and we write $\pi_{\sigma}(n) := \pi_{\sigma}(1,n)$. For $k \geq 1$, we use the shorthand notations $\arm_k$ and $\pi_k$ in the particular case when $\sigma =(ovo\ldots) \in \colorseq_k$ is alternating.

\begin{remark} \label{rem:triang_lattice}
Note that even if we expect our methods to work for any lattice with enough symmetries, such as $\ZZ^2$, as well as for analogous processes defined in terms of bond percolation, we have to focus on site percolation on $\TT$. Indeed, it is the case for which the most precise results are known (especially \eqref{eq:arm_exponent} below), thanks to the SLE (Schramm-Loewner Evolution) technology. Our proofs require a good control on arm events, as explained in the beginning of Section \ref{sec:stability_other_arm}, and at the moment, the bounds available for other lattices are not sufficiently accurate. It was also the case in \cite{BKN2015} that the main results could be established for the triangular lattice only. However, the construction of the scaling limit of near-critical percolation \cite{GPS2013a} was a crucial ingredient in \cite{BKN2015}, while it is not needed here.
\end{remark}

\subsection{2D percolation at and near criticality} \label{sec:near_critical_perc}

The usual characteristic length $L$ is defined by:
\begin{equation} \label{eq:def_L}
\text{for $p < p_c = \frac{1}{2}$,} \quad L(p) := \min \big\{ n \geq 1 \: : \: \PP_p \big( \Cv( [0,2n] \times [0,n] ) \big) \leq 0.001 \big\},
\end{equation}
and $L(p) = L(1-p)$ for $p > p_c$. It follows from the Russo-Seymour-Welsh (RSW) bounds that at $p = p_c$, the probability in the right-hand side of \eqref{eq:def_L} is $> 0.001$ for all $n \geq 1$, so $L(p) \to \infty$ as $p \to p_c$, and we define $L(p_c) := \infty$. In the present paper, we consider a regularized version $\tilde L$, defined as follows. First, we set $\tilde L(p) = L(p)$ at each point of discontinuity $p \in(0,p_c) \cup (p_c,1)$ of $L$, $L(0) = L(1) = 0$, and then we extend linearly $\tilde L$ to $[0,1] \setminus \{p_c\}$. The function $\tilde L$ has the additional property of being continuous and strictly increasing (resp. strictly decreasing) on $[0,p_c)$ (resp. $(p_c,1]$). In particular, it is a bijection from $[0,p_c)$ (resp. $(p_c,1]$) to $[0,\infty)$. In the following, we simply write $L$ instead of $\tilde L$.

Throughout the paper, we make use of the following classical properties of Bernoulli percolation, at and near the critical point $p_c$.

\begin{enumerate}[(i)]
\item \emph{RSW-type bounds.} For all $K \geq 1$, there exists a constant $\delta_4 = \delta_4(K) > 0$ such that: for all $p \in (0,1)$ and $n \leq K L(p)$,
\begin{equation} \label{eq:RSW}
\PP_p \big( \Ch( [0,4n] \times [0,n] ) \big) \geq \delta_4 \quad \text{and} \quad \PP_p \big( \Ch^*( [0,4n] \times [0,n] ) \big) \geq \delta_4.
\end{equation}
Note that since $L(p_c) = \infty$, the first inequality actually holds for all $p \geq p_c$ and $n \geq 1$.

\item \emph{Exponential decay property.} There exist universal constants $C_1, C_2 > 0$ such that: for all $p > p_c$ and $n \geq 1$,
\begin{equation} \label{eq:exp_decay}
\PP_p \big( \Ch( [0, 4n] \times [0,n] ) \big) \geq 1 - C_1 e^{- C_2 \frac{n}{L(p)}}
\end{equation}
(see Lemma 39 in \cite{No2008}).

\item \emph{Extendability of arm events.} For all $k \geq 1$ and $\sigma \in \colorseq_k$, there exists a constant $C > 0$ (that depends on $\sigma$ only) such that: for all $0 \leq n_1 < n_2$,
\begin{equation} \label{eq:extendability}
\pi_{\sigma} \Big( \frac{n_1}{2}, n_2 \Big), \: \pi_{\sigma}(n_1, 2 n_2) \geq C \pi_\sigma(n_1, n_2)
\end{equation}
(see Proposition 16 in \cite{No2008}).

\item \emph{Quasi-multiplicativity of arm events.} For all $k \geq 1$ and $\sigma \in \colorseq_k$, there exist $C_1, C_2 > 0$ (depending only on $\sigma$) such that: for all $0 \leq n_1 < n_2 < n_3$,
\begin{equation} \label{eq:quasi_mult}
C_1 \pi_{\sigma}(n_1, n_3) \leq \pi_{\sigma}(n_1, n_2) \pi_{\sigma}(n_2, n_3) \leq C_2 \pi_{\sigma}(n_1, n_3)
\end{equation}
(see Proposition 17 in \cite{No2008}).

\item \emph{Arm exponents at criticality.} For all $k \geq 1$, and $\sigma \in \colorseq_k$, there exists $\alpha_{\sigma} > 0$ such that
\begin{equation} \label{eq:arm_exponent}
\pi_{\sigma}(k,n) = n^{- \alpha_{\sigma} + o(1)} \quad \text{as $n \to \infty$.}
\end{equation}
Moreover, the value of $\alpha_{\sigma}$ is known, except in the monochromatic case (for $k \geq 2$ arms of the same type).
\begin{itemize}
\item For $k=1$, $\alpha_{\sigma} = \frac{5}{48}$.

\item For all $k \geq 2$, and $\sigma \in \colorseq_k$ containing both types, $\alpha_{\sigma} = \frac{k^2-1}{12}$.
\end{itemize}
These arm exponents were derived in \cite{LSW2002, SW2001}, based on the conformal invariance property of critical percolation \cite{Sm2001} and properties of the Schramm-Loewner Evolution (SLE) processes (with parameter $6$, here) \cite{LSW_2001a, LSW_2001b}.

\item \emph{Upper bound on monochromatic arm events.} For all $k \geq 2$, let $\sigma = (o \ldots o) \in \colorseq_k$ be the monochromatic sequence of length $k$. There exist $C_k, \beta_k > 0$ such that: for all $0 \leq n_1 < n_2$,
\begin{equation} \label{eq:monochromatic}
\pi_{\sigma}(n_1, n_2) \leq C_k \bigg( \frac{n_1}{n_2} \bigg)^{\beta_k} \pi_k(n_1, n_2).
\end{equation}
This follows from the proof of Theorem 5 in \cite{BN2011} (see in particular Step 1).

\item \emph{Stability for arm events near criticality.} For all $k \geq 1$, $\sigma \in \colorseq_k$, and $K \geq 1$, there exist constants $C_1, C_2 > 0$ (depending on $\sigma$ and $K$) such that: for all $p \in (0,1)$, and all $0 \leq n_1 < n_2 \leq K L(p)$,
\begin{equation} \label{eq:near_critical_arm}
C_1 \pi_{\sigma}(n_1, n_2) \leq \PP_p \big( \arm_{\sigma}(\Ann_{n_1, n_2}) \big) \leq C_2 \pi_{\sigma}(n_1, n_2)
\end{equation}
 (see Theorem 27 in \cite{No2008}).

\item \emph{Asymptotic equivalences for $\theta$ and $L$.} We have
\begin{equation} \label{eq:equiv_theta}
\theta(p) \asymp \pi_1(L(p)) \quad \text{as $p \searrow p_c$}
\end{equation}
(see Theorem 2 in \cite{Ke1987}, or (7.25) in \cite{No2008}), and
\begin{equation} \label{eq:equiv_L}
\big| p - p_c \big| L(p)^2 \pi_4 \big( L(p) \big) \asymp 1 \quad \text{as $p \to p_c$}
\end{equation}
(see (4.5) in \cite{Ke1987}, or Proposition 34 in \cite{No2008}).

\item \emph{A-priori bounds on arm events.} There exist universal constants $C_i > 0$ ($1 \leq i \leq 4$) and $\beta_j > 0$ ($1 \leq j \leq 3$) such that the following inequalities hold. For all $p \in (0,1)$ and $0 \leq n_1 < n_2 \leq L(p)$,
\begin{equation} \label{eq:1arm}
C_1 \bigg( \frac{n_1}{n_2} \bigg)^{1/2} \leq \PP_p \big( \arm_1(\Ann_{n_1, n_2}) \big) \leq C_2 \bigg( \frac{n_1}{n_2} \bigg)^{\beta_1}
\end{equation}
(the upper bound is an immediate consequence of \eqref{eq:RSW}, while the lower bound follows from the van den Berg-Kesten inequality), and
\begin{equation} \label{eq:4arms}
C_3 \bigg( \frac{n_1}{n_2} \bigg)^{2 - \beta_2} \leq \PP_p \big( \arm_4(\Ann_{n_1, n_2}) \big) \leq C_4 \bigg( \frac{n_1}{n_2} \bigg)^{1 + \beta_3}.
\end{equation}
The left-hand inequality in \eqref{eq:4arms} follows from the ``universal'' arm exponent for $\arm_5$ which is equal to $2$ (see Theorem 24 (3) in \cite{No2008}) together with \eqref{eq:1arm}, and the right-hand inequality follows from Corollary 2 in \cite{Ke1987} (more precisely, the lower bound on the critical exponent $\nu$ associated with $L$) combined with \eqref{eq:equiv_L}.

\item \emph{Volume estimates.} Let $(n_k)_{k \geq 1}$ be a sequence of integers, with $n_k \to \infty$ as $k \to \infty$, and $(p_k)_{k \geq 1}$ satisfying $p_c < p_k < 1$. If $L(p_k) \ll n_k$ as $k \to \infty$, then
\begin{equation} \label{eq:largest_cluster}
\text{for all $\ve > 0$,} \quad \PP_{p_k} \left( \frac{|\lclus_{\Ball_{n_k}}|}{\theta(p_k) |\Ball_{n_k}|} \notin (1 - \ve, 1 + \ve) \right) \stackrel[k \to \infty]{}{\longrightarrow} 0,
\end{equation}
where we denote by $|\lclus_{\Ball_{n_k}}|$ the volume of the largest occupied cluster in $\Ball_{n_k}$ (see Theorem 3.2 in \cite{BCKS2001}).
%We will also use this result with the boxes $B(n_k)$ replaced by the annuli $A(\eta n_k, n_k)$, for some fixed $\eta \in (0,1)$: it is straightforward to adapt the proofs in \cite{BCKS} to this situation. Note that the condition $L(p_k) \ll n_k$ is satisfied in particular when $p_k \equiv p \in (p_c,1)$.
\end{enumerate}

\subsection{Additional results} \label{sec:add_results}

The following geometric construction is used repeatedly in our proofs.

\begin{definition} \label{def:net}
For $\kappa, n \geq 1$, let $\net_p(n,\kappa)$ be the event that there exists a $p$-occupied crossing in the long direction in each of the (horizontal and vertical) rectangles of the form
$$\frac{\kappa}{2} \big( [-4, 4] \times [-1, 1] + (6i, 6j-3) \big) \quad \text{and} \quad \frac{\kappa}{2} \big( [-1, 1] \times [-4, 4] + (6i-3, 6j) \big)$$
($i$, $j$ integers) that intersect the box $\Ball_n$ (see Figure \ref{fig:net}).
\end{definition}

\begin{figure}
\begin{center}

\includegraphics[width=.5\textwidth]{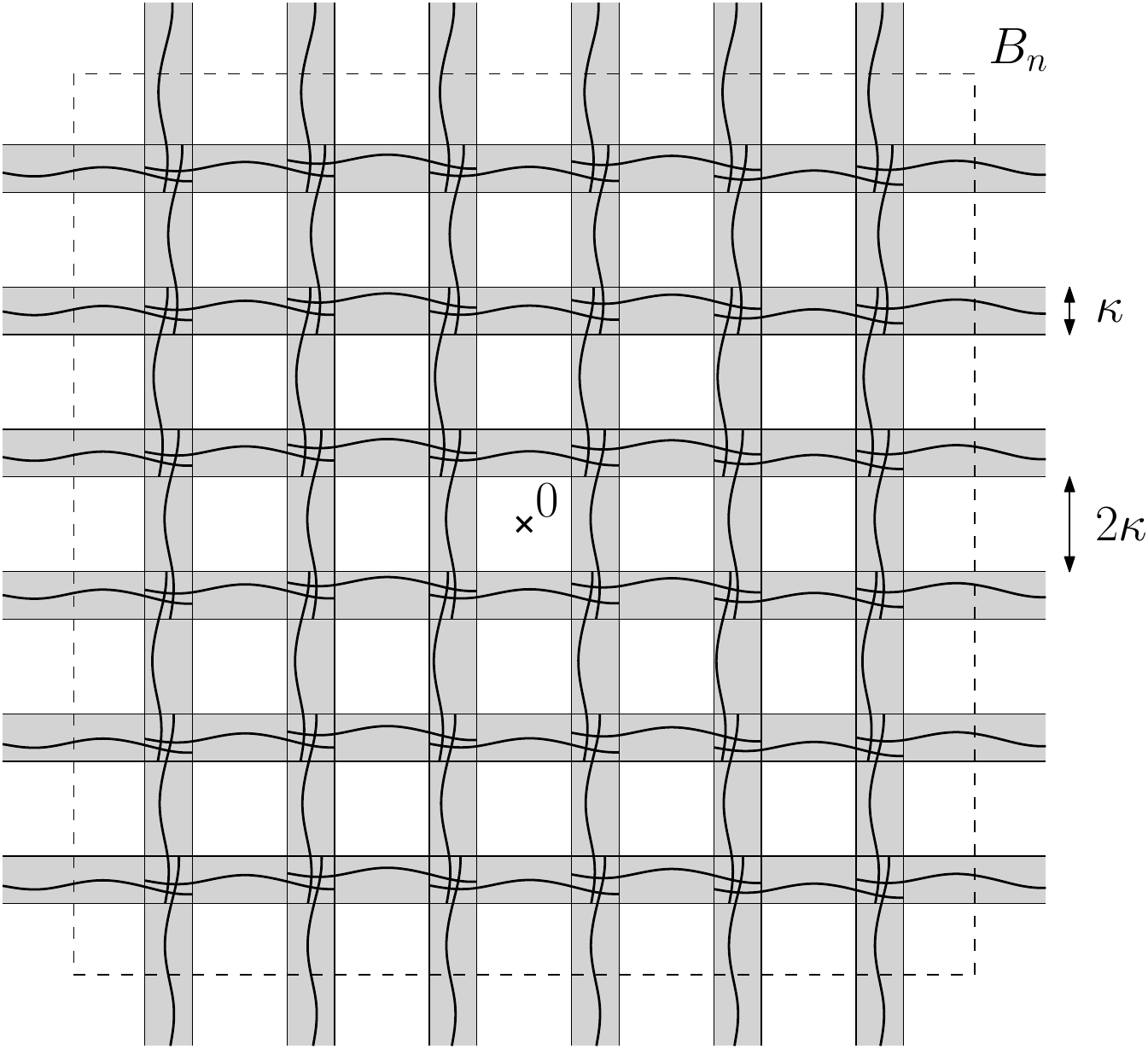}
\caption{\label{fig:net} This figure depicts the event $\net_p(n,\kappa)$ that there exists a net ``spanning'' the box $\Ball_n$. The definition of this event only involves the sites in the gray areas: in particular, it does not depend on the sites in $\Ball_{\kappa}$.}

\end{center}
\end{figure}

For the percolation configuration inside $\Ball_n$, the event $\net_p(n, \kappa)$ implies the existence of a $p$-occupied connected set $\net$ such that all the connected components of its complement (so in particular, all the $p$-occupied and $p$-vacant connected components other than the cluster $\cluster_{\net}$ of $\net$ in $\Ball_n$) have a diameter at most $4 \kappa$. Such a set $\net$ is called \emph{net with mesh $\kappa$}. Note also that $\net_p(n,\kappa)$ does not depend on the sites in $\Ball_\kappa$.

\begin{lemma} \label{lem:net}
There exist universal constants $C_1, C_2 > 0$ such that: for all $n \geq \kappa \geq 1$ and $p > p_c$,
\begin{equation} \label{eq:net}
\PP \big( \net_p(n,\kappa) \big) \geq 1 - C_1 \Big( \frac{n}{\kappa} \Big)^2 e^{-C_2 \frac{\kappa}{L(p)}}.
\end{equation}
\end{lemma}

\begin{proof}[Proof of Lemma \ref{lem:net}]
Observe that the definition of $\net_p(n,\kappa)$ involves of order $\big( \frac{n}{\kappa} \big)^2$ rectangles, each with side lengths $\kappa$ and $4\kappa$. Hence, \eqref{eq:net} is an immediate consequence of the exponential decay property \eqref{eq:exp_decay}.
\end{proof}

Recall the following exponential upper bound for the probability of observing abnormally large clusters (see Lemma 4.4 in \cite{BKN2015}).
\begin{lemma} \label{lem:BCKS}
There exist universal constants $C_1, C_2, X > 0$ such that: for all $p > p_c$, $n \geq L(p)$, and $x \geq X$,
\begin{equation} \label{eq:BCKS}
\PP_p \left( \big| \lclus_{\Ball_n} \big| \geq x n^2 \theta(p) \right) \leq C_1 e^{- C_2 x \frac{n^2}{L(p)^2}}.
\end{equation}
\end{lemma}

\begin{remark} \label{rem:BCKS}
Note that with high probability as $k \to \infty$, $\lclus_{\Ball_{n_k}}$ in \eqref{eq:largest_cluster} contains a net $\net$ with mesh $(n_k L(p_k))^{1/2}$. Indeed, it follows from Lemma \ref{lem:net} that $\net$ exists with high probability, and such an $\net$ then subdivides $\Ball_{n_k}$ into of order $\frac{n_k}{L(p_k)}$ ``cells'', each with a diameter at most $4 (n_k L(p_k))^{1/2}$. Hence, Lemma \ref{lem:BCKS} implies that the probability that one cluster, other than $\cluster_{\net}$, has a volume $\geq \frac{1}{2} \theta(p_k) |\Ball_{n_k}|$ is at most
\begin{equation}
C'_1 \frac{n_k}{L(p_k)} e^{- C_2 x \frac{(4 (n_k L(p_k))^{1/2})^2}{L(p_k)^2}}, 
\end{equation}
with $x = \frac{1}{2} |\Ball_{n_k}| / (4 (n_k L(p_k))^{1/2})^2 \asymp n_k / L(p_k) \to \infty$ as $k \to \infty$.
\end{remark}

We will also need a more uniform version of \eqref{eq:arm_exponent}.

\begin{lemma} \label{lem:ratio_limit}
For all $k \geq 1$, $\sigma \in \colorseq_k$, and $\ve > 0$, there exist $0 < C_1 < C_2$ (depending on $\sigma$ and $\ve$) such that: for all $0 \leq n_1 < n_2$,
\begin{equation}
C_1 \bigg( \frac{n_1}{n_2} \bigg)^{\alpha_{\sigma} + \ve} \leq \pi_{\sigma}(n_1, n_2) \leq C_2 \bigg( \frac{n_1}{n_2} \bigg)^{\alpha_{\sigma} - \ve}.
\end{equation}
\end{lemma}

\begin{proof}[Proof of Lemma \ref{lem:ratio_limit}]
This follows easily from \eqref{eq:quasi_mult}, and the property that: for all $k \geq 1$ and $\sigma \in \colorseq_k$,
\begin{equation}
\lim_{n \to \infty} \pi_{\sigma}(n, \lambda n) = \lambda^{- \alpha_{\sigma} + o(1)} \quad \text{as $\lambda \to \infty$}
\end{equation}
(see e.g. \cite{SW2001}).
\end{proof}

\section{Percolation process with heavy-tailed impurities} \label{sec:process_holes}

We now introduce the percolation process on a lattice with impurities (Section \ref{sec:holes_def}), and we establish an elementary upper bound on the probability of observing ``large'' impurities in an annulus (Section \ref{sec:holes_crossing}), which will be quite useful in the subsequent sections.

\subsection{Motivation and definition} \label{sec:holes_def}

Recall the forest fire processes described in the Introduction. Intuitively, as $t$ approaches $t_c$, larger and larger clusters appear, thus creating larger and larger vacant regions when they burn. We have to understand the cumulative effect of these burnings on the connectivity of the percolation configuration: in principle, the ``destroyed'' areas could, at some point, be so large that they hinder the creation of new large connected components. We study the interplay between these competing effects by introducing a percolation process on a ``randomly perforated'' lattice. Roughly speaking, as we will see later, this process provides a good picture of the forest slightly before time $t_c$, in particular whether it is sufficiently connected for large-scale fires to occur.

In the following, we consider a lattice with ``impurities'', that we call \emph{holes} from now on. Let $\pi$ be a parameter in $[0,1]$, and $\rho$ a distribution on $[0,+\infty)$ that describes the radii of the holes. We put holes on $V$ in the following fashion. For each vertex $v \in V$, independently of the other vertices, we draw a radius $r_v$ distributed according to $\rho$, and we put a hole centered on $v$ with a probability $\pi$ (typically $\ll 1$): in this case, we remove from the lattice all the vertices in the hole $H_v := \Ball_{r_v}(v)$, i.e. within a distance $r_v$ (for the norm $\|.\|_{\infty}$) from $v$ (see Figure \ref{fig:impurities}). If there is no hole centered on $v$, we set $H_v := \emptyset$.

For technical reasons, we also need to allow $\pi$ to depend on $v$, i.e. we consider inhomogeneous $\pi = (\pi_v)_{v \in V}$. Let $I_v$ be the indicator that there is a hole centered at $v$ (so that $I_v = 1$ with probability $\pi_v$). We always assume that the random variables $(I_v)_{v \in V}$ and $(r_v)_{v \in V}$ are independent.

\begin{figure}
\begin{center}

\includegraphics[width=.6\textwidth]{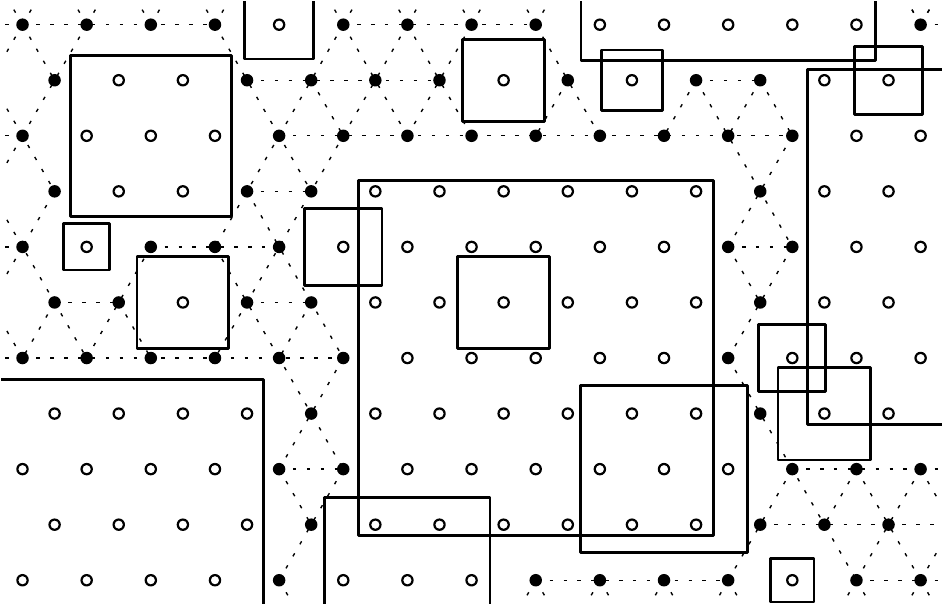}
\caption{\label{fig:impurities} The lattice $\TT$ ``perforated'' by holes of varying sizes.}

\end{center}
\end{figure}

We obtain in this way a random subgraph of $\TT$, and we are interested in its connectedness. In particular, we want to study independent site percolation with parameter $p \in [0,1]$ on this subgraph, i.e. on the complement of the holes. Note that a totally equivalent way of seeing it is by first considering an independent site percolation configuration on $\TT$, and then partially ``destroy'' it with holes (we think of the vertices in the holes as simply being vacant): are the large-scale connectivity properties of the percolation configuration significantly affected by the holes? We obtain a probability measure on configurations of holes and percolation configurations on the complement of the holes, that we denote by $\PPh_p^{\pi,\rho}$. Sometimes, we forget about the dependence on $\pi$ and $\rho$, when they are clear from the context, and we just write $\PPh_p$. For events regarding only the configuration of holes, we use the notation $\PPh^{\pi,\rho}$.

\begin{remark} \label{rem:FKG_holes}
For future use, observe that the FKG inequality holds for the $\PPh_p^{\pi,\rho}$ process with independent holes. Indeed, if we denote by $X_v$ ($v \in V$) the indicator of the event that $v$ is occupied in the underlying percolation process, and by $Y_v$ the indicator of the event that $v$ is occupied in the model with holes, then for each $v \in V$, $Y_v$ is increasing in the $X$-values, and decreasing in the $I$- and $r$-values. Moreover, the random variables $(X_v)_{v \in V}$, $(I_v)_{v \in V}$ and $(r_v)_{v \in V}$ are independent, so the collection $(Y_v)_{v \in V}$ is positively associated (this can be seen by following the proof of the Harris inequality).
\end{remark}

Of course, the macroscopic behavior of the $\PPh_p^{\pi,\rho}$ process depends on the particular choice of $\pi$ and $\rho$, and we focus on the following setting where they depend on a parameter $m \to \infty$. We assume that $\rho^{(m)}$ is ``heavy-tailed'', and that a uniform power-law upper bound on $\pi^{(m)}_v$ holds. More precisely:

\begin{assumption} \label{ass:rho_pi}
For some constants $c_1, c_2, c_3 \in (0,+\infty)$, and some exponents $\alpha < 2$ and $\beta > 0$, we have for all sufficiently large $m$:
\begin{equation} \label{eq:assump_holes}
\rho^{(m)} \big( [r,+\infty) \big) \leq c_1 r^{\alpha-2} e^{-c_2 r / m} \text{ for all } r \geq 1, \quad \text{and} \quad \pi^{(m)}_v \leq c_3 m^{- \beta} \text{ for all } v \in V.
\end{equation}
In this particular setting where $\pi$ and $\rho$ are parametrized by $m$, we write
\begin{equation} \label{eq:def_P_bar}
\PPh_p^{(m)} := \PPh_p^{\pi^{(m)},\rho^{(m)}}.
\end{equation}
\end{assumption}
We will be mostly interested in values of $p$ in near-critical windows around $p_c$ of the form $\{ p \: : \: L(p) \geq \kappa m\}$, for fixed $\kappa > 0$. As we explain in Section \ref{sec:comparison_holes}, the $\PPh_p^{(m)}$ processes arise naturally in the study of forest fires, at times close to the critical time $t_c$. In this case, the truncation parameter $m$ (the typical radius of the largest holes) plays the role of a characteristic scale for the holes created by fires up to some time slightly before $t_c$.

As we mentioned in Section \ref{sec:holes_phase_diagram}, the $\PPh_p^{(m)}$ process behaves asymptotically (as $m \to \infty$) in very different ways according to the values of $\alpha < 2$ and $\beta > 0$. For applications to forest fires, the relevant values turn out to belong to Domain I of Figure \ref{fig:phase_diagram}. We thus focus on this domain, i.e. we assume the following in the remainder of the paper.

\begin{assumption} \label{ass:alpha_beta}
The exponents $\alpha$ and $\beta$ satisfy
\begin{equation} \label{eq:assump_exp}
\alpha \in \bigg( \frac{3}{4},2 \bigg) \quad \text{and} \quad \beta > \alpha.
\end{equation}
\end{assumption}

As we will see, the most interesting behavior arises precisely in this domain. We prove that for any such $\alpha$ and $\beta$, the holes do not have a significant effect on the connectedness of the lattice, in the following sense. As $m \to \infty$, for values $p \in (p_c,1)$ satisfying $L(p) \asymp m$, percolation outside the holes stays ``near-critical'': it is comparable to critical percolation up to scales of order $m$, and to supercritical percolation on larger scales (see Sections \ref{sec:four_arm_stability} and \ref{sec:other_stability} for precise statements).

\subsection{Crossing holes} \label{sec:holes_crossing}

Recall that from now on (and until the end of Section \ref{sec:other_stability}), we consider a sequence of measures $\PPh_p^{(m)} = \PPh_p^{\pi^{(m)},\rho^{(m)}}$, where we assume that $\rho^{(m)}$ and $(\pi^{(m)}_v)_{v \in V}$ satisfy \eqref{eq:assump_holes} for some given $c_1, c_2, c_3 \in (0,+\infty)$, and $(\alpha,\beta)$ as in \eqref{eq:assump_exp}, i.e. in Domain I. All the asymptotic results stated below for $m \to \infty$ are uniform in such $\rho^{(m)}$ and $(\pi^{(m)}_v)_{v \in V}$, although we will not repeat it every time, for the sake of conciseness.

The following lemma turns out to be particularly handy, and we make repeated use of it in our proofs. For an annulus $A = \Ann_{n_1,n_2}(z)$ ($z \in \CC$, $1 \leq n_1 < n_2$), we introduce the event that it is ``crossed'' by a hole, i.e.
\begin{equation} \label{eq:def_H}
\calH(A) := \big\{ \exists v \in V \: : \: H_v \cap \partial \Ball_{n_1}(z) \neq \emptyset \text{ and } H_v \cap \partial \Ball_{n_2}(z) \neq \emptyset \big\}.
\end{equation}
Note that in this definition, we do not require the vertex $v$ to be in $A$: the crossing hole is allowed to be centered outside of $A$. Occasionally, we will use the straightforward generalization of \eqref{eq:def_H} when instead of $\Ball_{n_1}(z)$ and $\Ball_{n_2}(z)$, we have rectangles $R$ and $R'$ with $R \subseteq R'$.

\begin{lemma} \label{lem:no_crossing}
There exist $C$, $C'$ (depending on $c_1$, $c_2$, $c_3$, $\alpha$, $\beta$) such that the following holds. For all $m \geq 1$, for all annuli $A = \Ann_{n_1,n_2}(z)$ with $z \in V$ and $1 \leq n_1 \leq \frac{n_2}{2}$,
\begin{equation}
\PPh^{(m)} \big( \calH(A) \big) \leq \frac{C}{m^{\beta - \alpha}} e^{-C' n_1/m}.
\end{equation}
\end{lemma}

\begin{figure}
\begin{center}

\includegraphics[width=.5\textwidth]{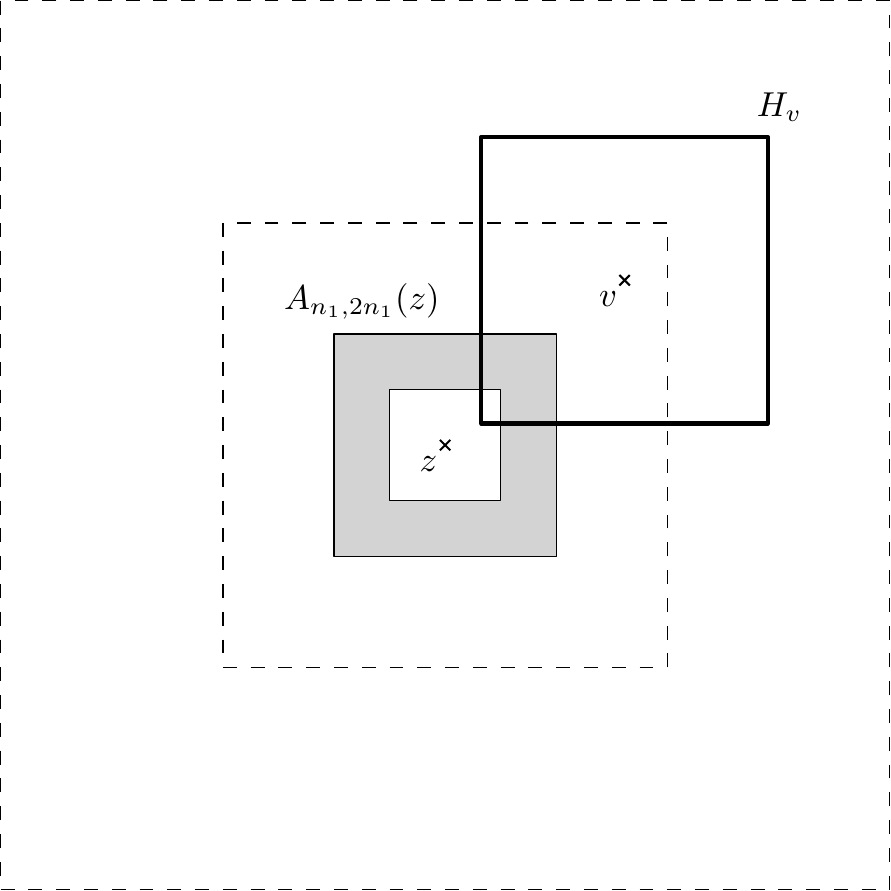}
\caption{\label{fig:crossing_hole} Example of a hole $H_v$ crossing the (gray) annulus $\Ann_{n_1,2 n_1}(z)$.}

\end{center}
\end{figure}

\begin{proof}[Proof of Lemma \ref{lem:no_crossing}]
In this proof, we use $C_1, C_2, \ldots$ to denote ``constants'' of which the precise value does not matter. They are allowed to depend on $c_1$, $c_2$, $c_3$, $\alpha$, and $\beta$, but not on $m$, $n_1$, $n_2$, or $z$. A similar remark holds for most of the proofs in the remainder of this paper.

Since obviously $\calH(A) \subseteq \calH( \Ann_{n_1,2 n_1}(z) )$, we may assume \emph{wlog} that $n_2 = 2 n_1$. We consider all possible locations for the center $v$ of a crossing hole. For that, we introduce the concentric annuli $\Ann_{2^i n_1, 2^{i+1} n_1}(z)$ ($i \geq 1$), as well as the ball $\Ball_{2 n_1}(z)$ (see Figure \ref{fig:crossing_hole}). We note that if $v \in \Ann_{2^i n_1, 2^{i+1} n_1}(z)$ for $i \geq 1$, then necessarily $r_v \geq \frac{1}{2} \cdot 2^i n_1 = 2^{i-1} n_1$, and the same holds true when $v \in \Ball_{2 n_1}(z)$, with $i = 0$. Hence,
\begin{align}
\PPh^{(m)} \big( \calH(A) \big) & \leq \sum_{v \in V} \PPh^{(m)} \big( H_v \cap \partial \Ball_{n_1}(z) \neq \emptyset \text{ and } H_v \cap \partial \Ball_{2 n_1}(z) \neq \emptyset \big) \nonumber \\
& \leq (c_3 m^{-\beta}) \cdot \sum_{i \geq 0} \Big( \big| \Ball_{2^{i+1} n_1}(z) \big| \cdot \rho^{(m)}\big( [2^{i-1} n_1, +\infty) \big) \Big) \nonumber \\
& \leq (c_3 m^{-\beta}) \cdot \sum_{i \geq 0} \Big( C_1 (2^{i+1} n_1)^2 \cdot c_1 (2^{i-1} n_1)^{\alpha-2} e^{-c_2 2^{i-1} n_1 / m} \Big) \label{eq:no_crossing_hole_pf1}
\end{align}
(using the assumption \eqref{eq:assump_holes} for $\pi^{(m)}$, and then for $\rho^{(m)}$).

We now distinguish the two cases $n_1 \leq m$ and $n_1 > m$. We first assume $n_1 \leq m$, and we let $I := \big\lfloor \log_2 \big( \frac{m}{n_1} \big) \big\rfloor \geq 0$ (so that $2^I n_1 \leq m \leq 2^{I+1} n_1$). We subdivide the sum in \eqref{eq:no_crossing_hole_pf1} into two sums, over $i \leq I+1$ and $i \geq I+2$. On the one hand,
\begin{equation} \label{eq:no_crossing_hole_pf2}
\sum_{i = 0}^{I+1} \Big( (2^{i+1} n_1)^2 \cdot (2^{i-1} n_1)^{\alpha-2} e^{-c_2 2^{i-1} n_1 / m} \Big) \leq C_2 n_1^{\alpha} \cdot \sum_{i = 0}^{I+1} (2^i)^{\alpha} \leq C_3 (2^I n_1)^{\alpha} \leq C_3 m^{\alpha}
\end{equation}
(we used that $\alpha > 0$). On the other hand,
\begin{equation} \label{eq:no_crossing_hole_pf3}
\sum_{i \geq I+2} \Big( (2^{i+1} n_1)^2 \cdot (2^{i-1} n_1)^{\alpha-2} e^{-c_2 2^{i-1} n_1 / m} \Big) \leq C_4 m^{\alpha} \cdot \sum_{i \geq 0} (2^i)^{\alpha} e^{-c_2 2^i} \leq C_5 m^{\alpha}.
\end{equation}
We can now obtain the desired upper bound by combining \eqref{eq:no_crossing_hole_pf1}, \eqref{eq:no_crossing_hole_pf2} and \eqref{eq:no_crossing_hole_pf3}.

In the case $n_1 > m$, we write
\begin{align}
\sum_{i \geq 0} \Big( (2^{i+1} n_1)^2 \cdot (2^{i-1} n_1)^{\alpha-2} e^{-c_2 2^{i-1} n_1 / m} \Big) & \leq C_6 n_1^{\alpha} e^{-c_2 n_1 / (4m)} \cdot \sum_{i \geq 0} (2^i)^{\alpha} e^{-c_2 (2^{i-1}-2^{-2}) n_1 / m} \nonumber \\
& \leq C_6 m^{\alpha} \Big( \frac{n_1}{m} \Big)^{\alpha} e^{-c_2 n_1 / (4m)} \cdot \sum_{i \geq 0} (2^i)^{\alpha} e^{-c_2 2^{i-2}} \nonumber \\
& \leq C_7 m^{\alpha} e^{-c_2 n_1 / (8m)} \label{eq:no_crossing_hole_pf4},
\end{align}
which completes the proof.
\end{proof}

For an annulus $A = \Ann_{n_1,n_2}(z)$ with $z \in V$ and $1 \leq n_1 \leq \frac{n_2}{2}$, we define a \emph{big hole} in $A$ as a hole $H_v$, $v \in V$, that crosses one of the sub-annuli $\Ann_{2^h, 2^{h+1}}(z) \subseteq A$, with $h$ a non-negative integer. Lemma \ref{lem:no_crossing} provides immediately an upper bound on the probability that such a hole exists: there exists $C$ such that for all $m \geq 1$,
\begin{equation} \label{eq:at_least_one_hole}
\PPh^{(m)} \bigg( \bigcup_{h \geq 0} \calH \big( \Ann_{2^h, 2^{h+1}} \big) \bigg) \leq C \frac{\log m}{m^{\beta - \alpha}}.
\end{equation}

For an annulus $A = \Ann_{n_1,n_2}(z)$, we introduce the following sub-event of $\calH(A)$:
$$\overline{\calH}(A) := \big\{ \exists v \in V \: : \: H_v \cap \partial \Ball_{n_1}(z) \neq \emptyset, \: H_v \cap \partial \Ball_{n_2}(z) \neq \emptyset, \: H_v \cap \partial \Ball_{\frac{n_1}{2}}(z) = \emptyset, \text{ and } H_v \cap \partial \Ball_{2 n_2}(z) = \emptyset \big\}$$
(i.e. $A$ is crossed by a hole $H_v$ which crosses neither $\Ann_{\frac{n_1}{2}, n_2}(z)$ nor $\Ann_{n_1, 2 n_2}(z)$, so that $\Ann_{n_1, n_2}(z)$ is approximately ``maximal''). For technical reasons, we also consider
$$\overline{\overline{\calH}}(A) := \big\{ \exists v \in V \: : \: H_v \cap \partial \Ball_{n_1}(z) \neq \emptyset, \: H_v \cap \partial \Ball_{n_2}(z) \neq \emptyset, \: H_v \nsupseteq \Ball_{n_1}(z), \text{ and } H_v \cap \partial \Ball_{2 n_2}(z) = \emptyset \big\}$$
(note that $\calH(A) \supseteq \overline{\overline{\calH}}(A) \supseteq \overline{\calH}(A)$).

\begin{lemma} \label{lem:Hbar}
There exist $C$, $C'$ (depending on $c_1$, $c_2$, $c_3$, $\alpha$, $\beta$) such that the following holds. For all $m \geq 1$, for all annuli $A = \Ann_{n_1,n_2}(z)$ with $z \in V$ and $1 \leq n_1 \leq \frac{n_2}{2}$,
\begin{equation} \label{eq:Hbar}
\PPh^{(m)} \big( \overline{\overline{\calH}}(A) \big) \leq C m^{-\beta} n_1 n_2^{\alpha-1} e^{-C' n_2 / m}.
\end{equation}
\end{lemma}

\begin{proof}[Proof of Lemma \ref{lem:Hbar}]
This follows from a similar computation as for Lemma \ref{lem:no_crossing}. If $\overline{\overline{\calH}}(\Ann_{n_1,n_2}(z))$ occurs, then the properties in its definition are satisfied by $H_v$ for some $v \in V$. Necessarily, $\| v - z \| \in \big( \frac{n_2 - n_1}{2}, \frac{2 n_2 + n_1}{2} \big]$, and for such a $v$, $r_v$ must satisfy the inequalities
$$\max \big( \|v - z\| - n_1, n_2 - \|v - z\| \big) \leq r_v \leq \min \big( \|v - z\| + n_1, 2 n_2 - \|v - z\| \big).$$
We deduce
\begin{align*}
\PPh^{(m)} \big( \overline{\overline{\calH}}(\Ann_{n_1, n_2}(z)) \big) & \leq \sum_{r = \frac{n_2 - n_1}{2}}^{\frac{n_2 + n_1}{2}} \sum_{v \text{ s.t. } \|v-z\|=r} \PPh^{(m)} \big( H_v \neq \emptyset, \: r_v \in [n_2 - r, r + n_1] \big)\\
& \hspace{2cm} + \sum_{r = \frac{n_2 + n_1}{2}}^{\frac{2 n_2 + n_1}{2}} \sum_{v \text{ s.t. } \|v-z\|=r} \PPh^{(m)} \big( H_v \neq \emptyset, \: r_v \in [r - n_1, r + n_1] \big).
\end{align*}
Note that in both sums, $r_v$ takes only values $\geq \frac{n_2 - n_1}{2}$, and that in each term, $r_v$ ranges over an interval of length at most $2 n_1$. Hence, also noting that the number of vertices $v$ with $\|v-z\|=r$ is at most $C_1 r$,
\begin{align*}
\PPh^{(m)} \big( \overline{\overline{\calH}}(\Ann_{n_1, n_2}(z)) \big) & \leq (c_3 m^{-\beta}) \cdot 2 \cdot C_1 \frac{2 n_2 + n_1}{2} \cdot \sum_{r \geq \frac{n_2 - n_1}{2}} \rho^{(m)} \big( [r, r + 2 n_1] \big)\\
& \leq C_2 m^{-\beta} n_2 (2 n_1 + 1) \rho^{(m)} \Big( \Big[ \frac{n_2}{4}, +\infty \Big) \Big)
\end{align*}
(since $\frac{n_2 - n_1}{2} \geq \frac{n_2}{4}$). We finally obtain
\begin{equation}
\PPh^{(m)} \big( \overline{\overline{\calH}}(\Ann_{n_1, n_2}(z)) \big) \leq C_3 m^{-\beta} n_2 n_1 (n_2)^{\alpha-2} e^{-C' n_2 / m},
\end{equation}
which gives \eqref{eq:Hbar}.
\end{proof}

\begin{remark} \label{rem:Hbarbar}
Note that in the previous proof, the center $v$ of a hole with the desired properties must satisfy $\| v - z \| > \frac{n_2 - n_1}{2} \geq \frac{n_1}{2}$. This shows that the event $\overline{\overline{\calH}}(\Ann_{n_1,n_2}(z))$ only depends on the holes centered in $\Ann_{\frac{n_1}{2},\infty}(z)$.
\end{remark}

For $n_1 \geq 1$ and $z \in V$, we also define
\begin{equation} \label{eq:def_Hbarbarstar}
\overline{\overline{\calH}}(\Ann_{n_1,*}(z)) := \bigcup_{i \geq 1} \overline{\overline{\calH}}(\Ann_{n_1, 2^i n_1}(z)).
\end{equation}

\section{Four-arm stability} \label{sec:four_arm_stability}

Recall that we are considering probability measures $\PPh_p^{(m)}$ satisfying Assumption~\ref{ass:rho_pi} and Assumption~\ref{ass:alpha_beta}. We would like to derive stability properties for percolation with impurities, i.e to show that under certain hypotheses, the holes do not affect too much the connectivity properties of the percolation configuration. As usual when studying near-critical percolation and related processes, it is crucial to obtain first a good control on the probability of four-arm events, which we do in this section, before deriving further stability results in Section \ref{sec:other_stability}.

\subsection{Notation and result}

Recall the notation introduced in Section \ref{sec:setting}, in particular the paragraph containing \eqref{eq:def_pi}. We will prove that $\PPh_p^{(m)} \big( \arm_4(n_1,n_2) \big)$ stays of order at most $C \pi_4(n_1,n_2)$, for some constant $C = C(c_1,c_2,c_3,\alpha,\beta)$, uniformly for $n_1 \leq \frac{n_2}{32} \leq n_2 \leq m$, and $p$ in the near-critical window $\{p' \: : \: L(p') \geq n_2\}$. In other words, our stability result for four arms, as well as our other stability results obtained later, are stated for scales up to $m \wedge L(p)$: the system remains near-critical on scales which are at the same time below $L(p)$ (which is not surprising), and below $m$, which can also be seen as a ``characteristic length'' (this will become more clear later, from the way $m$ arises in our applications).

We actually prove a stronger result, Theorem \ref{thm:four_arm_stability} below. Before stating it, we need to introduce some notation. The objects that we consider depend both on the configuration $\omega \in \Omega$ and on the collection of holes. However, to keep our notation short, we will only emphasize the dependence on $\omega$.

For $\omega \in \Omega$ and $U \subseteq V$, we denote by $(\omega^{(U)}) := (\omega^{(U)}(v))_{v \in V}$ the configuration obtained from $\omega$ by removing the holes centered in $U$, i.e.
$$\omega^{(U)}(v) := \omega(v) \ind_{v \notin \bigcup_{u \in U} H_u}$$
(recall that $H_u$ can be empty, in the case where there is no hole centered on $u$). For an annulus $A := \Ann_{n_1,n_2}(z)$, let
\begin{equation} \label{eq:def_W4}
\calW_4 ( A ) := \big\{ \exists U \subseteq V \: : \: \omega^{(U)} \text{ satisfies } \arm_4( A ) \big\}.
\end{equation}
In other words, $\calW_4 ( A )$ is the event that the configuration $\omega$ together with a subcollection of the holes satisfy $\arm_4( A )$.

\begin{theorem} \label{thm:four_arm_stability}
Let $K \geq 1$. There exists $C = C(c_1,c_2,c_3,\alpha,\beta,K) \in (0,+\infty)$ such that, for all $m$ large enough, the following holds. For all $p \in (0,1)$, and all $1 \leq n_1 < n_2 \leq K (m \wedge L(p))$,
\begin{equation} \label{eq:four_arm_stability}
\PPh_p^{(m)} \big( \calW_4( \Ann_{n_1,n_2} ) \big) \leq C \cdot \pi_4(n_1,n_2).
\end{equation}
\end{theorem}

Note that the reverse inequality (with a different $C$) follows immediately from the definition \eqref{eq:def_W4} and the classical stability result \eqref{eq:near_critical_arm}.

\begin{remark} \label{rem:Ke_GPS}
Stability results for arm events go back to the celebrated work by Kesten \cite{Ke1987} (where they played a crucial role to establish certain scaling relations). More recently, Garban, Pete and Schramm built further on these ideas \cite{GPS2013a} (where it was one of the many ingredients in their construction of the scaling limits of near-critical and dynamical percolation), and modified the arguments so as to incorporate more flexibility, see Lemma 8.4 in that paper (we follow some of their notation). Both of these works were in the context of single-site updates (impurities), and we expand the techniques further, into the situation of ``heavy-tailed'' impurities, where new subtle complications arise, and a more delicate analysis is required.
\end{remark}

\subsection{Proof of Theorem \ref{thm:four_arm_stability}}

\begin{proof}
First, we observe that the result holds for $n_2 \leq 1000 n_1$, since in this case, $\pi_4(n_1, n_2) \geq C_1$ for some universal constant $C_1$ (this follows easily from \eqref{eq:RSW}). Also, it is enough to prove the result for all $n_1$ and $n_2$ of the form $n_1 = 2^i$ and $n_2 = 2^j$, with $2^j \leq 2 K (m \wedge L(p))$. We prove it by induction over $j$ and $(j-i)$. From our previous observation, it holds for $j - i \leq 6$.

Now, let $i \geq 1$ and $j \geq i+7$ (with $2^j \leq 2 K (m \wedge L(p))$), and assume that the desired inequality \eqref{eq:four_arm_stability} holds true for all smaller values of $j$, and also for the same $j$ but all larger values of $i$. Here, we assume that \eqref{eq:four_arm_stability} is valid for some appropriate constant $C$: we explain later how to choose it. Let $\calD$ be the event that $\arm_4( \Ann_{2^{i+3},2^{j-3}} )^c$ holds without the holes.

\subsubsection{Case $\alpha > 1$} \label{sec:case_alpha_1}

We first consider $\alpha \in (1,2)$, since the combinatorics in the proof turns out to be somewhat simpler in this case. Moreover, as we explain in Section \ref{sec:comparison_holes}, our applications to forest fire processes involve only values of $\alpha$ in this interval. In Section \ref{sec:general_alpha}, we treat the general case $\alpha \in \big(\frac{3}{4}, 2 \big)$.

We introduce the following two events (recall the definition of a big hole above \eqref{eq:at_least_one_hole}).
\begin{itemize}
\item $\calE_1 := \{$there is no big hole in $\Ann_{2^i, 2^j} \}$.

\item $\calE_2 := \{$there is at least one big hole in $\Ann_{2^i, 2^j} \}$.
\end{itemize}
We start by writing
\begin{align}
\PPh_p^{(m)} \big( \calW_4 ( \Ann_{2^i, 2^j} ) \cap \calD \big) & \leq \PPh_p^{(m)} \big( \calW_4( \Ann_{2^i, 2^j} ) \cap \calD \cap \calE_1 \big) + \PPh_p^{(m)} \big( \calW_4( \Ann_{2^i, 2^j} ) \cap \calE_2 \big) \nonumber \\[1mm]
& =: (\text{Term } 1) + (\text{Term } 2). \label{eq:fourarm_term1and2}
\end{align}
We now handle these two terms separately, showing that each of them is a $o( \pi_4(2^i, 2^j) )$ as $m \to \infty$.

\underline{Term $1$:} Suppose that $\calW_4( \Ann_{2^i, 2^j} )$ and $\calD$ occur. Take $\omega$ and $U$ as in the definition of $\calW_4( \Ann_{2^i, 2^j} )$, and let $\omega' = \omega^{(U)}$. Hence,
\begin{itemize}
\item $\omega'$ satisfies $\arm_4( \Ann_{2^i, 2^j} )$,

\item while $\omega$ does not satisfy $\arm_4( \Ann_{2^{i+3},2^{j-3}} )$ (since $\omega \in \calD$).
\end{itemize}
We first ``add'' the holes with centers in $(A')^c$, where $A' := \Ann_{2^{i+2},2^{j-2}}$. More precisely, we consider the configuration $\omega^{(U \cap (A')^c)}$. From the event $\calE_1$ that there is no big hole in $\Ann_{2^i, 2^j}$, $\arm_4( \Ann_{2^{i+3},2^{j-3}} )$ is still not satisfied at this stage. Indeed, none of the holes centered in $(A')^c$ can intersect $\Ann_{2^{i+3},2^{j-3}}$ so they have no influence on the occurrence (or not) of $\arm_4( \Ann_{2^{i+3},2^{j-3}} )$.

We then add one by one the holes of $\omega'$ that are centered in the annulus $A'$, until $\arm_4( \Ann_{2^i, 2^j} )$ is satisfied. Let $\hat{\omega}$ denote the corresponding configuration, and let $H_v$ be the last added hole ($v \in A'$), which is thus ``pivotal''. Let $l \in \{i+2, \ldots, j-3\}$ be such that $v \in \Ann_{2^l, 2^{l+1}}$. From the event $\calE_1$, $H_v$ does not intersect $\Ball_{2^{l-1}}$ and $\din \Ball_{2^{l+2}}$. The configuration $\hat{\omega}$ satisfies $\arm_4( \Ann_{r_v, 2^{l-1}}(v) )$, which does not involve the regions $\Ball_{2^{l-1}}$ and $(\Ball_{2^{l+2}})^c$. In these two regions, $\hat{\omega}$ satisfies, respectively, $\arm_4( \Ann_{2^i, 2^{l-1}} )$ and $\arm_4( \Ann_{2^{l+2}, 2^j} )$. We thus obtain
\begin{equation}
(\text{Term } 1) \leq \sum_{l = i+2}^{j-3} \sum_{v \in \Ann_{2^l, 2^{l+1}}} \sum_{h=1}^l \PPh_p^{(m)} \big( H_v \neq \emptyset, \: r_v \in [2^h, 2^{h+1}) \big) \cdot \PPh_p^{(m)} \big( \calW_4^{(1)} \big) \cdot \PPh_p^{(m)} \big( \calW_4^{(2)} \big) \cdot \PPh_p^{(m)} \big( \calW_4^{(3)} \big),
\end{equation}
with the folowing events:
\begin{itemize}
\item $\calW_4^{(1)} := \big\{ \exists U^{(1)} \subseteq \Ann_{2^{l-2},2^{l+3}} \: : \: \omega^{(U^{(1)})}$ satisfies $\arm_4( \Ann_{2^{h+1}, 2^{l-1}}(v) ) \big\}$,

\item $\calW_4^{(2)} := \big\{ \exists U^{(2)} \subseteq \Ball_{2^{l-2}} \: : \: \omega^{(U^{(2)})}$ satisfies $\arm_4( \Ann_{2^i, 2^{l-3}} ) \big\}$,

\item $\calW_4^{(3)} := \big\{ \exists U^{(3)} \subseteq \Ann_{2^{l+3},\infty} \: : \: \omega^{(U^{(3)})}$ satisfies $\arm_4( \Ann_{2^{l+4}, 2^j} ) \big\}$.
\end{itemize}
So, informally speaking, $\calW_4^{(1)}$ is the event that $\calW_4( \Ann_{2^{h+1}, 2^{l-1}}(v) )$ occurs ``with only the holes centered in $\Ann_{2^{l-2},2^{l+3}}$'', and analogously for $\calW_4^{(2)}$ and $\calW_4^{(3)}$ (in the remainder of the paper, we will frequently use such informal terminology, with similar meaning). Note that these three events are independent (see Figure \ref{fig:four_arm_stability} for an illustration). We also note that, if $n_1 \geq n_2$, we consider $\arm_4( \Ann_{n_1, n_2}(z) )$ to be automatically satisfied, and $\pi_4(n_1, n_2)$ to be equal to $1$.

\begin{figure}
\begin{center}

\includegraphics[width=.95\textwidth]{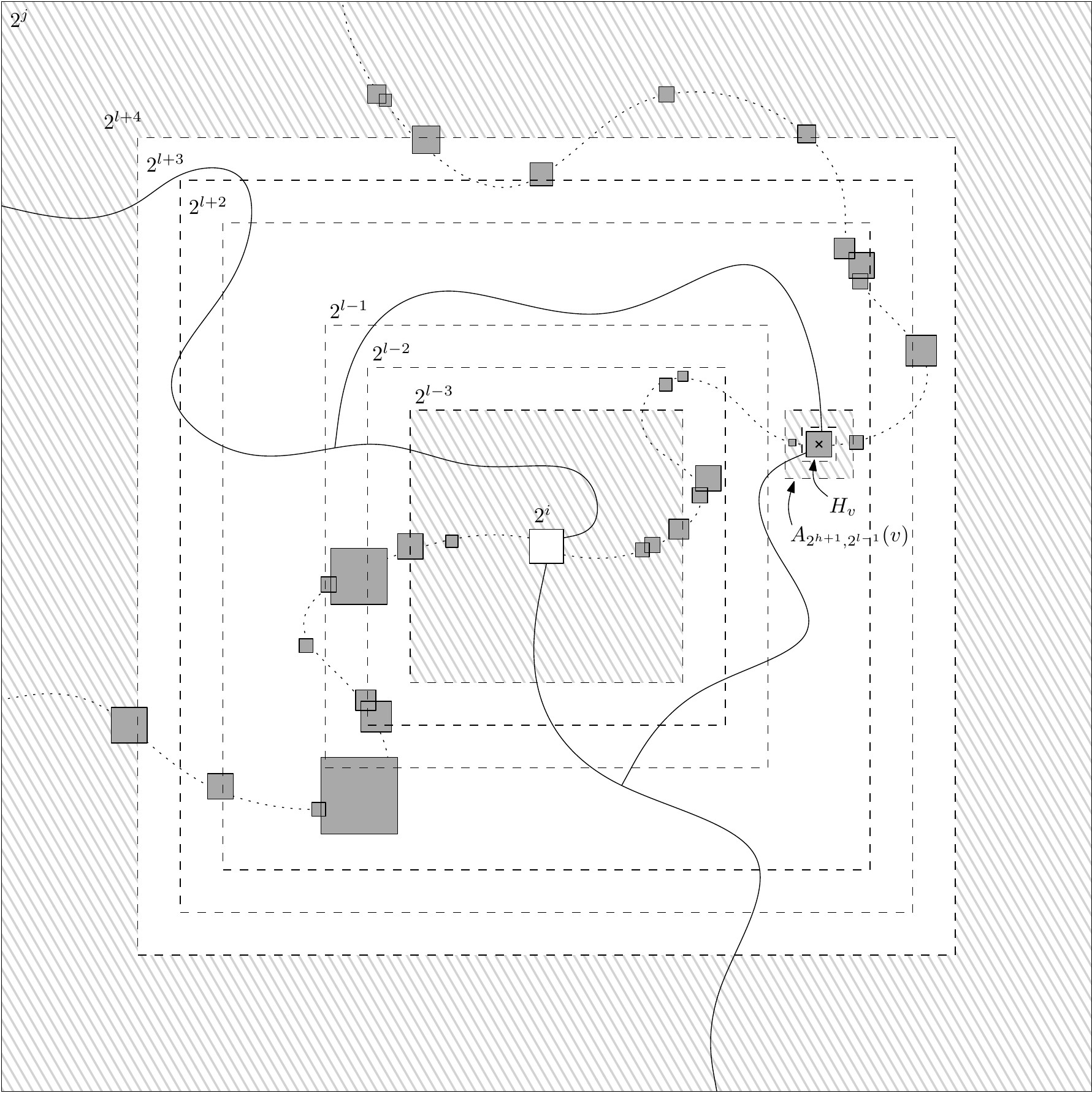}
\caption{\label{fig:four_arm_stability} The three shaded regions (in grey) are the annuli $\Ann_{2^{h+1}, 2^{l-1}}(v)$, $\Ann_{2^i, 2^{l-3}}$ and $\Ann_{2^{l+4}, 2^j}$, where we look for four arms with alternating types (in the events $\calW_4^{(1)}$, $\calW_4^{(2)}$ and $\calW_4^{(3)}$, respectively). The annuli $\Ann_{2^{l-3},2^{l-2}}$, $\Ann_{2^{l-2},2^{l-1}}$, $\Ann_{2^{l+2},2^{l+3}}$ and $\Ann_{2^{l+3},2^{l+4}}$ are ``safety areas'': we know from $\calE_1$ that none of them is crossed by a hole.}

\end{center}
\end{figure}

It then follows from the induction hypothesis that
\begin{align*}
(\text{Term } 1) & \leq \sum_{l = i+2}^{j-3} \sum_{v \in \Ann_{2^l, 2^{l+1}}} \sum_{h=1}^l c_3 m^{-\beta} \rho^{(m)} \big( [2^h, +\infty) \big) C^3 \pi_4(2^{h+1}, 2^{l-1}) \pi_4(2^i, 2^{l-3}) \pi_4(2^{l+4}, 2^j)\\
& \leq C_1 m^{-\beta} \sum_{l = i+2}^{j-3} \big| \Ball_{2^{l+1}} \big| \sum_{h=1}^l c_1 (2^h)^{\alpha-2} \pi_4(2^{h+1}, 2^{l-1}) \pi_4(2^i, 2^j),\\
& \leq C_2 m^{-\beta} \pi_4(2^i, 2^j) \sum_{l = i+2}^{j-3} 2^{2l} \sum_{h=1}^l (2^h)^{\alpha-2} \pi_4(2^{h+1}, 2^{l-1}),
\end{align*}
using the assumption \eqref{eq:assump_holes} for $\pi^{(m)}$ and $\rho^{(m)}$, and the properties \eqref{eq:extendability} and \eqref{eq:quasi_mult} for $\pi_4$. Lemma \ref{lem:ratio_limit} implies that for any $\upsilon > 0$ fixed,
\begin{align}
\sum_{l = i+2}^{j-3} 2^{2l} \sum_{h=1}^l (2^h)^{\alpha-2} \pi_4(2^{h+1}, 2^{l-1}) & \leq C_3 \sum_{l = i+2}^{j-3} 2^{2l} \sum_{h=1}^l (2^h)^{\alpha-2} \bigg( \frac{2^{h+1}}{2^{l-1}} \bigg)^{\frac{5}{4} - \upsilon} \nonumber \\
& \leq C_4 \sum_{l = i+2}^{j-3} (2^l)^{\frac{3}{4} + \upsilon} \sum_{h=1}^l (2^h)^{\alpha-\frac{3}{4} - \upsilon}. \label{eq:rem_domain2_1}
\end{align}
By assumption \eqref{eq:assump_exp}, $\alpha > \frac{3}{4}$ (actually, in this subsection we even assume $\alpha > 1$, but we do not use it at this point), so we can pick $\upsilon > 0$ sufficiently small so that $\alpha - \frac{3}{4} - \upsilon > 0$. We deduce
\begin{equation} \label{eq:rem_domain2_2}
\sum_{l = i+2}^{j-3} 2^{2l} \sum_{h=1}^l (2^h)^{\alpha-2} \pi_4(2^{h+1}, 2^{l-1}) \leq C_5 \sum_{l = i+2}^{j-3} (2^l)^{\frac{3}{4} + \upsilon} (2^l)^{\alpha-\frac{3}{4} - \upsilon} \leq C_6 (2^j)^{\alpha} \leq C_6 (2 K m)^{\alpha}.
\end{equation}
We thus obtain
\begin{equation} \label{eq:fourarm_term1}
(\text{Term } 1) \leq C_7 m^{\alpha-\beta} \pi_4(2^i, 2^j) = o( \pi_4(2^i, 2^j) )
\end{equation}
as $m \to \infty$, which is the desired upper bound for Term $1$.

\underline{Term $2$:} Assume that the event $\calW_4( \Ann_{2^i, 2^j} ) \cap \calE_2$ occurs. There exists $h \in \{i,\ldots,j-1\}$ for which $\Ann_{2^h,2^{h+1}}$ is crossed by a big hole, and we let $\bar{h}$ be the smallest such $h$. For $h \in \{i,\ldots,j-1\}$,
\begin{equation} \label{eq:four_arms_term2}
\PPh_p^{(m)} \big( \calW_4( \Ann_{2^i, 2^j} ) \cap \calE_2 \cap \{\bar{h} = h\} \big) \leq \PPh_p^{(m)} \big( \calW_4^{(1)} \cap \tilde{\arm}_{(oo)}(\Ann_{2^h,2^j}) \cap \overline{\overline{\calH}}(\Ann_{2^h,*}) \big),
\end{equation}
where $\tilde{\arm}_{(oo)}(\Ann_{2^h,2^j})$ denotes the event that the monochromatic two-arm event $\arm_{(oo)}(\Ann_{2^h,2^j})$ occurs without the holes, and $\calW_4^{(1)}$ the event that $\calW_4( \Ann_{2^i, 2^{h-2}} )$ occurs with only the holes centered in $\Ball_{2^{h-1}}$. Recall also the definition of $\overline{\overline{\calH}}(\Ann_{.,*})$ in \eqref{eq:def_Hbarbarstar}. Note that we use the event $\overline{\overline{\calH}}$, and not simply the event $\overline{\calH}$, in order to take into account the case $h = i$. Indeed, the big hole in this case may cross further annuli inside, and even cover the origin, but it is not allowed to cover the whole of $\Ball_{2^i}$ (since otherwise, no occupied arm in $\Ann_{2^i, 2^j}$ could exist).

It follows from Remark \ref{rem:Hbarbar} that $\overline{\overline{\calH}}(\Ann_{2^h,*})$ only depends on the holes centered in $\Ann_{2^{h-1},\infty}$, so that the three events in the right-hand side of \eqref{eq:four_arms_term2} are independent. Hence,
\begin{equation} \label{eq:term2_indep}
\PPh_p^{(m)} \big( \calW_4( \Ann_{2^i, 2^j} ) \cap \calE_2 \cap \{\bar{h} = h\} \big) \leq \PPh_p^{(m)} \big( \calW_4^{(1)} \big) \PP_p \big( \arm_{(oo)}(\Ann_{2^h,2^j}) \big) \PPh^{(m)} \big( \overline{\overline{\calH}}(\Ann_{2^h,*}) \big).
\end{equation}

We first claim that
\begin{equation} \label{eq:cross_and_2arm}
\PP_p \big( \arm_{(oo)}(\Ann_{2^h,2^j}) \big) \PPh^{(m)} \big( \overline{\overline{\calH}}(\Ann_{2^h,*}) \big) \leq C' m^{\alpha-\beta} \pi_4(2^h,2^j).
\end{equation}
Indeed, using the assumption $\alpha > 1$, we obtain from Lemma \ref{lem:Hbar} that
\begin{equation} \label{eq:Hbar_infty}
\PPh^{(m)} \big( \overline{\overline{\calH}}(\Ann_{n_1,*}(z)) \big) \leq \sum_{k \geq 1} \PPh^{(m)} \big( \overline{\overline{\calH}}(\Ann_{n_1, 2^k n_1}(z)) \big) \leq C'' m^{-\beta} n_1 m^{\alpha-1}.
\end{equation}
On the other hand, \eqref{eq:near_critical_arm} implies (since $2^j \leq 2 K L(p)$)
\begin{equation} \label{eq:near_critical_two_arms}
\PP_p \big( \arm_{(oo)}(\Ann_{2^h,2^j}) \big) \leq C_1 \pi_{(oo)}(2^h, 2^j).
\end{equation}
It follows from the inequality \eqref{eq:monochromatic} between the monochromatic and the polychromatic two-arm events, and Lemma \ref{lem:ratio_limit}, that: for some $\upsilon > 0$ small enough,
\begin{equation} \label{eq:2arm_ineq}
\PP_p \big( \arm_{(oo)}(\Ann_{2^h,2^j}) \big) \leq C_2 \bigg( \frac{2^h}{2^j} \bigg)^{\frac{1}{4} + \upsilon}.
\end{equation}
Hence, by combining \eqref{eq:Hbar_infty}, \eqref{eq:2arm_ineq}, and then $m^{-1} \leq 2 K (2^j)^{-1}$,
\begin{align*}
\PP_p \big( \arm_{(oo)}(\Ann_{2^h,2^j}) \big) \PPh^{(m)} \big( \overline{\overline{\calH}}(\Ann_{2^h,*}) \big) & \leq C_3 m^{-\beta} 2^h m^{\alpha-1} \bigg( \frac{2^h}{2^j} \bigg)^{\frac{1}{4} + \upsilon}\\
& \leq 2K C_3 m^{\alpha - \beta} \bigg( \frac{2^h}{2^j} \bigg)^{\frac{5}{4} + \upsilon}\\
& \leq C_4 m^{\alpha - \beta} \pi_4(2^h,2^j)
\end{align*}
(using Lemma \ref{lem:ratio_limit} for the last inequality), which establishes \eqref{eq:cross_and_2arm}.

By combining \eqref{eq:term2_indep} and \eqref{eq:cross_and_2arm}, and applying the induction hypothesis for $\PPh_p^{(m)} \big( \calW_4^{(1)} \big)$, we deduce (using also the properties \eqref{eq:extendability} and \eqref{eq:quasi_mult} for $\pi_4$)
\begin{equation}
\PPh_p^{(m)} \big( \calW_4( \Ann_{2^i, 2^j} ) \cap \calE_2 \cap \{\bar{h} = h\} \big) \leq C \pi_4(2^i, 2^{h-2}) \cdot C' m^{\alpha - \beta} \pi_4(2^h, 2^j) \leq C_5 m^{\alpha - \beta} \pi_4(2^i, 2^j).
\end{equation}
We then sum over the possible values of $\bar{h}$, producing an extra $\log m$ factor:
\begin{equation} \label{eq:fourarm_term2}
(\text{Term } 2) = \sum_{h=i}^{j-1} \PPh_p^{(m)} \big( \calW_4( \Ann_{2^i, 2^j} ) \cap \calE_2 \cap \{\bar{h} = h\} \big) \leq C_6 m^{\alpha - \beta} (\log m) \pi_4(2^i, 2^j).
\end{equation}
By combining \eqref{eq:fourarm_term1and2}, \eqref{eq:fourarm_term1}, and \eqref{eq:fourarm_term2}, and using that $\beta > \alpha$, we obtain (in the case $\alpha \in (1,2)$)
\begin{equation} \label{eq:fourarm_W4}
\PPh_p^{(m)} \big( \calW_4 ( \Ann_{2^i, 2^j} ) \cap \calD \big) = o( \pi_4(2^i, 2^j) ) \quad \text{as $m \to \infty$.}
\end{equation}
In Section \ref{sec:end_proof_four_arm}, we will show that this implies \eqref{eq:four_arm_stability}.

\subsubsection{General case $\alpha > \frac{3}{4}$} \label{sec:general_alpha}

We now prove \eqref{eq:fourarm_W4} for a general $\alpha \in \big(\frac{3}{4}, 2 \big)$. We need to introduce the following three events, for a well-chosen $M = M(\alpha,\beta)$.
\begin{itemize}
\item $\calE_1 := \{$there is no big hole in $\Ann_{2^i, 2^j} \}$.

\item $\calE_2 := \{$there are between $1$ and $M$ big holes in $\Ann_{2^i, 2^j} \}$.

\item $\calE_3 := \{$there are at least $M+1$ big holes in $\Ann_{2^i, 2^j} \}$.
\end{itemize}
We write
\begin{align}
\PPh_p^{(m)} \big( \calW_4 ( \Ann_{2^i, 2^j} ) \cap \calD \big) & \leq \PPh_p^{(m)} \big( \calW_4( \Ann_{2^i, 2^j} ) \cap \calD \cap \calE_1 \big) + \PPh_p^{(m)} \big( \calW_4( \Ann_{2^i, 2^j} ) \cap \calE_2 \big) + \PPh_p^{(m)} \big( \calE_3 \big) \nonumber \\[1mm]
& =: (\text{Term } 1) + (\text{Term } 2) + (\text{Term } 3). \label{eq:fourarm_term1and2and3}
\end{align}
Similarly to the case $\alpha \in (1,2)$, we need to show that each term is a $o( \pi_4(2^i, 2^j) )$ as $m \to \infty$.

\underline{Term $1$:} It can be handled in exactly the same way as before (as the reader can check, for that term we did not use the fact that $\alpha > 1$), and we obtain again \eqref{eq:fourarm_term1}.

\underline{Term $3$:} As we now explain, this term can be handled easily by choosing a sufficiently large value of $M$. For that, we derive an upper bound on the probability that there exists a large number of big holes, i.e. of $v_1, \ldots, v_{M+1} \in V$ distinct such that each $H_{v_i}$ ($1 \leq i \leq M+1$) is a big hole: we claim that there exists $M = M(\alpha,\beta)$ for which
\begin{equation} \label{eq:many_big_holes}
\PPh^{(m)} \big( \text{there are at least $M+1$ big holes in $\Ann_{2^i, 2^j}$} \big) = o( \pi_4(m) ) \quad \text{as $m \to \infty$.}
\end{equation}
Indeed, since the events $\{H_v$ is a big hole in $\Ann_{2^i, 2^j}\}$, $v \in V$, are independent, we have
\begin{align}
\PPh^{(m)} \big( \exists v_1, \ldots, v_{M+1} \in V \text{ distinct } & \text{s.t. all } (H_{v_i})_{1 \leq i \leq M+1} \text{ are big holes} \big) \nonumber \\
& \leq \Big( \PPh^{(m)} \big( \exists v \in V \text{ s.t. } H_v \text{ is a big hole} \big) \Big)^{M+1} \nonumber \\
& \leq (C_1)^{M+1} \frac{(\log m)^{M+1}}{m^{(\beta - \alpha) (M+1)}} \label{eq:proba_big_holes}
\end{align}
(where we used \eqref{eq:at_least_one_hole} for the last inequality). For any $\upsilon > 0$, $\pi_4(m) \geq C_2 m^{-\frac{5}{4} - \upsilon}$ for some $C_2 = C_2(\upsilon) > 0$ (from \eqref{eq:arm_exponent}). Hence, the right-hand side of \eqref{eq:proba_big_holes} is a $o( \pi_4(m) )$ for all $M$ large enough (so that $(\beta - \alpha) (M+1) > \frac{5}{4}$), which gives \eqref{eq:many_big_holes}. From now on, we assume $M$ to be chosen in that way, so that $(\text{Term } 3) = o( \pi_4(m) )$, and in particular
\begin{equation} \label{eq:fourarm_term3}
(\text{Term } 3) = o( \pi_4(2^i, 2^j) ) \quad \text{as $m \to \infty$.}
\end{equation}

\underline{Term $2$:} This term requires more care. Let us assume that the corresponding event holds, so that the number $b$ of big holes in $\Ann_{2^i,2^j}$ satisfies $1 \leq b \leq M$. We list which sub-annuli are crossed by such big holes: there are integers $1 \leq n \leq M$, $i \leq h_1 < h_2 < \ldots < h_n < j$ and $i < k_1 < k_2 < \ldots < k_n \leq j$, with $h_l < k_l$ ($1 \leq l \leq n$), such that the following subevent of $\calE_2$ holds, which we denote by $\tilde{\calE}_2 = \tilde{\calE}_2(n, h_1, \ldots, h_n, k_1, \ldots, k_n)$.
\begin{itemize}
\item For all $1 \leq l \leq n$, $\overline{\overline{\calH}}(\Ann_{2^{h_l},2^{k_l}})$ holds (unless $l = n$ and $k_n = j$, in which case we require $\overline{\overline{\calH}}(\Ann_{2^{h_n},*})$ instead),

\item \emph{and} no other sub-annuli are crossed by big holes: for all $i \leq h \leq j-1$ with $[h,h+1] \nsubseteq \bigcup_{1 \leq l \leq n} [h_l, k_l]$, $\calH(\Ann_{2^h,2^{h+1}})$ does not occur.
\end{itemize}
Note that it may be the case that $n < b$.

We now group the big holes as follows. We say that two successive intervals $[h_l,k_l]$ and $[h_{l+1},k_{l+1}]$ ``overlap'' if $h_{l+1} \leq k_l$. Consider a block of overlapping intervals $[h_{\underline{l}},k_{\underline{l}}], \ldots, [h_{\overline{l}},k_{\overline{l}}]$ ($1 \leq \underline{l} \leq \overline{l} \leq n$), i.e. such that any two successive intervals overlap. Later, we will ``label'' such a block simply by $\llbracket h_{\underline{l}},k_{\overline{l}} \rrbracket$. For simplicity, let us first assume that we are not in the case $\overline{l} = n$ and $k_n = j$. By ``independence'', and then using Lemma \ref{lem:Hbar}, we have
\begin{align}
\PPh^{(m)} \bigg( \bigcap_{l = \underline{l}}^{\overline{l}} \overline{\overline{\calH}}(\Ann_{2^{h_l},2^{k_l}}) \bigg) & \leq \prod_{l = \underline{l}}^{\overline{l}} \PPh^{(m)} \big( \overline{\overline{\calH}}(\Ann_{2^{h_l},2^{k_l}}) \big) \label{eq:prod_overlapping0} \\
& \leq \prod_{l = \underline{l}}^{\overline{l}} \big( C m^{-\beta} (2^{h_l}) (2^{k_l})^{\alpha-1} \big) \nonumber \\
& \leq \big( C m^{-\beta} \big)^{\overline{l} - \underline{l} + 1} 2^{h_{\underline{l}}} (2^{k_{\overline{l}}})^{\alpha-1} \prod_{l = \underline{l}}^{\overline{l}-1} \big( (2^{k_l})^{\alpha-1} (2^{h_{l+1}}) \big). \label{eq:prod_overlapping1}
\end{align}
For each term in the product, since $h_{l+1} \leq k_l$, we have
\begin{equation} \label{eq:prod_overlapping2}
(2^{k_l})^{\alpha-1} (2^{h_{l+1}}) \leq (2^{k_l})^{\alpha} \leq C_1 m^{\alpha},
\end{equation}
using also $2^{k_l} \leq 2^j \leq 2 K m$ (this is where we have to be careful that $\alpha$ might be $<1$, and use the ``overlapping'' assumption). On the other hand,
\begin{equation} \label{eq:prod_overlapping3}
(2^{k_{\overline{l}}})^{\alpha-1} = (2^{k_{\overline{l}}})^{\alpha} (2^{k_{\overline{l}}})^{-1} \leq C_1 m^{\alpha} (2^{k_{\overline{l}}})^{-1}.
\end{equation}
We deduce from \eqref{eq:prod_overlapping1}, \eqref{eq:prod_overlapping2}, and \eqref{eq:prod_overlapping3}, that
\begin{equation} \label{eq:block_holes}
\PPh^{(m)} \bigg( \bigcap_{l = \underline{l}}^{\overline{l}} \overline{\overline{\calH}}(\Ann_{2^{h_l},2^{k_l}}) \bigg) \leq \big( C_2 m^{\alpha - \beta} \big)^{\overline{l} - \underline{l} + 1} 2^{h_{\underline{l}}} (2^{k_{\overline{l}}})^{-1}.
\end{equation}

This implies that for some $\upsilon > 0$ small enough, using \eqref{eq:2arm_ineq} (recall the definition of $\tilde{\arm}_{(oo)}$ from the line below \eqref{eq:four_arms_term2}),
\begin{align}
\PPh_p^{(m)} \bigg( \tilde{\arm}_{(oo)}(\Ann_{2^{h_{\underline{l}}},2^{k_{\overline{l}}}}) \cap \bigcap_{l = \underline{l}}^{\overline{l}} \overline{\overline{\calH}}(\Ann_{2^{h_l},2^{k_l}}) \bigg) & \leq \PP_p \big( \arm_{(oo)}(\Ann_{2^{h_{\underline{l}}},2^{k_{\overline{l}}}}) \big) \big( C_2 m^{\alpha - \beta} \big)^{\overline{l} - \underline{l} + 1} 2^{h_{\underline{l}}} (2^{k_{\overline{l}}})^{-1} \label{eq:upper_bound_block_lhs}\\
& \leq C_3 \bigg( \frac{2^{h_{\underline{l}}}}{2^{k_{\overline{l}}}} \bigg)^{\frac{1}{4} + \upsilon} \big( C_2 m^{\alpha - \beta} \big)^{\overline{l} - \underline{l} + 1} 2^{h_{\underline{l}}} (2^{k_{\overline{l}}})^{-1} \nonumber \\
& = C_3 \bigg( \frac{2^{h_{\underline{l}}}}{2^{k_{\overline{l}}}} \bigg)^{\frac{5}{4} + \upsilon} \big( C_2 m^{\alpha - \beta} \big)^{\overline{l} - \underline{l} + 1} \nonumber \\
& \leq \big( C_4 m^{\alpha - \beta} \big)^{\overline{l} - \underline{l} + 1} \pi_4(2^{h_{\underline{l}}}, 2^{k_{\overline{l}}}) \label{eq:upper_bound_block}
\end{align}
(where the last inequality follows from Lemma \ref{lem:ratio_limit}).

In the case $\overline{l} = n$ and $k_n = j$, the same reasonings apply, except that in the product \eqref{eq:prod_overlapping0}, $\PPh^{(m)} \big( \overline{\overline{\calH}}(\Ann_{2^{h_{\overline{l}}},2^{k_{\overline{l}}}}) \big)$ has to be replaced by $\PPh^{(m)} \big( \overline{\overline{\calH}}(\Ann_{2^{h_{\overline{l}}},*}) \big)$. It follows from Lemma \ref{lem:Hbar} that
$$\PPh^{(m)} \big( \overline{\overline{\calH}}(\Ann_{2^{h_{\overline{l}}},*}) \big) \leq C'_1 m^{-\beta} (2^{h_{\overline{l}}}) (\log m) \cdot \max \big( (2^{k_{\overline{l}}})^{\alpha-1}, m^{\alpha-1} \big) \leq C'_2 m^{-\beta} (2^{h_{\overline{l}}}) (\log m) m^{\alpha} (2^{k_{\overline{l}}})^{-1}$$
(the extra $\log m$ is here for the case $\alpha=1$), and the rest of the calculations is identical to those that led to \eqref{eq:upper_bound_block}, now with an additional $\log m$ factor.

We now group the intervals $[h_l,k_l]$ ($1 \leq l \leq n$) into \emph{maximal} blocks of overlapping intervals $\llbracket \tilde{h}_l,\tilde{k}_l \rrbracket$, where $i \leq \tilde{h}_1 < \tilde{k}_1 < \tilde{h}_2 < \ldots < \tilde{k}_q \leq j$, and $q$ ($\leq n \leq M$) is the number of such blocks. We denote by $n_l$ the number of overlapping intervals that the $l$th block contains, so that $n_1 + \ldots + n_q = n$. For $h < k$, we denote $\tilde{\calW}_4( \Ann_{2^h, 2^k} ) := \{ \calW_4( \Ann_{2^h, 2^k} )$ occurs with only the holes centered in $\Ann_{2^{h-1}, 2^{k+1}} \}$. For notational convenience, we set, for $h \geq k$, $\tilde{\calW}_4( \Ann_{2^h, 2^k} ) := \Omega$ and $\pi_4(2^h, 2^k) := 1$.

Observe that, on the event $\calW_4( \Ann_{2^i, 2^j} ) \cap \tilde{\calE}_2$, for each block $\llbracket \tilde{h}_l,\tilde{k}_l \rrbracket$ the event in the left-hand side of \eqref{eq:upper_bound_block_lhs} holds (with $h_{\underline{l}}$ and $k_{\overline{l}}$ replaced by $\tilde{h}_l$ and $\tilde{k}_l$, respectively), and that (if $l \leq q-1$) for the annulus between this block and the next one (i.e. the block $\llbracket \tilde{h}_{l+1},\tilde{k}_{l+1} \rrbracket$), the event $\tilde{\calW}_4( \Ann_{2^{\tilde{k}_l+2}, 2^{\tilde{h}_{l+1}-2}} )$ holds. Such considerations, together with appropriate use of independence (and application of \eqref{eq:upper_bound_block}) gives
\begin{align*}
\PPh_p^{(m)} \big( \calW_4( \Ann_{2^i, 2^j} ) \cap \tilde{\calE}_2 \big) & \leq \PPh_p^{(m)} \big( \tilde{\calW}_4( \Ann_{2^i, 2^{\tilde{h}_1-2}} ) \big) \bigg( \prod_{l=1}^{q-1} \PPh_p^{(m)} \big( \tilde{\calW}_4( \Ann_{2^{\tilde{k}_l+2}, 2^{\tilde{h}_{l+1}-2}} ) \big) \bigg) \PPh_p^{(m)} \big( \tilde{\calW}_4( \Ann_{2^{\tilde{k}_q+2}, 2^j} ) \big)\\
& \hspace{2cm} \cdot \bigg( \prod_{l=1}^q \big( C_4 m^{\alpha - \beta} (\log m) \big)^{n_l} \pi_4(2^{\tilde{h}_l}, 2^{\tilde{k}_l}) \bigg).
\end{align*}
Then, by applying $q+1$ times the induction hypothesis, we obtain
\begin{align*}
\PPh_p^{(m)} \big( \calW_4( \Ann_{2^i, 2^j} ) \cap \tilde{\calE}_2 \big) & \leq ( C \pi_4(2^i, 2^{\tilde{h}_1-2}) ) \bigg( \prod_{l=1}^{q-1} C \pi_4(2^{\tilde{k}_l+2}, 2^{\tilde{h}_{l+1}-2}) \bigg) ( C \pi_4(2^{\tilde{k}_q+2}, 2^j) )\\
& \hspace{2cm} \cdot \bigg( \prod_{l=1}^q \pi_4(2^{\tilde{h}_l}, 2^{\tilde{k}_l}) \bigg) \big( C_4 m^{\alpha - \beta} (\log m) \big)^{n_1 + \ldots + n_q}.
\end{align*}
This yields, using \eqref{eq:extendability} and \eqref{eq:quasi_mult} (for $\pi_4$) repeatedly,
\begin{align*}
\PPh_p^{(m)} \big( \calW_4( \Ann_{2^i, 2^j} ) \cap \tilde{\calE}_2 \big) & \leq C^{q+1} (C_5)^{2q} \pi_4(2^i, 2^j) \big( C_4 m^{\alpha - \beta} (\log m) \big)^n.
\end{align*}

Hence,
\begin{align*}
(\text{Term } 2) & = \PPh_p^{(m)} \big( \calW_4( \Ann_{2^i, 2^j} ) \cap \calE_2 \big)\\
& \leq \sum_{1 \leq n \leq M} \sum_{\substack{i \leq h_1 < \ldots < h_n < j\\ i < k_1 < \ldots < k_n \leq j\\ h_l < k_l (1 \leq l \leq n)}} \PPh_p^{(m)} \big( \calW_4( \Ann_{2^i, 2^j} ) \cap \tilde{\calE}_2(n, h_1, \ldots, h_n, k_1, \ldots, k_n) \big)\\
& \leq M (C_6 \log m)^{2M} C^{M+1} \big( C_4 m^{\alpha - \beta} (\log m) \big) \pi_4(2^i, 2^j)
\end{align*}
(in the last inequality, we used the fact that $m^{\alpha - \beta} (\log m) \to 0$ as $m \to \infty$, since $\beta > \alpha$), so
\begin{equation} \label{eq:fourarm_term2b}
(\text{Term } 2) = o( \pi_4(2^i, 2^j) ) \quad \text{as $m \to \infty$.}
\end{equation}
By combining \eqref{eq:fourarm_term1and2and3}, \eqref{eq:fourarm_term1}, \eqref{eq:fourarm_term2b}, and \eqref{eq:fourarm_term3}, we obtain again \eqref{eq:fourarm_W4}, now for the general case $\alpha \in \big(\frac{3}{4}, 2 \big)$.

\subsubsection{End of the proof of Theorem \ref{thm:four_arm_stability}} \label{sec:end_proof_four_arm}

We are now in a position to conclude. We can write
\begin{equation}
\PPh_p^{(m)} \big( \calW_4 ( \Ann_{2^i, 2^j} ) \big) \leq \PPh_p^{(m)} \big( \calD^c \big) + \PPh_p^{(m)} \big( \calW_4 ( \Ann_{2^i, 2^j} ) \cap \calD \big).
\end{equation}
We also have $\PPh_p^{(m)} \big( \calD^c \big) = \PP_p \big( \arm_4( \Ann_{2^{i+3},2^{j-3}} ) \big) \leq \hat{C} \pi_4(2^i, 2^j)$ (using \eqref{eq:extendability}, \eqref{eq:near_critical_arm}, and $2^j \leq 2 K L(p)$ for the inequality). Note that $\hat{C}$ depends only on $K$. Combining this with \eqref{eq:fourarm_W4}, we get that
\begin{equation}
\PPh_p^{(m)} \big( \calW_4 ( \Ann_{2^i, 2^j} ) \big) \leq \hat{C} \pi_4(2^i, 2^j) + o( \pi_4(2^i, 2^j) ) \quad \text{as $m \to \infty$.}
\end{equation}
This implies that if we choose $C > \hat{C}$, the inequality in \eqref{eq:four_arm_stability} holds for all $m$ large enough (depending on $c_1$, $c_2$, $c_3$, $\alpha$, $\beta$, and $K$), uniformly in $i$ and $j$ satisfying the requirements in the statement of the theorem. This completes the proof of Theorem \ref{thm:four_arm_stability}.
\end{proof}

\subsubsection{Remark on Domain II} \label{sec:rem_domain2}

%As we mentioned earlier, a very similar proof could be carried out in the case of $(\alpha, \beta)$ belonging to Domain II. We make a quick comment about how the computations need to be adapted in this case.

Note that, strictly speaking, by monotonicity Domain II is covered by Domain I (indeed, for every $(\alpha, \beta)$ in Domain II, we can find $(\alpha', \beta)$ in Domain I with $\alpha' > \alpha$). Still, it might be interesting to see ``what happens to our computations'' in the case of Domain II.

The main difference appears just below \eqref{eq:rem_domain2_1}: $\alpha < \frac{3}{4}$ so for any $\upsilon > 0$, $\alpha - \frac{3}{4} - \upsilon < 0$. Hence, $\sum_{h=1}^l (2^h)^{\alpha-\frac{3}{4} - \upsilon} \leq C' < \infty$, and \eqref{eq:rem_domain2_2} becomes
\begin{equation}
\sum_{l = i+2}^{j-3} 2^{2l} \sum_{h=1}^l (2^h)^{\alpha-2} \pi_4(2^{h+1}, 2^{l-1}) \leq C_5 \sum_{l = i+2}^{j-3} (2^l)^{\frac{3}{4} + \upsilon} C' \leq C_6 (2^j)^{\frac{3}{4} + \upsilon}.
\end{equation}
This implies the following analog of \eqref{eq:fourarm_term1}:
\begin{equation} \label{eq:fourarm_term1_domain2}
(\text{Term } 1) \leq C_7 m^{\frac{3}{4} + \upsilon - \beta} \pi_4(2^i, 2^j),
\end{equation}
which is a $o( \pi_4(2^i, 2^j) )$ as $m \to \infty$, if we choose $\upsilon$ small enough so that $\beta > \frac{3}{4} + \upsilon$.
%Finally, for clear monotonicity reasons, we can assume that $\alpha > 0$, and the rest of the proof follows as for Domain I.

This computation shows that the phenomenology in Domain II is different from Domain I, in the sense that the contribution of pivotal holes is mostly produced by microscopic holes. As the reader can check, exactly the same calculation would appear in the ``further stability results'' below: for one-arm events (see the reasonings after \eqref{eq:sum_pivotal_one_arm}), and for crossing probabilities (see below \eqref{eq:crossing_pf3}). In both \eqref{eq:sum_pivotal_one_arm2} and \eqref{eq:crossing_pf3}, the term $m^{\alpha-\beta}$ would become $m^{\frac{3}{4} + \upsilon - \beta}$, as in \eqref{eq:fourarm_term1_domain2}.

\subsection{General comments on the stability of arm events} \label{sec:stability_other_arm}

In this section, we use the notation $\alpha_j := \alpha_{\sigma}$ for the critical arm exponent in the case when $\sigma = (ovo\ldots) \in \colorseq_j$ is alternating ($j \geq 1$). The four-arm stability result, Theorem \ref{thm:four_arm_stability}, comes from a subtle balance between opposite effects of the holes on the occupied and vacant arms. At its core, the proof relies on the inequality $\alpha_{(oo)} > \alpha_4 - 1$: in the computations below \eqref{eq:near_critical_two_arms} (for the case $\alpha > 1$), and below \eqref{eq:block_holes} (for the general case). This inequality itself comes from $\alpha_{(oo)} > \alpha_2$ (from \eqref{eq:monochromatic}), and the numerical values $\alpha_2 = \frac{1}{4}$ and $\alpha_4 = \frac{5}{4}$ (see the paragraph below \eqref{eq:arm_exponent}). But there does not seem to be any conceptual reason why the four-arm event should be stable.

Also, note that the a-priori bounds available for other lattices, e.g. the square lattice $\ZZ^2$, do not seem to be accurate enough to make our proof of four-arm stability work there. Maybe a more detailed geometric analysis could still provide a proof for those lattices, but this is not clear at the moment (and beyond the main purpose of this paper).

Further stability results will be derived in Section \ref{sec:other_stability}, in particular for one occupied arm (i.e. $\arm_1 = \arm_{(o)}$), see Proposition \ref{prop:one_arm_stability}. In order to illustrate that arm stability is not obvious at all, we now point out that it does not hold for all types of arm events. To make things more concrete, let us assume (in this section only) that $\rho^{(m)} \big( [r,+\infty) \big) = c_1 r^{\alpha-2} e^{-c_2 r / m}$ (for $r \geq 1$) and $\pi^{(m)}_v = c_3 m^{- \beta}$ (for all $v \in V$), for some $\alpha$ and $\beta$ as in \eqref{eq:assump_exp}.

%When we use the event $\overline{\overline{\calH}}$ to take care of the big holes crossing the inner boundary $\partial \Ball_{n_1}(z)$ of the annulus, we use that the sequence $\sigma = (ovov)$ contains at least one occupied arm, which prevents such big holes to cover completely $\Ball_{n_1}(z)$.

First, it is easy to see that the one-arm event for $\sigma = (v)$ (one vacant arm) is \emph{not stable} if $\beta - \alpha < \alpha_1$ ($= \frac{5}{48}$): indeed, we have that for some $\upsilon > 0$ small enough,
$$\PPh_{1/2}^{(m)} \big( \arm_{(v)}( \Ann_{n_1, m}(z) ) \big) \geq \PPh^{(m)} \big( \calH( \Ann_{n_1, m}(z) ) \big) \asymp m^{\alpha - \beta} \gg m^{-\alpha_1 + \upsilon} \gg \PP_{1/2} \big( \arm_{(v)}( \Ann_{n_1, m}(z) ) \big)$$
as $m \to \infty$, if $n_1$ is fixed (or $n_1$ grows at most like $m^{\upsilon'}$, for some sufficiently small $\upsilon' > 0$), using \eqref{eq:arm_exponent}. In fact, this argument shows that for every $\sigma = (v \ldots v) \in \colorseq_j$, the event $\arm_{\sigma}$ is not stable.

Now, we will point out that even sequences containing occupied arms are not necessarily stable. Indeed, let us consider the $j$-arm event with sequence $\sigma = (ov \ldots v) \in \colorseq_j$ (with one occupied arm, and $j-1$ vacant arms). A similar computation as for Lemma \ref{lem:Hbar} yields
\begin{equation} \label{eq:hole_quarter_plane}
\PPh^{(m)} \big( \overline{\calH}^{[1/4]}(\Ann_{n_1, m}) \big) \asymp m^{-\beta} n_1 m^{\alpha-1} \quad \text{as $m \to \infty$,}
\end{equation}
where (for an annulus $A$), $\overline{\calH}^{[1/4]}(A)$ denotes the event that $\overline{\calH}(A)$ is realized by a hole that furthermore stays in the quarter-plane $(0,+\infty)^2$. Using the notation $\arm^{[3/4]}_{(o)}$ for the one-arm event in the complementary three-quarter plane $\RR^2 \setminus (0,+\infty)^2$, we can write
\begin{align}
\PPh^{(m)}_{1/2} \big( \arm_{\sigma}(\Ann_{n_1, m}) \big) & \geq \PPh^{(m)}_{1/2} \big( \arm^{[3/4]}_{(o)}(\Ann_{n_1, m}) \cap \overline{\calH}^{[1/4]}(\Ann_{n_1/2, m}) \big) \nonumber \\
& = \PP_{1/2} \big( \arm^{[3/4]}_{(o)}(\Ann_{n_1, m}) \big) \PPh^{(m)} \big( \overline{\calH}^{[1/4]}(\Ann_{n_1/2, m}) \big). \label{eq:hole_and_arm}
\end{align}
The arm exponent $\alpha^{[3/4]}_1$ corresponding to $\arm^{[3/4]}_{(o)}$ can be obtained from the half-plane one-arm exponent $\alpha^{[1/2]}_1 = \frac{1}{3}$, as $\alpha^{[3/4]}_1= \frac{2}{3} \cdot \alpha^{[1/2]}_1 = \frac{2}{9}$ (by ``conformal invariance''). Hence, combined with \eqref{eq:hole_quarter_plane} and \eqref{eq:hole_and_arm}, we obtain that for any $\upsilon > 0$,
\begin{equation}
\PPh^{(m)}_{1/2} \big( \arm_{\sigma}(\Ann_{n_1, m}) \big) \geq C \bigg( \frac{n_1}{m} \bigg)^{\frac{2}{9} + \upsilon} m^{-\beta} n_1 m^{\alpha-1} = C m^{\alpha - \beta} \bigg( \frac{n_1}{m} \bigg)^{\frac{11}{9} + \upsilon}.
\end{equation}
For all $j \geq 4$, $\alpha_j > \frac{11}{9}$, so for $\beta - \alpha$ small enough, we can find $\upsilon > 0$ so that
$$\PPh^{(m)}_{1/2} \big( \arm_{\sigma}(\Ann_{n_1, m}) \big) \gg \bigg( \frac{n_1}{m} \bigg)^{\alpha_j - \upsilon} \gg \pi_j(n_1,m)$$
as $m \to \infty$ (using \eqref{eq:arm_exponent}), where again $n_1$ is fixed or grows as a small power of $m$. Hence, the $j$-arm event with sequence $\sigma$ is not stable as soon as $j \geq 4$. Note that a similar construction can be made for sequences $\sigma$ containing more than one occupied arm, as long as there are enough vacant arms.

Let us also mention that we expect the six-arm event with sequence $(ovvovv)$ to be of particular importance. This event is a classical a-priori estimate for near-critical percolation, which plays in particular a central role in \cite{KMS2015}. It should be relevant for extending our results in Section \ref{sec:existence_excep_scales} to forest fires \emph{with recovery} (see the discussion in Section \ref{sec:forest_fires}). This event turns out to be stable as well, but proving it requires more careful combinatorics than for Theorem \ref{thm:four_arm_stability}, and we plan to write it out in detail in a separate paper.

\section{Further stability results} \label{sec:other_stability}

In this section, we still suppose that the probability measures $\PPh_p^{(m)}$ (see \eqref{eq:def_P_bar}) satisfy Assumption \ref{ass:rho_pi} and Assumption \ref{ass:alpha_beta}. Recall that $\alpha$, $\beta$, $c_1$, $c_2$ and $c_3$ are parameters appearing in the definition of these measures. In our setting of percolation with holes, we prove several results which extend classical properties of usual Bernoulli percolation. We first use the four-arm stability result Theorem \ref{thm:four_arm_stability} to prove the stability of one-arm events (Section \ref{sec:one_arm_stability}), and of crossing events in rectangles (Section \ref{sec:crossing}). The stability of crossing probabilities is then used in Section \ref{sec:exp_decay} to establish an exponential decay property for these probabilities, similar to \eqref{eq:exp_decay}. Finally, in Section \ref{sec:BCKS} we combine the one-arm stability result and the exponential decay property to obtain estimates for the volume of the largest cluster in a box, analogous to \eqref{eq:largest_cluster}.

\subsection{One-arm event} \label{sec:one_arm_stability}

In this section, we prove stability for the existence of one occupied arm.

\begin{proposition} \label{prop:one_arm_stability}
Let $K \geq 1$. We have
\begin{equation}
\PPh_p^{(m)} \big( \arm_1( \Ann_{n_1,n_2} ) \big) = \PP_p \big( \arm_1( \Ann_{n_1,n_2} ) \big) \cdot (1 + o(1)) \quad \text{as } m \to \infty,
\end{equation}
uniformly in $p \in (0,1)$, and $1 \leq n_1 \leq \frac{n_2}{32} \leq n_2 \leq K (m \wedge L(p))$ (i.e. the $o(1)$ depends on $c_1$, $c_2$, $c_3$, $\alpha$, $\beta$ and $K$, but not on $p$, $n_1$ and $n_2$ satisfying the conditions stated).
\end{proposition}

\begin{proof}[Proof of Proposition \ref{prop:one_arm_stability}]

We first assume that $n_1 \geq m^{\beta/3}$. Let $\eta \in (0, \frac{1}{8})$. We consider the annuli $A := \Ann_{n_1,n_2}$, $A' := \Ann_{(1-\eta) n_1, (1+\eta) n_2}$, and $A'' := \Ann_{(1-2\eta) n_1, (1+2\eta) n_2}$. We prove that
\begin{equation} \label{eq:one_arm_stability_suff}
\PPh_p^{(m)} \big( \arm_1( A ) \big) \geq \PP_p \big( \arm_1( A'' ) \big) \cdot (1 + o(1)) \quad \text{as } m \to \infty,
\end{equation}
which is enough to establish Proposition \ref{prop:one_arm_stability}. Indeed, it follows from standard arguments for ordinary Bernoulli percolation (from the fact that the critical exponent for three arms in a half plane is equal to $2$, so in particular strictly larger than $1$, see for example Theorem 24 in \cite{No2008}) that the ratio of $\PP_p \big( \arm_1( A'' ) \big)$ and $\PP_p \big( \arm_1( A ) \big)$ can be made arbitrarily close to $1$ by choosing $\eta > 0$ small enough, uniformly in $p \in (0,1)$ and $m^{\beta/3} \leq n_1 \leq \frac{n_2}{32} \leq n_2 \leq K L(p)$ (for $m$ large enough).

Because of boundary effects, we first ``add'' (in a similar sense as in Section \ref{sec:case_alpha_1}) the holes with centers in $A'$ (i.e. at a sufficient distance from the boundary of $A''$), and then the remaining holes, with centers in $(A')^c$. For that, we introduce the intermediate families $\pi'^{(m)}_v := \pi^{(m)}_v \ind_{v \in A'}$ and $\pi''^{(m)}_v := \pi^{(m)}_v \ind_{v \in (A')^c}$ (so that $\pi^{(m)} = \pi'^{(m)} + \pi''^{(m)}$). If we denote by $\tilde{\arm}_1( A'' )$ the event that $\arm_1( A'' )$ holds without the holes, we have $\arm_1( A'' ) \subseteq \tilde{\arm}_1( A'' )$ so
\begin{align}
\PPh_p^{\pi'^{(m)},\rho^{(m)}} \big( \arm_1( A'' ) \big) & = \PPh_p^{\pi'^{(m)},\rho^{(m)}} \big( \tilde{\arm}_1( A'' ) \big) - \PPh_p^{\pi'^{(m)},\rho^{(m)}} \big( \tilde{\arm}_1( A'' ) \setminus \arm_1( A'' ) \big), \nonumber \\
& = \PP_p \big( \arm_1( A'' ) \big) - \PPh_p^{\pi'^{(m)},\rho^{(m)}} \big( \tilde{\arm}_1( A'' ) \setminus \arm_1( A'' ) \big). \label{eq:one_arm_stability_pf1}
\end{align}
We claim that
\begin{equation} \label{eq:one_arm_stability_pf2}
\PPh_p^{\pi'^{(m)},\rho^{(m)}} \big( \tilde{\arm}_1( A'' ) \setminus \arm_1( A'' ) \big) = \PP_p \big( \arm_1( A'' ) \big) \cdot o(1),
\end{equation}
which we now prove.

\begin{figure}
\begin{center}

\includegraphics[width=.9\textwidth]{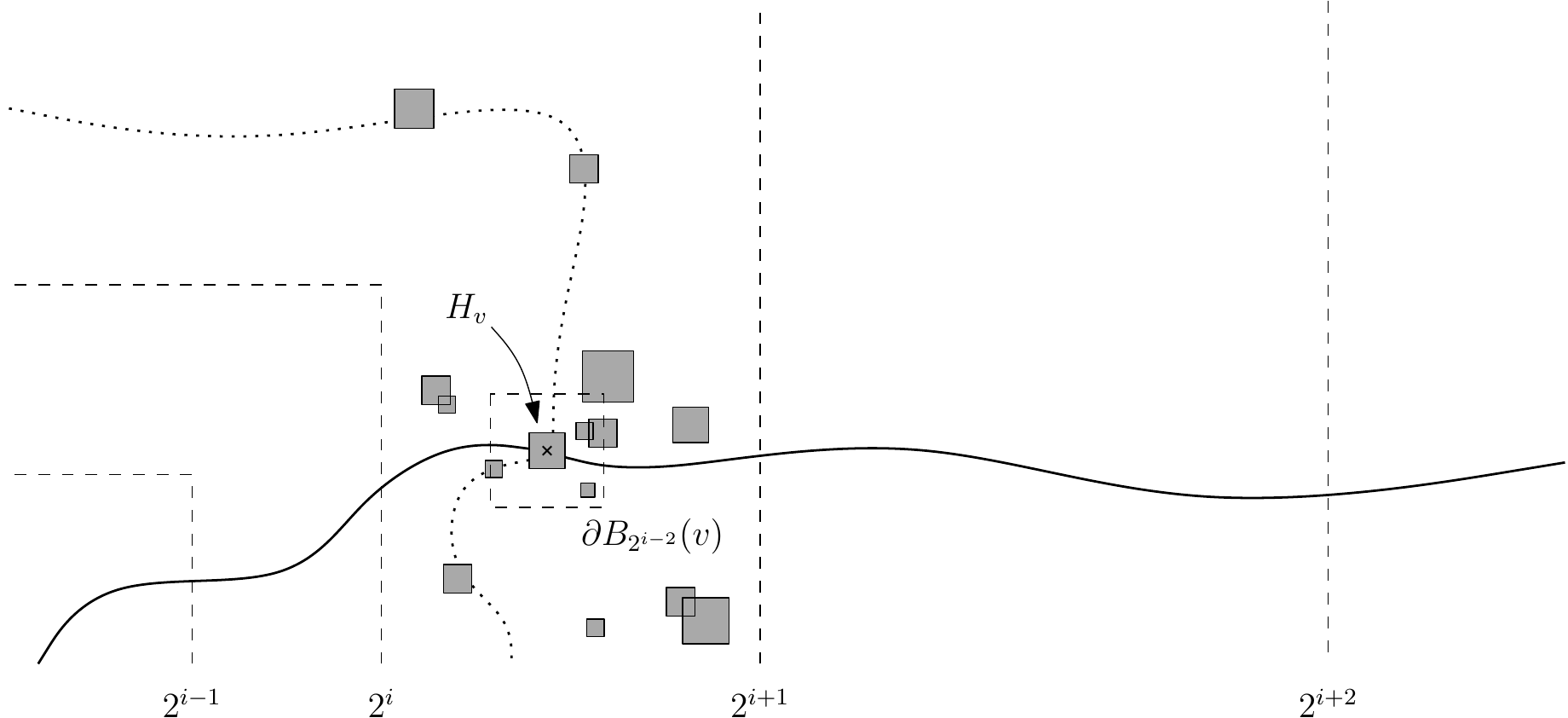}
\caption{\label{fig:one_arm_stability} If $H_v$ is a ``pivotal'' hole, with center $v \in \Ann_{2^i, 2^{i+1}}$ and radius $r_v \in [2^j, 2^{j+1})$ ($j \leq i-4$), we consider the three events $\calW_4( \Ann_{2^{j+1}, 2^{i-2}}(v) )$, $\tilde{\arm}_1( \Ann_{(1-2\eta) n_1, 2^{i-1}} )$ and $\tilde{\arm}_1( \Ann_{2^{i+2}, (1+2\eta) n_2} )$.}

\end{center}
\end{figure}

We follow a similar procedure as for Theorem \ref{thm:four_arm_stability}. By adding the holes with centers in $A'$ one by one, until the one-arm event fails, we see that there must exist a ``pivotal'' hole $H_v$, with $v \in A'$. Let $i \geq 0$ be such that $v \in \Ann_{2^i, 2^{i+1}}$. Clearly, either $r_v \geq 2^{i-3}$, or $2^j \leq r_v < 2^{j+1}$ for some $j \leq i-4$. In the latter case, the event $\calW_4( \Ann_{2^{j+1}, 2^{i-2}}(v) )$ occurs (see Figure \ref{fig:one_arm_stability}). We deduce, with $I := \lfloor \log_2 ( n_2 ) \rfloor$,
\begin{align}
& \PPh_p^{\pi'^{(m)},\rho^{(m)}} \big( \tilde{\arm}_1( A'' ) \setminus \arm_1( A'' ) \big) \leq c_3 m^{-\beta} \sum_{i \leq I} \big| \Ball_{2^{i+1}} \big| \cdot \bigg[ \rho^{(m)} \big( [2^{i-3}, +\infty) \big) \cdot \PPh_p^{\pi'^{(m)},\rho^{(m)}} \big( \tilde{\arm}_1( A'' ) \big) \nonumber \\
& \hspace{0.6cm} + \sum_{j=0}^{i-4} \Big[ \rho^{(m)} \big( [2^j, +\infty) \big) \cdot \PPh_p^{\pi'^{(m)},\rho^{(m)}} \big( \calW_4( \Ann_{2^{j+1}, 2^{i-2}}(v) ) \big) \cdot \PPh_p^{\pi'^{(m)},\rho^{(m)}} \big( \arm^{(1)} \big) \cdot \PPh_p^{\pi'^{(m)},\rho^{(m)}} \big( \arm^{(2)} \big) \Big] \bigg], \label{eq:one_arm_stability_pf2a}
\end{align}
where $\arm^{(1)} := \tilde{\arm}_1( \Ann_{(1-2\eta) n_1, 2^{i-1}} )$ and $\arm^{(2)} := \tilde{\arm}_1( \Ann_{2^{i+2}, (1+2\eta) n_2} )$ (note that the three events above $\calW_4( \Ann_{2^{j+1}, 2^{i-2}}(v) )$, $\arm^{(1)}$ and $\arm^{(2)}$ are independent, since they involve disjoint regions of the plane, and only $\calW_4( \Ann_{2^{j+1}, 2^{i-2}}(v) )$ involves the holes). We know from Theorem \ref{thm:four_arm_stability} that
\begin{equation} \label{eq:one_arm_stability_pf2b}
\PPh_p^{\pi'^{(m)},\rho^{(m)}} \big( \calW_4( \Ann_{2^{j+1}, 2^{i-2}}(v) ) \big) \leq C_1 \pi_4(2^{j+1}, 2^{i-2}) \leq C_2 \pi_4(2^j, 2^i)
\end{equation}
(the second inequality follows from \eqref{eq:extendability}). We also have $\PPh_p^{\pi'^{(m)},\rho^{(m)}} \big( \tilde{\arm}_1( A'' ) \big) = \PP_p \big( \arm_1( A'' ) \big)$, and
\begin{align}
\PPh_p^{\pi'^{(m)},\rho^{(m)}} \big( \arm^{(1)} \big) \cdot \PPh_p^{\pi'^{(m)},\rho^{(m)}} \big( \arm^{(2)} \big) & = \PP_p \big( \arm_1( \Ann_{(1-2\eta) n_1, 2^{i-1}} ) \big) \cdot \PP_p \big( \arm_1( \Ann_{2^{i+2}, (1+2\eta) n_2} ) \big) \nonumber \\
& \leq C_3 \PP_p \big( \arm_1( A'' ) \big). \label{eq:one_arm_stability_pf2c}
\end{align}
Hence, by combining \eqref{eq:one_arm_stability_pf2a}, \eqref{eq:one_arm_stability_pf2b} and \eqref{eq:one_arm_stability_pf2c}, and using \eqref{eq:assump_holes}, we obtain
\begin{equation} \label{eq:one_arm_stability_pf2d}
\PPh_p^{\pi'^{(m)},\rho^{(m)}} \big( \tilde{\arm}_1( A'' ) \setminus \arm_1( A'' ) \big) \leq C_4 m^{-\beta} \sum_{i \leq I} 2^{2i} \cdot \bigg[ (2^i)^{\alpha-2} + \sum_{j=0}^{i-4} \Big[ (2^j)^{\alpha-2} \cdot \pi_4(2^j, 2^i) \Big] \bigg] \cdot \PP_p \big( \arm_1( A'' ) \big).
\end{equation}
Lemma \ref{lem:ratio_limit} implies that for any $\upsilon > 0$ fixed,
\begin{equation} \label{eq:sum_pivotal_one_arm}
\sum_{j=0}^{i-4} \Big[ (2^j)^{\alpha-2} \cdot \pi_4(2^j, 2^i) \Big] \leq C_5 \sum_{j=0}^{i-4} (2^j)^{\alpha-2} \bigg( \frac{2^j}{2^i} \bigg)^{\frac{5}{4} - \upsilon} = C_5 (2^i)^{- \frac{5}{4} + \upsilon} \sum_{j=0}^{i-4} (2^j)^{\alpha-2 + \frac{5}{4} - \upsilon}.
\end{equation}
By Assumption \ref{ass:alpha_beta}, we have $\alpha-2 > - \frac{5}{4}$ (see \eqref{eq:assump_exp}), so we can choose $\upsilon$ small enough so that $\alpha-2 + \frac{5}{4} - \upsilon > 0$, and we deduce
\begin{equation} \label{eq:one_arm_stability_pf2e}
\sum_{j=0}^{i-4} \Big[ (2^j)^{\alpha-2} \cdot \pi_4(2^j, 2^i) \Big] \leq C_6 (2^i)^{- \frac{5}{4} + \upsilon} (2^i)^{\alpha-2 + \frac{5}{4} - \upsilon} = C_6 (2^i)^{\alpha-2}.
\end{equation}
It then follows from \eqref{eq:one_arm_stability_pf2d} and \eqref{eq:one_arm_stability_pf2e} that
\begin{equation} \label{eq:one_arm_stability_pf2f}
\PPh_p^{\pi'^{(m)},\rho^{(m)}} \big( \tilde{\arm}_1( A'' ) \setminus \arm_1( A'' ) \big) \leq C_4 \PP_p \big( \arm_1( A'' ) \big) \cdot m^{-\beta} \sum_{i \leq I} 2^{2i} \cdot (C_6+1) (2^i)^{\alpha-2}.
\end{equation}
Since
\begin{equation}
\sum_{i \leq I} 2^{2i} \cdot (2^i)^{\alpha-2} = \sum_{i \leq I} (2^i)^{\alpha} \leq C_7 (n_2)^{\alpha} \leq C_7 K^{\alpha} m^{\alpha}
\end{equation}
(where we used $\alpha > 0$), we finally obtain
\begin{equation} \label{eq:sum_pivotal_one_arm2}
\PPh_p^{\pi'^{(m)},\rho^{(m)}} \big( \tilde{\arm}_1( A'' ) \setminus \arm_1( A'' ) \big) \leq C_8 \PP_p \big( \arm_1( A'' ) \big) \cdot m^{\alpha - \beta},
\end{equation}
which (since $\beta > \alpha$) establishes \eqref{eq:one_arm_stability_pf2}.

By using monotonicity, and then combining \eqref{eq:one_arm_stability_pf1} and \eqref{eq:one_arm_stability_pf2}, we obtain
\begin{equation} \label{eq:one_arm_stability_pf3}
\PPh_p^{\pi'^{(m)},\rho^{(m)}} \big( \arm_1( A ) \big) \geq \PPh_p^{\pi'^{(m)},\rho^{(m)}} \big( \arm_1( A'' ) \big) = \PP_p \big( \arm_1( A'' ) \big) \cdot (1 + o(1)).
\end{equation}
We then add the holes with centers in $(A')^c$, and use Lemma \ref{lem:no_crossing}. We can write
\begin{align}
\PPh_p^{\pi^{(m)},\rho^{(m)}} \big( \arm_1( A ) \big) & = \PPh_p^{\pi^{(m)},\rho^{(m)}} \big( \tilde{\tilde{\arm}}_1( A ) \big) - \PPh_p^{\pi^{(m)},\rho^{(m)}} \big( \tilde{\tilde{\arm}}_1( A ) \setminus \arm_1( A ) \big) \nonumber \\
& = \PPh_p^{\pi'^{(m)},\rho^{(m)}} \big( \arm_1( A ) \big) - \PPh_p^{\pi^{(m)},\rho^{(m)}} \big( \tilde{\tilde{\arm}}_1( A ) \setminus \arm_1( A ) \big), \label{eq:one_arm_stability_pf4}
\end{align}
where we denote by $\tilde{\tilde{\arm}}_1( A )$ the event that $\arm_1( A )$ holds without the holes in $(A')^c$. Let $\calH := \calH( \Ann_{(1-\eta) n_1, n_1} ) \cup \calH( \Ann_{n_2, (1+\eta) n_2} )$. We have
\begin{equation} \label{eq:one_arm_stability_pf5}
\PPh_p^{\pi^{(m)},\rho^{(m)}} \big( \tilde{\tilde{\arm}}_1( A ) \setminus \arm_1( A ) \big) \leq \PPh^{\pi''^{(m)},\rho^{(m)}} \big( \calH \big) \cdot \PPh_p^{\pi'^{(m)},\rho^{(m)}} \big( \arm_1( A ) \big) \leq \frac{2 C}{m^{\beta-\alpha}} \cdot \PPh_p^{\pi'^{(m)},\rho^{(m)}} \big( \arm_1( A ) \big)
\end{equation}
(using Lemma \ref{lem:no_crossing}). By combining \eqref{eq:one_arm_stability_pf4} and \eqref{eq:one_arm_stability_pf5}, we obtain
\begin{equation} \label{eq:one_arm_stability_pf6}
\PPh_p^{\pi^{(m)},\rho^{(m)}} \big( \arm_1( A ) \big) = \PPh_p^{\pi'^{(m)},\rho^{(m)}} \big( \arm_1( A ) \big) \cdot (1 + o(1)).
\end{equation}
The desired result \eqref{eq:one_arm_stability_suff} then follows immediately from \eqref{eq:one_arm_stability_pf3} and \eqref{eq:one_arm_stability_pf6}.

In the case when $n_1 < m^{\beta/3}$, we proceed in a similar way, but we handle separately the holes with centers close to $\din \Ball_{n_1}$. For that, we start by adding the holes centered in $\Ball_{3 m^{\beta/3}}$: with probability $1 - O(m^{-\beta/3})$, there are no such holes (using \eqref{eq:assump_holes}). If $n_2 < 2 m^{\beta/3}$, we can conclude immediately by using Lemma \ref{lem:no_crossing} that with probability $1 - O(m^{-\beta/3}) - O(m^{\alpha - \beta})$, no hole intersects $\Ann_{n_1,n_2}$. Otherwise, the remainder of the proof is the same as in the case $n_1 \geq m^{\beta/3}$.
\end{proof}

\subsection{Box crossing probabilities} \label{sec:crossing}

In this section, we establish stability for certain box crossing events.

\begin{proposition} \label{prop:crossing}
Let $K \geq 1$. We have
\begin{equation}
\PPh_p^{(m)} \big( \Ch([0,2n] \times [0,n]) \big) = \PP_p \big( \Ch([0,2n] \times [0,n]) \big) + o(1) \quad \text{as } m \to \infty,
\end{equation}
uniformly in $p \in (0,1)$, and $1 \leq n \leq K (m \wedge L(p))$ (i.e. the $o(1)$ depends on $c_1$, $c_2$, $c_3$, $\alpha$, $\beta$ and $K$, but not on $p$ and $n$).
\end{proposition}

\begin{proof}[Proof of Proposition \ref{prop:crossing}]

First, note that we can assume $n \geq m^{\beta/3}$: otherwise, it follows from \eqref{eq:assump_holes} and Lemma \ref{lem:no_crossing} that with probability $1 - O(m^{-\beta/3}) - O(m^{\alpha - \beta})$, no hole intersects $[0,2n] \times [0,n]$.

We are interested in horizontal crossings of the rectangle $R := [0,2n] \times [0,n]$. Let $\eta \in (0, \frac{1}{8})$, and consider the auxiliary rectangles $R' := [-2 \eta n, (2+2\eta)n] \times [2 \eta n, (1-2\eta)n]$ and $R'' := [-2 \eta n, (2+2\eta)n] \times [0, n]$ (see Figure \ref{fig:crossing_proba}). In order to take care of boundary effects, we add successively the holes centered in the following three regions, forming a partition of $V$:
\begin{itemize}
\item $V_I := \big( (-3 \eta n, (2 + 3 \eta) n) \times (\eta n, (1 - \eta) n) \big)^c \cap V$,

\item $V_{II} := \big( \big( [- \eta n, (2 + \eta) n] \times [\eta n, (1 - \eta) n] \big) \cap V \big) \setminus V_I$,

\item and $V_{III} := \big( \big( ([- 3 \eta n, - \eta n] \cup [(2 + \eta) n, (2 + 3 \eta) n]) \times [\eta n, (1 - \eta) n] \big) \cap V \big) \setminus (V_I \cup V_{II})$.
\end{itemize}

\begin{figure}[t]
\begin{center}

\includegraphics[width=.95\textwidth]{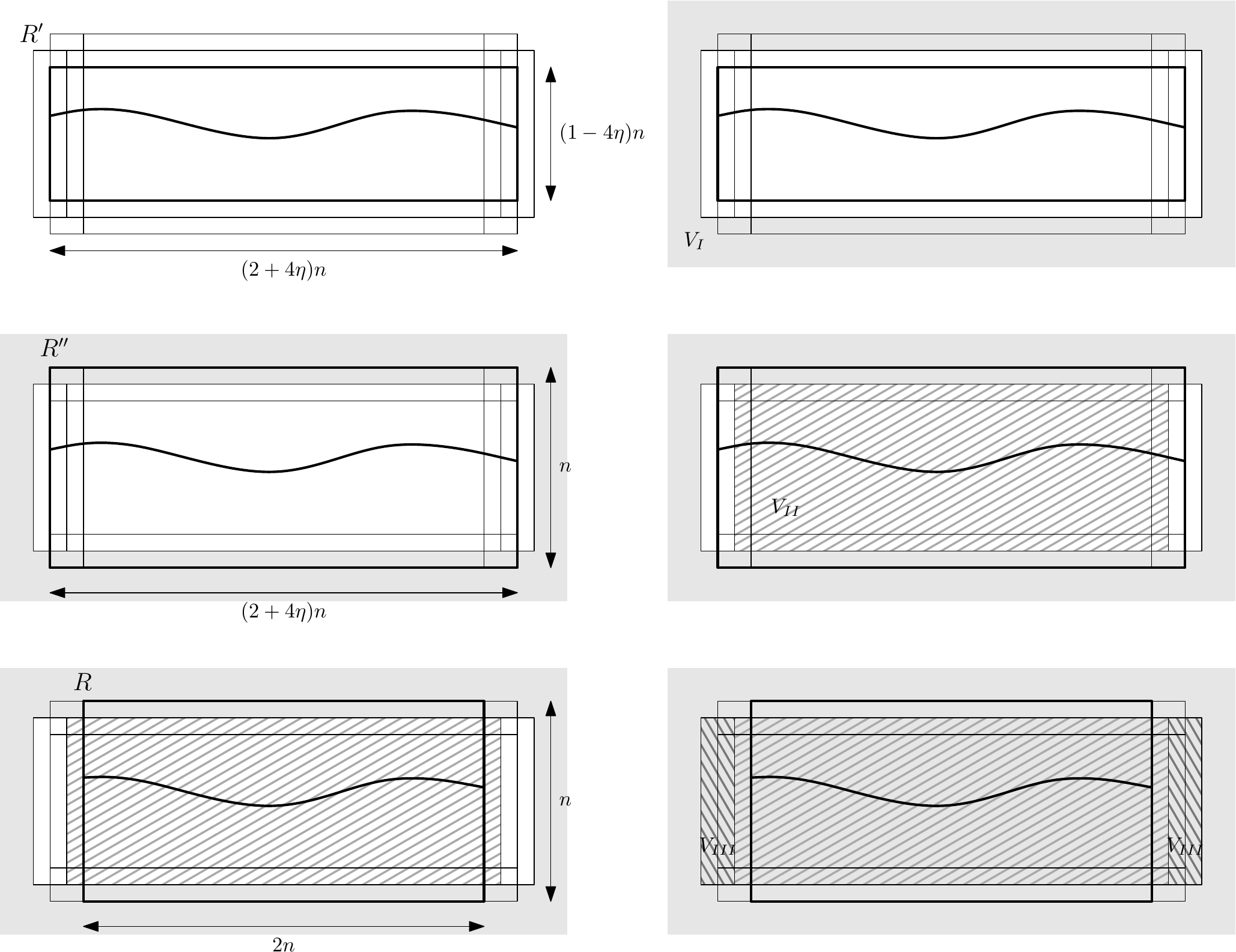}
\caption{\label{fig:crossing_proba} In order to take care of the boundary effects, we add the holes in three successive steps. We denote by $V_I$, $V_{II}$ and $V_{III}$ the corresponding subsets of vertices.}

\end{center}
\end{figure}

We thus introduce $\pi'^{(m)}$ and $\pi''^{(m)}$, defined by
$$\pi'^{(m)}_v := \pi^{(m)}_v \ind_{v \in V_I} \quad \text{and} \quad \pi''^{(m)}_v := \pi^{(m)}_v \ind_{v \in V_I \cup V_{II}} \quad (v \in V).$$
First, it follows from a similar computation as in the proof of Lemma \ref{lem:no_crossing} that
$$\PPh_p^{(m)} \big(  \exists v \in V_I \: : \: H_v \cap R' \neq \emptyset \big) = O(m^{\alpha - \beta}) \quad \text{as } m \to \infty,$$
uniformly in $n$ and $p$ with the required properties (for a fixed $\eta > 0$). Hence,
\begin{equation} \label{eq:crossing_pf1}
\PPh_p^{\pi'^{(m)},\rho^{(m)}} \big( \Ch(R') \big) = \PP_p \big( \Ch(R') \big) + O(m^{\alpha - \beta}).
\end{equation}
By monotonicity, we have
\begin{equation} \label{eq:crossing_pf2}
\PPh_p^{\pi'^{(m)},\rho^{(m)}} \big( \Ch(R'') \big) \geq \PPh_p^{\pi'^{(m)},\rho^{(m)}} \big( \Ch(R') \big).
\end{equation}
We now add the holes with centers in the ``middle'' of $R''$, i.e. at a distance at least $\eta n$ from the boundary of $R''$. This is the region that we denote by $V_{II}$, and we claim that
\begin{equation} \label{eq:crossing_pf3}
\PPh_p^{\pi''^{(m)},\rho^{(m)}} \big( \Ch(R'') \big) = \PPh_p^{\pi'^{(m)},\rho^{(m)}} \big( \Ch(R'') \big) + O(m^{\alpha - \beta}).
\end{equation}
Indeed, this follows from a similar reasoning as for Proposition \ref{prop:one_arm_stability}: by adding the holes with centers in $V_{II}$ one by one, until the crossing event fails, we see that there must be a ``pivotal'' hole $H_v$ ($v \in V_{II}$) from which four arms originate to the four sides of $R''$. For $J := \lfloor \log_2 ( \eta n ) \rfloor$, we obtain, by distinguishing whether $r_v \geq 2^{J-1}$, or $2^j \leq r_v < 2^{j+1}$ for some $j \leq J-2$,
\begin{align*}
\PPh_p^{\pi'^{(m)},\rho^{(m)}} & \big( \Ch(R'') \big) - \PPh_p^{\pi''^{(m)},\rho^{(m)}} \big( \Ch(R'') \big)\\
& \leq c_3 m^{-\beta} \sum_{v \in V_{II}} \bigg[ \rho^{(m)} \big( [2^{J-1}, +\infty) \big) + \sum_{j=0}^{J-2} \rho^{(m)} \big( [2^j, +\infty) \big) \cdot \PPh_p^{\pi''^{(m)},\rho^{(m)}} \big( \calW_4( \Ann_{2^{j+1}, 2^J}(v) ) \big) \bigg]
\end{align*}
(similarly to \eqref{eq:one_arm_stability_pf2a}). Using $\PPh_p^{\pi''^{(m)},\rho^{(m)}} \big( \calW_4( \Ann_{2^{j+1}, 2^J}(v) ) \big) \leq C_1 \pi_4(2^j, 2^J)$ (from Theorem \ref{thm:four_arm_stability} and \eqref{eq:extendability}) and \eqref{eq:assump_holes}, we obtain
\begin{align*}
\PPh_p^{\pi'^{(m)},\rho^{(m)}} \big( \Ch(R'') \big) & - \PPh_p^{\pi''^{(m)},\rho^{(m)}} \big( \Ch(R'') \big)\\
& \leq c_3 m^{-\beta} \big| V_{II} \big| \cdot \bigg[ c_1 (2^{J-1})^{\alpha-2} + \sum_{j=0}^{J-2} c_1 (2^j)^{\alpha-2} \cdot C_1 \pi_4(2^j, 2^J) \bigg]\\
& \leq C_2 m^{-\beta} n^2 \cdot n^{\alpha-2}
\end{align*}
(by a summation argument similar to \eqref{eq:sum_pivotal_one_arm}, \eqref{eq:one_arm_stability_pf2e}), which establishes the claim \eqref{eq:crossing_pf3} (since $n \leq Km$).
%Note that here, it is crucial that we have at our disposal the stronger version of four-arm stability, where we are allowed to ``choose which holes to keep''.

Using again monotonicity,
\begin{equation} \label{eq:crossing_pf4}
\PPh_p^{\pi''^{(m)},\rho^{(m)}} \big( \Ch(R) \big) \geq \PPh_p^{\pi''^{(m)},\rho^{(m)}} \big( \Ch(R'') \big).
\end{equation}
Finally, we add the holes with centers in $V_{III}$: a similar computation as for \eqref{eq:crossing_pf1} yields
\begin{equation} \label{eq:crossing_pf5}
\PPh_p^{\pi^{(m)},\rho^{(m)}} \big( \Ch(R) \big) = \PPh_p^{\pi''^{(m)},\rho^{(m)}} \big( \Ch(R) \big) + O(m^{\alpha - \beta}).
\end{equation}
Combining \eqref{eq:crossing_pf1}, \eqref{eq:crossing_pf2}, \eqref{eq:crossing_pf3}, \eqref{eq:crossing_pf4} and \eqref{eq:crossing_pf5}, we obtain
\begin{equation} \label{eq:crossing_pf6}
\PPh_p^{(m)} \big( \Ch(R) \big) \geq \PP_p \big( \Ch(R') \big) + o(1).
\end{equation}
This allows us to conclude, since we can make $\PP_p \big( \Ch(R') \big)$ as close as we want to $\PP_p \big( \Ch(R) \big)$ by choosing $\eta > 0$ small enough, uniformly in $p \in (0,1)$ and $m^{\beta/3} \leq n \leq K L(p)$, for $m$ large enough (using similar standard arguments as those mentioned below \eqref{eq:one_arm_stability_suff}, involving three-arm events in half planes).
\end{proof}

\subsection{Exponential decay property} \label{sec:exp_decay}

We now establish a (stretched) exponential convergence to $1$ for the probability under $\PPh_p^{(m)}$ of crossing a rectangle in the supercritical regime $p > p_c$, using the stability result, Proposition \ref{prop:crossing}, for these probabilities. Obviously, we can only hope for such a property on scales above $L(p)$ (so that the supercritical behavior emerges in the underlying Bernoulli percolation process). However, note that we also need the rectangles crossed to be of size at least $m$. Indeed, on scales below $m$, the probability to observe a crossing hole (which would block occupied crossings) is only polynomially small in $m$, so not decaying fast enough.

\begin{proposition} \label{prop:exp_decay}
Let $K \geq 1$ and $\gamma \in (0,1)$. There exist $\lambda_1, \lambda_2 > 0$ (depending on $c_1$, $c_2$, $c_3$, $\alpha$, $\beta$, $K$ and $\gamma$) such that for all $m$ sufficiently large, we have: for all $n \geq 1$, and all $p > p_c$ with $L(p) \leq K m$,
\begin{equation} \label{eq:stretched_exp_decay}
\PPh_p^{(m)} \big( \Ch([0,2n] \times [0,n]) \big) \geq 1 - \lambda_1 e^{- \lambda_2 ( n / m )^{\gamma}}.
\end{equation}
\end{proposition}

\begin{proof}[Proof of Proposition \ref{prop:exp_decay}]
In the proof, we adapt a standard block argument for the analogous result in Bernoulli percolation. Adaptations are needed to control the effect of large holes, disturbing the spatial independence (this is also the reason why we do not obtain \eqref{eq:stretched_exp_decay} for $\gamma=1$). We describe in detail which modifications are made, up to a point from which the proposition can be obtained from fairly straightforward computations.

For some given $p$ and $m$ as in the statement, let us denote $f(n) := \PPh_p^{(m)} \big( \Cv^*([0,2n] \times [0,n]) \big)$ ($n \geq 1$). We fix $\eta = \eta(\gamma) \in \big( 0, \frac{1}{4} \big]$ small enough so that
\begin{equation} \label{eq:choice_gamma}
\frac{\log 2}{\log 2 + \log(1+\eta)} \geq \gamma.
\end{equation}
Let us consider, for some $n \geq 1$, the construction depicted in Figure \ref{fig:exp_decay}: the two $5n$ by $n$ rectangles $R_n^+$ and $R_n^-$, and the ``fattened'' open rectangles $\overline{R}_n^+$ and $\overline{R}_n^-$, with side lengths $(5 + 2\eta) n$ and $(1 + 2 \eta) n$. We also denote by $\tilde{R}_n$ the rectangle obtained as the convex hull of $R_n^+$ and $R_n^-$, which has side lengths $5n$ and $2 (1 + \eta) n$.

\begin{figure}
\begin{center}

\includegraphics[width=.8\textwidth]{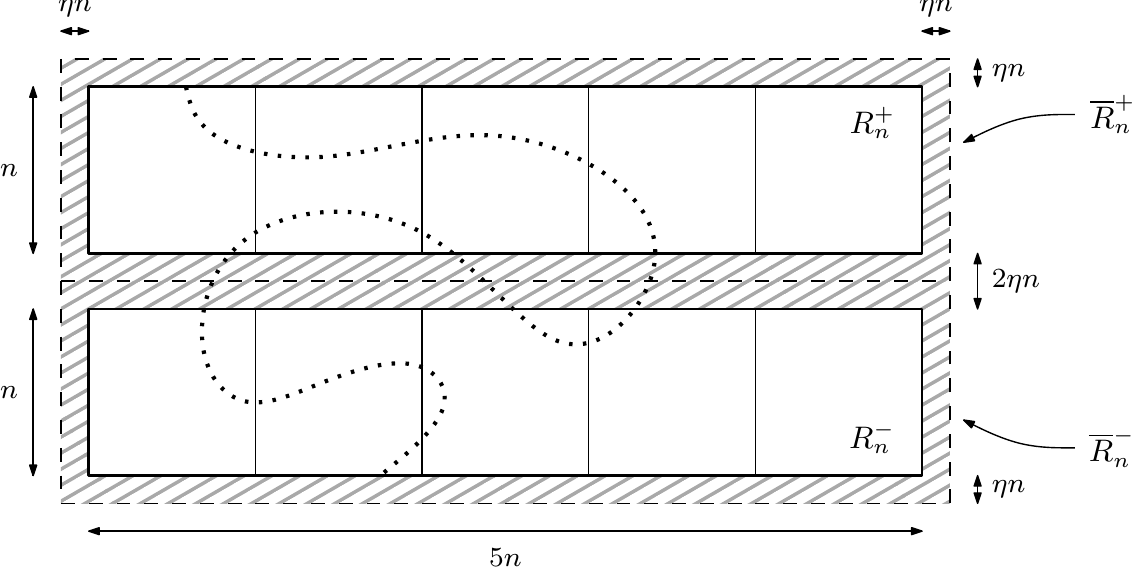}
\caption{\label{fig:exp_decay} The block argument on which the proof of Proposition \ref{prop:exp_decay} relies. We use ``safety strips'' of width $\eta n$ around the rectangles $R_n^+$ and $R_n^-$.}

\end{center}
\end{figure}

We now derive an upper bound on $f(2(1+\eta) n)$ in terms of $f(n)^2$. Here, extra care is needed (compared with Bernoulli percolation), due to the potential existence of holes overlapping both the upper and the lower rectangles $R_n^+$ and $R_n^-$ (and thus helping vacant crossings in both). For that, we use the ``safety strips'' around $R_n^+$ and $R_n^-$. Recall the definition \eqref{eq:def_H} of $\calH(.)$, and the notational remark a few lines below it. The same computation as for Lemma \ref{lem:no_crossing} yields that for some $C$, $C'$ (depending on $c_1$, $c_2$, $c_3$, $\alpha$, $\beta$, and also on $\eta$),
\begin{equation} \label{eq:upper_bound_H_R}
\PPh^{(m)} \big( \calH(\overline{R}_n^+ \setminus R_n^+) \big) \leq \frac{C}{m^{\beta - \alpha}} e^{-C' n/m},
\end{equation}
and similarly for $\PPh^{(m)} \big( \calH(\overline{R}_n^- \setminus R_n^-) \big)$. Let $\tilde{\Cv}^*(R_n^+)$ (resp. $\tilde{\Cv}^*(R_n^-)$) denote the event that $\Cv^*(R_n^+)$ (resp. $\Cv^*(R_n^-)$) occurs without the holes centered in $(\overline{R}_n^+)^c$ (resp. $(\overline{R}_n^-)^c$). Clearly, $\tilde{\Cv}^*(R_n^+)$ and $\tilde{\Cv}^*(R_n^-)$ are independent, and contained in $\Cv^*(R_n^+)$ and $\Cv^*(R_n^-)$ respectively. Hence, also using \eqref{eq:upper_bound_H_R} and $\beta > \alpha$, we obtain
\begin{align}
\PPh_p^{(m)} \big( \Cv^*(\tilde{R}_n) \big) & \leq \PPh^{(m)} \big( \calH(\overline{R}_n^+ \setminus R_n^+) \big) + \PPh^{(m)} \big( \calH(\overline{R}_n^- \setminus R_n^-) \big) + \PPh_p^{(m)} \big( \tilde{\Cv}^*(R_n^+) \big) \PPh_p^{(m)} \big( \tilde{\Cv}^*(R_n^-) \big) \nonumber \\
& \leq 2 C e^{-C' n/m} + \PPh_p^{(m)} \big( \Cv^*(R_n^+) \big) \PPh_p^{(m)} \big( \Cv^*(R_n^-) \big). \label{eq:block_argument1}
\end{align}
We then observe that if $\Cv^*(R_n^+)$ occurs, then at least one of four specified ``horizontal'' $2n$ by $n$ rectangles (see Figure \ref{fig:exp_decay}) has a vertical vacant crossing, or at least one of the three $n$ by $n$ squares located in the ``middle'' of $R_n^+$ has a horizontal vacant crossing. We deduce
\begin{equation} \label{eq:block_argument2}
\PPh_p^{(m)} \big( \Cv^*(R_n^+) \big) \leq 4 \PPh_p^{(m)} \big( \Cv^*([0,2n] \times [0,n]) \big) + 3 \PPh_p^{(m)} \big( \Ch^*([0,n]^2) \big) \leq 7 f(n),
\end{equation}
and similarly for $\PPh_p^{(m)} \big( \Cv^*(R_n^-) \big)$. Combined with \eqref{eq:block_argument1}, this implies (note that $4 (1+\eta) n \leq 5 n$, from our assumption $\eta \leq \frac{1}{4}$)
\begin{equation} \label{eq:block_argument3}
f(2 (1+\eta) n) \leq \PPh_p^{(m)} \big( \Cv^*(\tilde{R}_n) \big) \leq 2 C e^{-C' n/m} + C'' f(n)^2,
\end{equation}
where $C'' = 7^2$.

Note that the derivation above is not completely valid, since, strictly speaking, the crossing events are not translation invariant (the rectangles considered do not ``fit'' the lattice $\TT$). However, this issue can easily be solved by considering the maximum over all translated rectangles $z + [0,2n] \times [0,n]$, $z \in \CC$, in the definition of $f(n)$, and adapting the subsequent arguments accordingly.

We now use \eqref{eq:block_argument3} iteratively, starting from $n_0 = K_0 m$, where $K_0$ is chosen sufficiently large so that
\begin{equation} \label{eq:cond_crossing_K0}
f(K_0 m) \leq \frac{1}{4C''}
\end{equation}
for $m$ large enough, uniformly in $p$ as in the statement (i.e. for all $m \geq m_0 = m_0(c_1, c_2, c_3, \alpha, \beta, K)$, and all $p > p_c$ with $L(p) \leq K m$). Such a $K_0$ exists, from \eqref{eq:exp_decay} and Proposition \ref{prop:crossing}. We can also assume that $K_0$ is large enough so that
\begin{equation} \label{eq:cond_K0}
2 C 2^{-C' \eta K_0 / \log 2} \leq \frac{1}{4 C''} \quad \text{and} \quad \frac{C'}{2 \log 2} K_0 \geq 1
\end{equation}
(recall that $\eta$ is fixed, and depends only on the choice of $\gamma$). Let $\lambda := 2 (1+\eta) \in \big( 2, \frac{5}{2} \big]$. We claim that
\begin{equation} \label{exp_decay_induction}
\text{for all $k \geq 0$}, \quad f(\lambda^k n_0) \leq \frac{1}{2 C''} 2^{-2^k}.
\end{equation}
This can be proved by induction (note that the case $k = 0$ corresponds to \eqref{eq:cond_crossing_K0}), and we omit the details.
%Assuming that the result holds for some $k \geq 0$, \eqref{eq:block_argument3} (with $n = \lambda^k n_0$) implies
%\begin{equation} \label{eq:iterate_block1}
%f(\lambda^{k+1} n_0) \leq 2 C e^{-C' \lambda^k K_0} + C'' f(\lambda^k n_0)^2 \leq 2 C e^{-C' \lambda^k K_0} + C'' \bigg( \frac{1}{2 C''} 2^{-2^k} \bigg)^2.
%\end{equation}
%We can write, using $\lambda^k = 2^k (1 + \eta)^k \geq 2^k + \eta$,
%\begin{equation} \label{eq:iterate_block2}
%2 C e^{-C' \lambda^k K_0} = 2 C 2^{-C' \lambda^k K_0 / \log 2} \leq 2 C 2^{-C' \eta K_0 / \log 2} 2^{-C' 2^k K_0 / \log 2} \leq \frac{1}{4 C''} 2^{-2^{k+1}}
%\end{equation}
%(where the second inequality follows from the two conditions on $K_0$ in \eqref{eq:cond_K0}). Combining \eqref{eq:iterate_block1} and \eqref{eq:iterate_block2}, we obtain
%\begin{equation}
%f(\lambda^{k+1} n_0) \leq \frac{1}{4 C''} 2^{-2^{k+1}} + C'' \bigg( \frac{1}{2 C''} 2^{-2^k} \bigg)^2 = \frac{1}{2 C''} 2^{-2^{k+1}},
%\end{equation}
%which establishes \eqref{exp_decay_induction} for $k+1$. It thus follows by induction that for all $k \geq 0$,

Hence, with $\lambda_2 = \frac{\log 2}{K_0^{\gamma}}$,
\begin{equation}
f(\lambda^k n_0) \leq \frac{1}{2 C''} 2^{-2^k} \leq e^{- (\lambda^k)^{\gamma} \log 2} = e^{- \lambda_2 (\lambda^k n_0 / m)^{\gamma}},
\end{equation}
since $2^k = \lambda^{k \log 2 / \log \lambda} \geq \lambda^{k \gamma}$, from \eqref{eq:choice_gamma}. We can then write, as for \eqref{eq:block_argument2},
\begin{equation}
\PPh_p^{(m)} \big( \Cv^*([0, 5 \lambda^k n_0] \times [0, \lambda^k n_0]) \big) \leq 7 f(\lambda^k n_0) \leq 7 e^{- \lambda_2 (\lambda^k n_0 / m)^{\gamma}}.
\end{equation}
From this, \eqref{eq:stretched_exp_decay} follows easily for a general $n \geq n_0$, while the case $n < n_0$ is an immediate consequence of $n_0 = K_0 m$. This completes the proof of Proposition \ref{prop:exp_decay}.
%Now, consider a general $n \geq n_0$, and let $k \geq 0$ be such that $\lambda^k n_0 \leq n < \lambda^{k+1} n_0$. We have
%\begin{equation}
%\PPh_p^{(m)} \big( \Cv^*([0, 2 n] \times [0, n]) \big) \leq \PPh_p^{(m)} \big( \Cv^*([0, 5 \lambda^k n_0] \times [0, \lambda^k n_0]) \big)
%\end{equation}
%(recall that $\lambda \leq \frac{5}{2}$), so
%\begin{equation}
%\PPh_p^{(m)} \big( \Cv^*([0, 2 n] \times [0, n]) \big) \leq 7 e^{- \lambda_2 (\lambda^k n_0 / m)^{\gamma}} = 7 e^{- \lambda_2 (\lambda^{k+1} n_0 / m)^{\gamma} / \lambda^{\gamma}} \leq 7 e^{- \lambda_2 (n / m)^{\gamma} / \lambda^{\gamma}},
%\end{equation}
%which completes the proof of Proposition \ref{prop:exp_decay} (adjusting the constant in front, to take care of the values $n < n_0 = K_0 m$).
\end{proof}

\begin{corollary} \label{cor:exp_decay_theta}
Let $K \geq 1$.
\begin{itemize}
\item[(i)] There exist $\ul \lambda, \ol \lambda, m_0 > 0$ (depending on $c_1$, $c_2$, $c_3$, $\alpha$, $\beta$ and $K$) such that: for all $m \geq m_0$, all $p > p_c$ with $K^{-1} m \leq L(p) \leq K m$, and all $n \geq K^{-1} m$,
\begin{equation}
\ul \lambda \theta(p) \leq \PPh_p^{(m)} \big( 0 \lra \din \Ball_n \big) \leq \ol \lambda \theta(p).
\end{equation}
\item[(ii)] Moreover, for all $\ve > 0$, there exist $\kappa_0$ and $m_0$ (depending on $c_1$, $c_2$, $c_3$, $\alpha$, $\beta$, $K$ and $\ve$) such that: for all $m \geq m_0$, all $p > p_c$ with $K^{-1} m \leq L(p) \leq K m$, and all $n \geq \kappa_0 m$,
\begin{equation}
(1 - \ve) \theta(p) \leq \PPh_p^{(m)} \big( 0 \lra \din \Ball_n \big) \leq (1 + \ve) \theta(p).
\end{equation}
\end{itemize}
\end{corollary}

\begin{proof}[Proof of Corollary \ref{cor:exp_decay_theta}]
As for Proposition \ref{prop:exp_decay}, the proof is a suitable adaptation of that for a similar result in Bernoulli percolation.

(i) Using a sequence of overlapping rectangles as in Figure \ref{fig:overlapping_rectangles} (first with $n_0 = n$, and then with $n_0 = K^{-1} m$), we deduce from Proposition \ref{prop:exp_decay} and the FKG inequality (for the process with holes, see Remark \ref{rem:FKG_holes}), combined with Proposition \ref{prop:crossing} and \eqref{eq:RSW}, that for all $m$ sufficiently large,
\begin{equation}
\PPh_p^{(m)} \big( 0 \lra \din \Ball_n \big) \asymp \PPh_p^{(m)} ( 0 \lra \infty ) \asymp \PPh_p^{(m)} \big( 0 \lra \din \Ball_{K^{-1} m} \big)
\end{equation}
uniformly in $n$ and $p$ with the required properties. Using Proposition \ref{prop:one_arm_stability}, we have
\begin{equation}
\PPh_p^{(m)} \big( 0 \lra \din \Ball_{K^{-1} m} \big) = \PP_p \big( 0 \lra \din \Ball_{K^{-1} m} \big) \cdot (1 + o(1))
\end{equation}
as $m \to \infty$. Finally, (i) now follows from the standard result for Bernoulli percolation that $\PP_p \big( 0 \lra \din \Ball_{K^{-1} m} \big) \asymp \theta(p)$ (which can easily be obtained from \eqref{eq:exp_decay} and \eqref{eq:RSW}).

\begin{figure}
\begin{center}

\includegraphics[width=.7\textwidth]{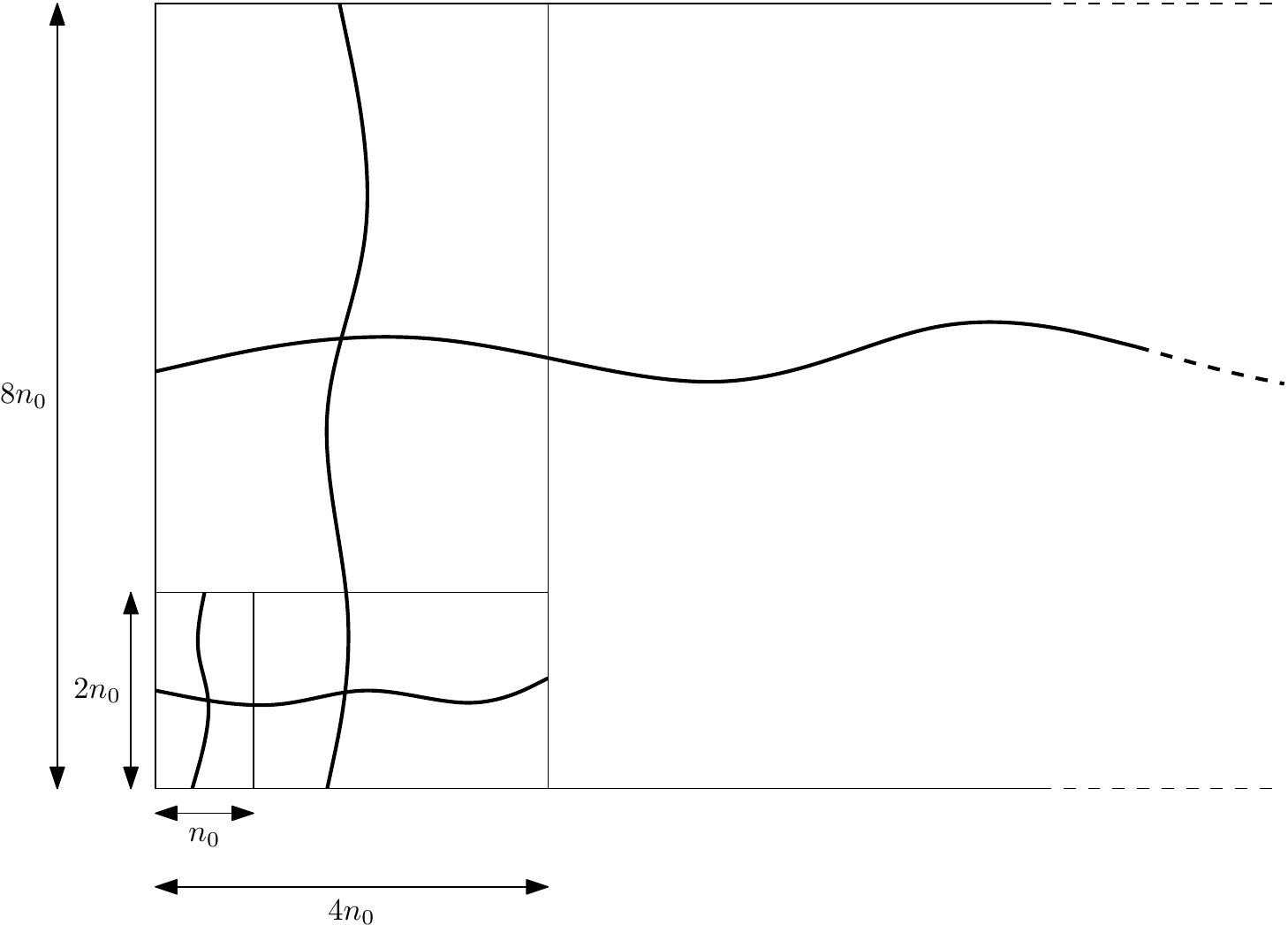}
\caption{\label{fig:overlapping_rectangles} We use a sequence of overlapping, ``horizontal'' and ``vertical'' rectangles, with side lengths $2^{i+1} n_0$ and $2^i n_0$ ($i \geq 0$), for some well-chosen $n_0 \geq 1$.}

\end{center}
\end{figure}

(ii) This follows from similar reasonings, noting that $\PPh_p^{(m)} \big( 0 \lra \din \Ball_{\kappa_0 m} \big)$ and $\PP_p \big( 0 \lra \din \Ball_{\kappa_0 m} \big)$ can be made arbitrarily close to, respectively, $\PPh_p^{(m)} ( 0 \lra \infty )$ and $\theta(p)$ (in ratio), by choosing $\kappa_0 \geq K$ large enough.
\end{proof}

\subsection{Largest cluster in a box} \label{sec:BCKS}

We now prove an analog of \eqref{eq:largest_cluster} in our setting, for boxes with side length $\gg m$. This result, Proposition \ref{prop:largest_cluster} below, is of key importance for our analysis of forest fire processes (FFWoR) in Section \ref{sec:existence_excep_scales}. Its proof follows similar ideas as for the analogous result for Bernoulli percolation in \cite{BCKS2001}, with extra care needed to handle the disturbing effect of large holes. We emphasize that, to do this, the four-arm stability in Section \ref{sec:four_arm_stability} is crucial, although this is not immediately visible in the proof of Proposition \ref{prop:largest_cluster}: it is used indirectly, via other stability results treated earlier in Section \ref{sec:other_stability}.

Recall the definition of a net in Definition \ref{def:net}.

\begin{proposition} \label{prop:largest_cluster}
Let $K \geq 1$, and $(p_m)_{m \geq 1}$ in $(p_c,1)$ satisfying $K^{-1} m \leq L(p_m) \leq K m$. If $(n_m)_{m \geq 1}$ is a sequence of integers such that $n_m \gg m (\log m)^2$ as $m \to \infty$, then for all $\ve > 0$: with high $\PPh_{p_m}^{(m)}$-probability as $m \to \infty$, there exists a net $\calN$ in $\Ball_{n_m}$ with mesh $\bar{n}_m := (n_m m)^{1/2}$, and the cluster $\cluster_{\calN}$ of this net (in $\Ball_{n_m}$) has a volume satisfying
\begin{equation} \label{eq:prop_largest_cluster}
\frac{|\cluster_{\calN}|}{|\Ball_{n_m}| \theta(p_m)} \in (1 - \ve, 1 + \ve).
\end{equation}
\end{proposition}

\begin{remark}
Though we will not use this fact, note that $\cluster_{\calN}$ then has to be the largest cluster in $\Ball_{n_m}$, with high probability (similarly to Remark \ref{rem:BCKS}). We also remark that the assumption $n_m \gg m (\log m)^2$ is not optimal, but it is enough for our purpose. Indeed, we will typically (in Section \ref{sec:existence_excep_scales}) apply Proposition \ref{prop:largest_cluster} to cases where $n_m/m$ is at least a small power of $m$.
\end{remark}

\begin{proof}[Proof of Proposition \ref{prop:largest_cluster}]
For similar reasons as for Lemma \ref{lem:net} (now using Proposition \ref{prop:exp_decay}, e.g. with $\gamma = \frac{1}{2}$, instead of \eqref{eq:exp_decay}), a net $\calN$ as stated in the proposition exists with high probability:
\begin{equation}
\PPh_{p_m}^{(m)}( \calN \text{ exists} ) \geq 1 - C_1 \bigg( \frac{n_m}{\bar{n}_m} \bigg)^2 e^{ - C_2 \big( \frac{\bar{n}_m}{K m} \big)^{1/2}} = 1 - C_1 \Big( \frac{n_m}{m} \Big) e^{ - \frac{C_2}{K^{1/2}} \big( \frac{n_m}{m} \big)^{1/4}} \stackrel[m \to \infty]{\longrightarrow}{} 1.
\end{equation}
On the event that $\calN$ exists, the volume of its cluster $\cluster_{\calN}$ satisfies
\begin{equation} \label{eq:comp_volume_BCKS}
Y_m \leq | \cluster_{\calN} | \leq Y_m + \eta_m,
\end{equation}
with
\begin{equation}
Y_m := \sum_{x \in \Ball_{n_m - 4 \bar{n}_m}} \ind_{x \lra \din \Ball_{4 \bar{n}_m}(x)} \quad \text{and} \quad \eta_m := \sum_{x \in \Ball_{n_m} \setminus \Ball_{n_m - 4 \bar{n}_m}} \ind_{x \lra \din \Ball_{4 \bar{n}_m}(x)}.
\end{equation}
We now use a second-moment argument for $Y_m$. We have, denoting by $\EEh_{p_m}^{(m)}$ the expectation with respect to $\PPh_{p_m}^{(m)}$,
\begin{equation}
\big( |\Ball_{n_m}| - C_3 \bar{n}_m n_m \big) \PPh_{p_m}^{(m)} \big( 0 \lra \din \Ball_{4 \bar{n}_m} \big) \leq \EEh_{p_m}^{(m)} \big[ Y_m \big] \leq |\Ball_{n_m}| \PPh_{p_m}^{(m)} \big( 0 \lra \din \Ball_{4 \bar{n}_m} \big).
\end{equation}
Since $4 \bar{n}_m \gg m$, we can apply (ii) of Corollary \ref{cor:exp_decay_theta}: for $m$ large enough,
\begin{equation} \label{eq:proof_BCKS_theta}
\Big( 1 - \frac{\ve}{4} \Big) \theta(p_m) \leq \PPh_{p_m}^{(m)} \big( 0 \lra \din \Ball_{4 \bar{n}_m} \big) \leq \Big( 1 + \frac{\ve}{4} \Big) \theta(p_m),
\end{equation}
and thus
\begin{equation} \label{eq:proof_BCKS_Y}
\Big( 1 - \frac{\ve}{2} \Big) |\Ball_{n_m}| \theta(p_m) \leq \EEh_{p_m}^{(m)} \big[ Y_m \big] \leq \Big( 1 + \frac{\ve}{4} \Big) |\Ball_{n_m}| \theta(p_m).
\end{equation}
We now estimate $\Var(Y_m)$, and for that we denote $E_m^x := \{x \lra \din \Ball_{4 \bar{n}_m}(x) \}$ ($x \in V$). Note that, contrary to the Bernoulli percolation case, even if $x$ and $y$ are far apart, the events $E_m^x$ and $E_m^y$ are not independent. This is due to the possible existence of large holes coming close to both $x$ and $y$. To control the effect of this, we introduce the auxiliary events $\tilde{E}_m^x := \{x \lra \din \Ball_{4 \bar{n}_m}(x)$ occurs without the holes centered in $(\Ball_{8 \bar{n}_m}(x))^c\}$. We have
\begin{equation}
\PPh_{p_m}^{(m)} \big( \ind_{\tilde{E}_m^x} \neq \ind_{E_m^x} \big) \leq \PPh_{p_m}^{(m)} \big( \calH \big( \Ann_{4 \bar{n}_m, 8 \bar{n}_m}(x) \big) \big) \cdot \PPh_{p_m}^{(m)} \big( \tilde{\tilde{E}}_m^x \big),
\end{equation}
where $\tilde{\tilde{E}}_m^x$ is the event that $x \lra \din \Ball_{4 \bar{n}_m}(x)$ occurs without the holes (and thus is independent of $\calH ( \Ann_{4 \bar{n}_m, 8 \bar{n}_m}(x) )$). We obtain from Lemma \ref{lem:no_crossing} that
\begin{equation} \label{eq:largest_cluster_no_crossing}
\PPh^{(m)}_{p_m} \big( \ind_{\tilde{E}_m^x} \neq \ind_{E_m^x} \big) \leq \frac{C}{m^{\beta - \alpha}} e^{-C' 4 \bar{n}_m/m} \cdot \PP_{p_m} \big( 0 \lra \din \Ball_{4 \bar{n}_m} \big) \leq C_4 e^{-C_5 (n_m/m)^{1/2}} \theta(p_m),
\end{equation}
using also that
\begin{equation} \label{eq:equiv_theta_BCKS}
\PP_{p_m} \big( 0 \lra \din \Ball_{4 \bar{n}_m} \big) \leq C'' \theta(p_m).
\end{equation}
(this follows easily from \eqref{eq:exp_decay} and \eqref{eq:RSW}). Moreover, for $x, y \in V$ with $\|x-y\|_{\infty} > 16 \bar{n}_m$, we have that $\tilde{E}_m^x$ and $\tilde{E}_m^y$ are independent, so $\Cov \big( \ind_{\tilde{E}_m^x}, \ind_{\tilde{E}_m^y} \big) = 0$. We deduce, for such $x$, $y$,
\begin{align*}
\Cov \big( \ind_{E_m^x}, \ind_{E_m^y} \big) & = \Cov \big( \ind_{E_m^x}, \ind_{E_m^y} \big) - \Cov \big( \ind_{\tilde{E}_m^x}, \ind_{\tilde{E}_m^y} \big)\\
& = \Cov \big( \ind_{E_m^x} - \ind_{\tilde{E}_m^x}, \ind_{E_m^y} \big) + \Cov \big( \ind_{\tilde{E}_m^x}, \ind_{E_m^y} - \ind_{\tilde{E}_m^y} \big)\\
& \leq \Big[ \PPh_{p_m}^{(m)} \big( \ind_{\tilde{E}_m^x} \neq \ind_{E_m^x} \big) \cdot \PPh_{p_m}^{(m)} \big( E_m^y \big) \Big]^{1/2} + \Big[ \PPh_{p_m}^{(m)} \big( \tilde{E}_m^x \big) \cdot \PPh_{p_m}^{(m)} \big( \ind_{\tilde{E}_m^y} \neq \ind_{E_m^y} \big) \Big]^{1/2}
\end{align*}
(applying the Cauchy-Schwarz inequality twice). Hence, \eqref{eq:largest_cluster_no_crossing} and \eqref{eq:equiv_theta_BCKS} imply (still for $x$, $y$ as mentioned above)
\begin{equation}
\Cov \big( \ind_{E_m^x}, \ind_{E_m^y} \big) \leq C_6 e^{- C_7 (n_m/m)^{1/2}} \theta(p_m).
\end{equation}
We now write
\begin{align}
\Var(Y_m) & = \sum_{x,y \in \Ball_{n_m - 4 \bar{n}_m}} \Cov \big( \ind_{E_m^x}, \ind_{E_m^y} \big) \nonumber \\
& \leq \sum_{\substack{x \in \Ball_{n_m - 4 \bar{n}_m}\\ y \: : \: \|x-y\|_{\infty} \leq 16 \bar{n}_m}} \Cov \big( \ind_{E_m^x}, \ind_{E_m^y} \big) + |\Ball_{n_m}|^2 \cdot C_6 e^{-C_7 (n_m/m)^{1/2}} \theta(p_m). \label{eq:end_proof_BCKS1}
\end{align}
For any $x \in \Ball_{n_m - 4 \bar{n}_m}$,
\begin{align}
\sum_{\substack{y \: : \: \|x-y\|_{\infty} \leq 16 \bar{n}_m}} \Cov \big( \ind_{E_m^x}, \ind_{E_m^y} \big) & \leq \sum_{y \: : \: \|x-y\|_{\infty} \leq 16 \bar{n}_m} \PPh_{p_m}^{(m)} \big( E_m^x \cap E_m^y \big) \nonumber \\
& \leq \sum_{y \: : \: \|x-y\|_{\infty} \leq 16 \bar{n}_m} \PP_{p_m} \big( \{ x \lra \din \Ball_{4 \bar{n}_m}(x) \} \cap \{ y \lra \din \Ball_{4 \bar{n}_m}(y) \} \big) \nonumber \\
& \leq C_8 |\Ball_{16 \bar{n}_m}| \PP_{p_m} \big( 0 \lra \din \Ball_{4 \bar{n}_m} \big)^2 \nonumber \\
& \leq C_9 (\bar{n}_m)^2 \theta(p_m)^2 \label{eq:end_proof_BCKS2}
\end{align}
where the third inequality follows from a standard summation argument (over $y \in \Ann_{2^i,2^{i+1}}(x)$, $0 \leq i \leq \lfloor \log_2 ( 16 \bar{n}_m ) \rfloor$), and the fourth inequality uses \eqref{eq:equiv_theta_BCKS}. By combining \eqref{eq:end_proof_BCKS1} and \eqref{eq:end_proof_BCKS2}, we obtain
\begin{align*}
\Var(Y_m) & \leq C_9 |\Ball_{n_m}| \cdot (\bar{n}_m)^2 \theta(p_m)^2 + |\Ball_{n_m}|^2 \cdot C_6 e^{-C_7 (n_m/m)^{1/2}} \theta(p_m)\\
& \leq C_{10} (n_m)^3 m \theta(p_m)^2 + C_{11} (n_m)^4 m^{-1} \theta(p_m)^2.
\end{align*}
For the last inequality, we used that $\theta(p_m) \geq m^{-\upsilon}$ for some $\upsilon > 0$ (from the assumption $L(p_m) \leq K m$, \eqref{eq:equiv_theta} and \eqref{eq:1arm}), so $e^{-C_7 (n_m/m)^{1/2}} \leq m^{-1} \theta(p_m)$ for $m$ large enough (since $n_m / m \gg (\log m)^2$). Hence,
$$\Var(Y_m) \ll (n_m)^4 \theta(p_m)^2 \asymp \big( \EEh_{p_m}^{(m)} \big[ Y_m \big] \big)^2$$
(using \eqref{eq:proof_BCKS_Y}), so
\begin{equation} \label{eq:end_proof_BCKS3}
\frac{Y_m}{|\Ball_{n_m}| \theta(p_m)} \in \Big( 1 - \frac{3 \ve}{4}, 1 + \frac{3 \ve}{4} \Big) \quad \text{w.h.p. as $m \to \infty$.}
\end{equation}
Finally,
\begin{equation}
\EEh_{p_m}^{(m)} \big[ \eta_m \big] \leq C_{12} n_m \bar{n}_m \PPh_{p_m}^{(m)} \big( 0 \lra \din \Ball_{4 \bar{n}_m} \big) \ll |\Ball_{n_m}| \theta(p_m)
\end{equation}
as $m \to \infty$ (using \eqref{eq:proof_BCKS_theta}), so we obtain from Markov's inequality that
\begin{equation} \label{eq:end_proof_BCKS4}
\frac{\eta_m}{|\Ball_{n_m}| \theta(p_m)} \leq \frac{\ve}{4} \quad \text{w.h.p. as $m \to \infty$.}
\end{equation}
This allows us to conclude, by combining \eqref{eq:comp_volume_BCKS}, \eqref{eq:end_proof_BCKS3} and \eqref{eq:end_proof_BCKS4}.
\end{proof}

\begin{remark} \label{rem:BCKS_annulus}
In Section \ref{sec:existence_excep_scales}, we also need a version of Proposition \ref{prop:largest_cluster} in annuli $\Ann_{\frac{1}{2} n_m, n_m}$ (instead of balls $\Ball_{n_m}$). It is easy to see that the same proof applies in this setting, so that an analogous result holds true, with $|\Ball_{n_m}|$ replaced by $\big|\Ann_{\frac{1}{2} n_m, n_m}\big|$.
\end{remark}

\section{Application: forest fires} \label{sec:application_FF}

We now turn to the forest fire processes, with or without recovery. After giving precise definitions in Section \ref{sec:def_FF}, we explain in Section \ref{sec:coupling} how to couple these processes with a process where ``cluster-distributed'' holes are independently ``removed'' at the ignition times. This coupling provides in particular a lower bound for the forest fire processes at a time $t_c - \ve$ slightly before $t_c$, and we estimate quantitatively this lower bound in Section \ref{sec:comparison_holes}. More precisely, we explain how it fits into the framework of percolation with holes, studied in Sections \ref{sec:process_holes} to \ref{sec:other_stability}, for some $\pi$ and $\rho$ that we compute. Before that, we need to introduce the exceptional scales for the forest fire processes, which we do in Section \ref{sec:def_exceptional_scales}. Even if this section seems to pertain only to usual near-critical Bernoulli percolation, it contains some computations required for Section \ref{sec:comparison_holes}, and it is central to Section \ref{sec:existence_excep_scales}.

\subsection{Definition of the processes} \label{sec:def_FF}

We now define precisely the various processes under consideration, that were already mentioned in the Introduction. Let $G$ be a finite subgraph of the full lattice $\TT$, with set of vertices $V_G$.

The first process that we consider is well-known, and it has a simple dynamics: we call it the \emph{pure birth} (or \emph{pure growth}) process. Initially, each vertex in $V_G$ is vacant (state $0$). Vacant vertices become occupied (state $1$), independently of each other, at rate $1$, and then remain occupied forever. Let $X_t(v)$ denote the state of vertex $v \in V_G$ at time $t$. Clearly, at each given time $t$, the random variables $X_t(v)$, $v \in V_G$, are i.i.d., equal to $0$ or $1$ with respective probabilities $e^{-t}$ and $1-e^{-t}$. We can thus see $(X_t(v))_{v \in V_G}$ as a percolation configuration with parameter $1 - e^{-t}$. We denote by $\cluster_t(v)$ the occupied cluster of $v$ at time $t$.

We also introduce \emph{forest fires} (or ``epidemics'') \emph{without recovery}. Again, each vertex is initially vacant (state $0$), and it becomes occupied (state $1$) at rate $1$. However, there is now an additional mechanism: vertices are hit by lightning (``spontaneously infected from outside'') at rate $\zeta$, the parameter of this model. If an occupied vertex is hit by lightning, then its entire occupied cluster is burnt immediately: all vertices in the cluster become vacant, and remain vacant forever; we call these vertices \emph{burnt} (state $-1$). The configuration at time $t$ is denoted by $(\sigma_t(v))_{v \in V_G}$.

Occasionally, we mention other processes, in particular \emph{forest fires with recovery}. This process corresponds to the classical Drossel-Schwabl model \cite{DrSc1992}, and we use the notation $\bar{\sigma}$ for it. The difference with the previous model is that now, burnt vertices behave the same as ``ordinary'' vacant vertices: they become occupied at rate $1$ (so this process has just two states: $0$ and $1$).

\begin{remark}
The processes above were defined for finite subgraphs of $\TT$. Obviously, the $X$ process can be defined for the full lattice $\TT$ as well. This is not clear at all for the $\sigma$ and $\bar{\sigma}$ processes, but it can be / has been done, by using clever arguments by D\"urre \cite{Du2006} (this stands in contrast with parameter-$N$ volume-frozen percolation, which can be represented as a finite-range interacting particle system, so that the general theory of such systems can be applied). However, in this paper we restrict to finite graphs, for which existence is clear. Also, we focus on the $\sigma$ process. Several of the results that we prove for $\sigma$ can be proved (in a very similar way) for $\bar{\sigma}$ as well. Unfortunately, we cannot (yet) prove analogs for $\bar{\sigma}$ of our main results in Section \ref{sec:existence_excep_scales}: see the comments in Section \ref{sec:forest_fires}.
\end{remark}

\subsection{Coupling with independently removed clusters} \label{sec:coupling}

The description of the processes $(X_t)$ and $(\sigma_t)$ shows immediately (by using the obvious coupling) that (with the natural order $-1 < 0 < 1$), the former dominates the latter. It is important for our purposes to also have at our disposal a domination relation in the other direction: for each $t \geq 0$, $\sigma_t$ dominates an auxiliary process obtained from $X_t$ by removing, at each ``ignition event'' $(\tau,v)$, with $\tau < t$, an ``independent copy'' of $\overline{\cluster_{\tau}(v)} := \cluster_{\tau}(v) \cup \big( \dout \cluster_{\tau}(v) \big)$. This additional process, that we denote by $Y$, will provide a connection with the general theory of percolation with impurities from Sections \ref{sec:process_holes}--\ref{sec:other_stability} (this connection is established more explicitly in Section \ref{sec:comparison_holes}, and we then apply it to obtain our main results for $\sigma$ in Section \ref{sec:existence_excep_scales}).

More formally, for each $v \in V_G$, let $\calT_v$ be the (random) set of ignition times at $v$, and for $t \geq 0$, $\calT^t_v := \{ \tau \in \calT_v \: : \: \tau < t \}$. For $v \in V_G$ and $t \geq 0$, we denote by $\mu_{v, t}$ the distribution of the occupied cluster $\cluster_t(v)$ of $v$ in the configuration $X_t$. We then introduce the marked Poisson point process obtained from the Poisson process of ignitions, by assigning, for each $v \in V_G$ and each $\tau \in \calT_v$, a random ``mark'' $\cluster_{v, \tau}$ drawn independently, according to the distribution $\mu_{v, \tau}$. Finally, we define $Y_t$ obtained from $X_t$ by ``removing'' the subsets $\overline{\cluster_{v, \tau}} = \cluster_{v, \tau} \cup \big( \dout \cluster_{v, \tau} \big)$ ($v \in V_G$, $\tau \in \calT_v$, $\tau < t$), i.e.
$$Y_t(v) := X_t(v) \ind_{v \notin \bigcup_{v' \in V, \tau \in \calT^t_{v'}} \overline{\cluster_{v', \tau}}} \quad (v \in V_G).$$

As said above, we claim that $\sigma$ dominates, in some sense, $Y$. Some nuance is needed here, because $\sigma$ has states $-1$, $0$ and $1$, while $Y$ has only states $0$ and $1$. More precisely, our claim is the following.

\begin{lemma} \label{lem:stoch_domination}
For all $t \geq 0$, $(\ind_{\sigma_t(v) = 1})_{v \in V_G}$ stochastically dominates $(Y_t(v))_{v \in V_G}$.
\end{lemma}

So, informally speaking, the $\sigma$ configuration at time $t$, with state $-1$ ``read'' as $0$, stochastically dominates the $Y$ configuration.

\begin{proof}[Proof of Lemma \ref{lem:stoch_domination}]
Let $t > 0$, and fix all the ignitions before time $t$, denoted by $(s_1, v_1), \ldots, (s_n, v_n)$ ($0 < s_1 < \ldots < s_n < t$, and $v_1, \ldots, v_n \in V_G$). Note that it is sufficient to prove the desired stochastic domination with fixed ignitions (and random births), which we do now.

Let $\calP$ denote the random process of births up to time $t$. We can visualize this process $\calP$ in the usual way: to each vertex $v$, we assign a half-line (corresponding to $[0,\infty)$), and for each of these half-lines, we consider a Poisson point process with intensity $1$. The $X$ process corresponding to $\calP$ is then, clearly, described as follows: for each $s \in (0, t)$,
$$X_s(v) = \ind_{\substack{\{\calP \text{ has a point before time } s\\ \text{on the half-line assigned to } v\} }} \quad (v \in V_G).$$

The proof is based on a coupling argument, and to do that, it is convenient (due to the notion of ``independent copy'' in the definition of the $Y$ process) to introduce $n+1$ independent copies of $\calP$, denoted by $(\calP^{(i)})_{1 \leq i \leq n+1}$, that we use to build a realization of the $\sigma$ process. The $X$ process corresponding to $\calP^{(i)}$ ($1 \leq i \leq n+1$) is then denoted by $X^{\calP^{(i)}}$ (and similarly for clusters), so for example, $X_s^{\calP^{(i)}}(v)$ is the indicator function of $\{\calP^{(i)}$ has a point before time $s$ on the half-line assigned to $v\}$.

\begin{figure}[t]
\begin{center}

\includegraphics[width=0.95\textwidth]{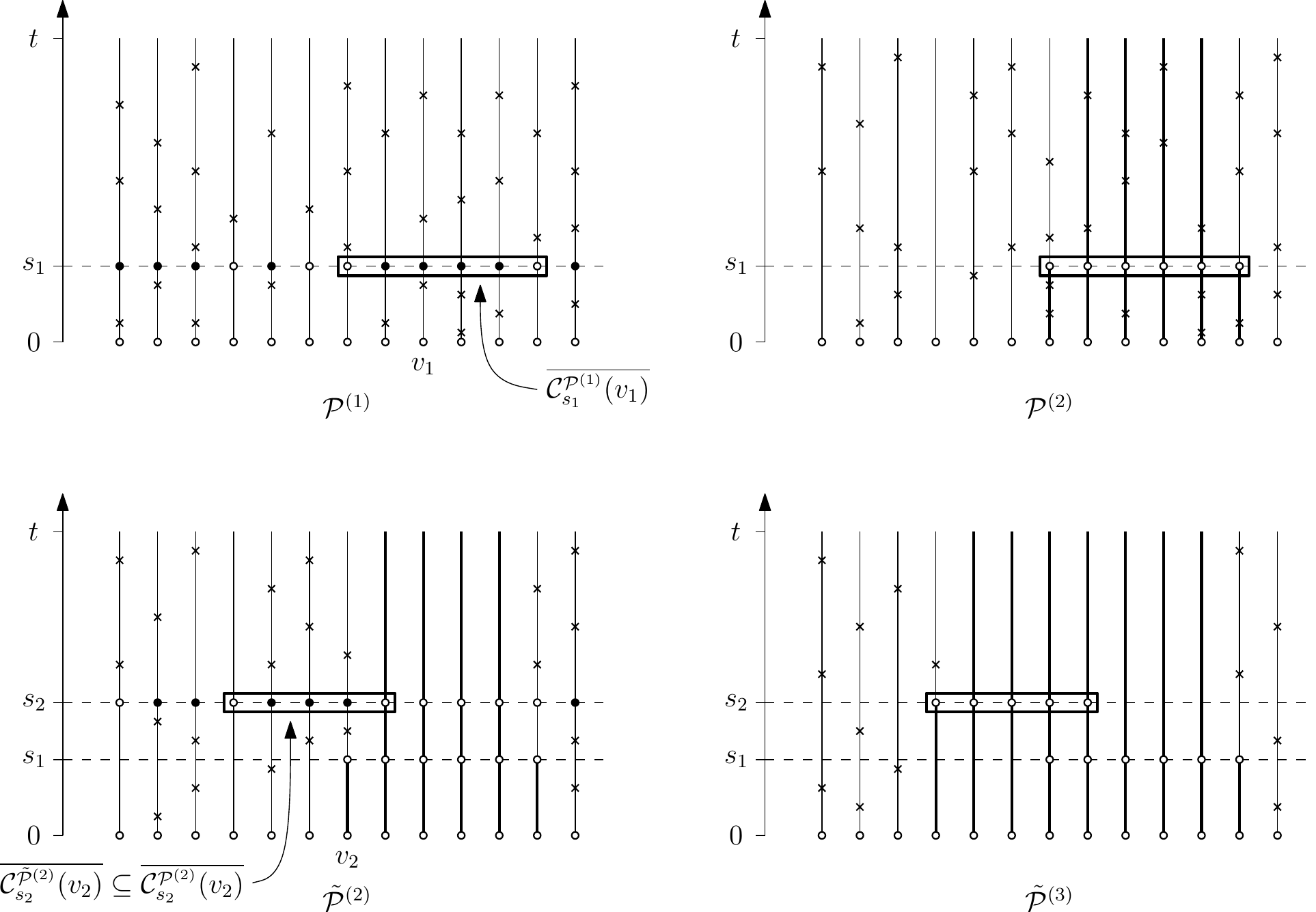}
\caption{\label{fig:coupling_holes} We construct the coupling in an iterative way: (1) the modified configuration $\tilde{\calP}^{(2)}$ is obtained from $\calP^{(2)}$ by erasing all births before time $s_1$ on $\dout \cluster_{s_1}^{\calP^{(1)}}(v_1)$, and all births on $\cluster_{s_1}^{\calP^{(1)}}(v_1)$, (2) the modified configuration $\tilde{\calP}^{(3)}$ is obtained from $\calP^{(3)}$ by erasing all births before time $s_2$ on $\dout \cluster_{s_2}^{\tilde{\calP}^{(2)}}(v_2)$, all births before time $s_1$ on $\dout \cluster_{s_1}^{\calP^{(1)}}(v_1)$, and all births on $\cluster_{s_1}^{\calP^{(1)}}(v_1) \cup \cluster_{s_2}^{\tilde{\calP}^{(2)}}(v_2)$, (3) we keep proceeding in the same way, producing a sequence of configurations $\tilde{\calP}^{(i)} \subseteq \calP^{(i)}$, $i = 4, \ldots, n+1$.}

\end{center}
\end{figure}

Here is, somewhat informally described, the construction (illustrated on Figure \ref{fig:coupling_holes}). Up to time $s_1$, we let the forest fire without recovery (FFWoR) process run, ``driven'' by $\calP^{(1)}$. Then, as we should, we ``burn'' the cluster at time $s_1$, i.e. all vertices of $\cluster_{s_1}^{\calP^{(1)}}(v_1)$ become vacant, and remain so forever. What information does this cluster give us about the states of the other vertices at time $s_1$ in the FFWoR process? Of course, the vertices on the outer boundary of this cluster are vacant at time $s_1$, but this is all information we have: all the vertices outside $\overline{\cluster_{s_1}^{\calP^{(1)}}(v_1)} = \cluster_{s_1}^{\calP^{(1)}}(v_1) \cup \big( \dout \cluster_{s_1}^{\calP^{(1)}}(v_1) \big)$ are still distributed as in $X^{\calP^{(1)}}$.

Hence, (the law of) the future evolution of the FFWoR process does not change if, as we will do, we replace the $\calP^{(1)}$ configurations on the entire timelines of these vertices by $\calP^{(2)}$ configurations. The same holds for the half-lines above time $s_1$ of the vertices in $\dout \cluster_{s_1}^{\calP^{(1)}}(v_1)$. Finally, since the vertices of $\cluster_{s_1}^{\calP^{(1)}}(v_1)$ burn at time $s_1$ and remain vacant forever, we may (without changing the distribution of the FFWoR process after time $s_1$) remove all the points from the entire timelines of these vertices. We denote the resulting point configuration by $\tilde{\calP}^{(2)}$. Note that, considered as sets of points, we clearly have that $\tilde{\calP}^{(2)} \subseteq \calP^{(2)}$. It is also clear from the above observations and definitions that, at each time $s \in [s_1, s_2)$, the FFWoR process at time $s$ (with state $-1$ ``read'' as $0$) has the same distribution as $X_s^{\tilde{\calP}^{(2)}}$.

Now, at time $s_2$, the occupied cluster of $v_2$, i.e. $\cluster_{s_2}^{\tilde{\calP}^{(2)}}(v_2)$, is burnt (note that it is contained in $\cluster_{s_2}^{\calP^{(2)}}(v_2)$). Using the same arguments as for the first burning event, i.e. $(s_1, v_1)$, we replace the point configurations of the entire timelines of the vertices outside $\overline{\cluster_{s_1}^{\calP^{(1)}}(v_1)} \cup \overline{\cluster_{s_2}^{\tilde{\calP}^{(2)}}(v_2)}$ by those of $\calP^{(3)}$. For the timelines of the vertices in $\dout \cluster_{s_2}^{\tilde{\calP}^{(2)}}(v_2) \setminus \cluster_{s_1}^{\calP^{(1)}}(v_1)$, we replace the point configurations after time $s_2$ by those of $\calP^{(3)}$. For the timelines of the vertices in $\dout \cluster_{s_1}^{\calP^{(1)}}(v_1) \setminus \overline{\cluster_{s_2}^{\tilde{\calP}^{(2)}}(v_2)}$, we replace the configurations after time $s_1$ by those of $\calP^{(3)}$. Finally, the points on the entire timelines of the vertices in $\cluster_{s_1}^{\calP^{(1)}}(v_1)$ are removed, and we do this also for $\cluster_{s_2}^{\tilde{\calP}^{(2)}}(v_2)$. We denote the resulting point configuration by $\tilde{\calP}^{(3)}$, which is clearly a subset of $\calP^{(3)}$.

Again, one can check that, for each time $s \in [s_2, s_3)$, the FFWoR process (with $-1$ read as $0$) has the same distribution as $X_s^{\tilde{\calP}^{(3)}}$. At time $s_3$, we burn $\cluster_{s_3}^{\tilde{\calP}^{(3)}}(v_3)$ (which is contained in $\cluster_{s_3}^{\calP^{(3)}}(v_3)$), and so on. Iterating this procedure, we obtain
$$\calP^{(1)}, \tilde{\calP}^{(2)} \subseteq \calP^{(2)}, \ldots, \tilde{\calP}^{(n+1)} \subseteq \calP^{(n+1)},$$
and conclude that the FFWoR process at time $t$ (with the state $-1$ read as $0$) has the same distribution as $X_t^{\tilde{\calP}^{(n+1)}}$. Moreover, it follows from the procedure that, outside of
\begin{equation} \label{eq:union_clusters}
\overline{\cluster_{s_1}^{\calP^{(1)}}(v_1)} \cup \overline{\cluster_{s_2}^{\tilde{\calP}^{(2)}}(v_2)} \cup \ldots \cup \overline{\cluster_{s_n}^{\tilde{\calP}^{(n)}}(v_n)},
\end{equation}
$X_t^{\tilde{\calP}^{(n+1)}}$ is equal to $X_t^{\calP^{(n+1)}}$.

Finally, since the set of vertices in \eqref{eq:union_clusters} is contained in the union of the sets $\overline{\cluster_{s_i}^{\calP^{(i)}}(v_i)}$, $1 \leq i \leq n$, and these $n$ sets are (clearly) independent of each other and of $X_t^{\calP^{(n+1)}}$, the result follows.
\end{proof}

\begin{remark} \label{rem:stoch_domination_inhomogeneous}
Note that Lemma \ref{lem:stoch_domination} still holds (with practically the same proof) in the case when the ignition and birth rates are vertex- and time-dependent, and also when the boundary of a burnt cluster stays vacant forever. Moreover, it is valid for forest fires with recovery as well, i.e. for the $\bar{\sigma}$ process, up to minor modifications: for instance, to produce $\tilde{\calP}^{(2)}$ from $\calP^{(2)}$, remove only the points below time $s_1$ (instead of all the points) from the timelines of the vertices in $\cluster_{s_1}^{\calP^{(1)}}(v_1)$ (since these vertices burn at time $s_1$, but they are allowed to recover at a later time).
\end{remark}

For a subset of vertices $\Lambda \subseteq V_G$, we can consider the FFWoR process on $\Lambda$, i.e. obtained when restricting the ``geographic universe'' to the graph $(\Lambda, E(\Lambda))$, where $E(\Lambda) := \{ e = \{x, y\} \in E \: : \: x,y \in \Lambda \}$. Note that this process on $\Lambda$ does not necessarily coincide with the restriction to $\Lambda$ of the FFWoR process on the whole of $V_G$. For all $t \geq 0$, we denote by $\sigma^{\Lambda}_t = (\sigma^{\Lambda}_t (v))_{v \in \Lambda}$ the configuration at time $t$ of the FFWoR process on $\Lambda$. We will actually make use of the following ``uniform'' version of Lemma \ref{lem:stoch_domination}.

\begin{remark} \label{rem:stoch_domination}
Let $V_i$ ($1 \leq i \leq I$) be subsets of $V_G$, and let $W := \bigcap_{1 \leq i \leq I} V_i$. Then, with $Y_t$ as before (``living'' on the entire graph $G$), it follows from the same coupling argument as for Lemma \ref{lem:stoch_domination} that $\big( \ind_{\min_{1 \leq i \leq I} \sigma^{V_i}_t(v) = 1} \big)_{v \in W}$ stochastically dominates $(Y_t(v))_{v \in W}$.
\end{remark}

\subsection{Exceptional scales} \label{sec:def_exceptional_scales}

In this section, we explicitly define certain length scales (as functions of the parameter $\zeta$), and call them ``exceptional''. In Section \ref{sec:existence_excep_scales}, we will prove that the FFWoR process on boxes with these length scales indeed exhibit an exceptional behavior, in the sense of Section \ref{sec:intro_heuristics} in the Introduction.

Let $p(t) := 1 - e^{-t}$ be the percolation parameter at time $t$, and $t_c := - \log(1 - p_c)$ be the unique value of $t$ for which $p(t) = p_c$. For the sake of convenience, we write, with a slight abuse of notation, $\theta(t) = \theta(p(t))$ and $L(t) = L(p(t))$.

We make repeated use of the correspondence, for $t > t_c$, between ``times'' and ``scales'', via $t \leftrightarrow L(t)$. Recall that we consider a regularized version of $L$, as explained in Section \ref{sec:near_critical_perc}, which is in particular (seen as a function of time $t$) a bijection from $(t_c,\infty)$ to $(0,\infty)$. We define the transformation $\next_{\zeta} : t \in (t_c, \infty) \mapsto \hat{t} = \hat{t}(t,\zeta) \in (t_c, \infty)$ satisfying
\begin{equation} \label{eq:def_t_hat}
L(t)^2 \theta(\hat{t}) \big| \hat{t} - t_c \big| = \zeta^{-1}
\end{equation}
(similar to (7.1) in \cite{BKN2015}). Note that it is well-defined since $\theta(.) \big| . - t_c \big|$ is continuous and strictly increasing on $[t_c, \infty)$, from $0$ to $\infty$. Moreover, $L$ is strictly decreasing on $(t_c, \infty)$ so $\next_{\zeta}$ is strictly increasing, and $\next_{\zeta}(t) \to \infty$ as $t \to \infty$ since $L(t) \to 0$.

It follows immediately from \eqref{eq:def_t_hat}, combined with \eqref{eq:equiv_theta} and \eqref{eq:equiv_L}, that
\begin{equation} \label{eq:rel_t_t_hat}
L(t)^2 \asymp \zeta^{-1} L(\hat{t})^2 \frac{\pi_4(L(\hat{t}))}{\pi_1(L(\hat{t}))}.
\end{equation}

\begin{lemma} \label{lem:fixed_pt}
For $t \in (t_c, \infty)$, let $\varphi(t) := L(t)^2 \theta(t) \big| t - t_c \big|$. We have
\begin{equation}
\varphi(t) \stackrel[t \searrow t_c]{\longrightarrow}{} \infty \quad \text{and} \quad \varphi(t) \stackrel[t \to \infty]{\longrightarrow}{} 0.
\end{equation}
\end{lemma}

\begin{proof}[Proof of Lemma \ref{lem:fixed_pt}]
Since $L(t)$ tends to $0$ exponentially fast as $t \to \infty$, the same is true for $\varphi(t)$, which gives the second limit.

For the first limit, we use successively \eqref{eq:equiv_theta} and \eqref{eq:equiv_L} to obtain: as $t \searrow t_c$,
\begin{equation} \label{eq:fixed_point_pf1}
L(t)^2 \theta(t) \big| t - t_c \big| \asymp L(t)^2 \pi_1(L(t)) \big| t - t_c \big| \asymp \frac{\pi_1(L(t))}{\pi_4(L(t))}.
\end{equation}
We then deduce from the Harris inequality and the observation that $L(t) \to \infty$ as $t \to t_c$ (see the paragraph below \eqref{eq:def_L}) that
\begin{equation} \label{eq:fixed_point_pf2}
\frac{\pi_1(L(t))}{\pi_4(L(t))} \geq \pi_1(L(t))^{-1} \stackrel[t \searrow t_c]{\longrightarrow}{} \infty.
\end{equation}
We get the desired result by combining \eqref{eq:fixed_point_pf1} and \eqref{eq:fixed_point_pf2}.
\end{proof}

Since $\varphi$ is continuous on $(t_c, \infty)$, Lemma \ref{lem:fixed_pt} implies that for all $\zeta > 0$, there exists $t \in (t_c, \infty)$ such that
\begin{equation} \label{eq:fixed_pt}
L(t)^2 \theta(t) \big| t - t_c \big| = \zeta^{-1}.
\end{equation}
This leads us to introduce the ``fixed point'' $t_{\infty}$ of $\next_{\zeta}$. Since there may be several of them ($\varphi$ is not necessarily monotone), we adopt the following definition.
\begin{definition} \label{def:fixed_pt}
For $\zeta > 0$, we introduce
\begin{equation} \label{eq:def_fixed_pt}
t_{\infty} = t_{\infty}(\zeta) := \sup \{ t > t_c \: : \: \next_{\zeta}(t) = t \} > t_c.
\end{equation}
\end{definition}
It follows from the fact that $\varphi(t) \to 0$ as $t \to \infty$ that: $t_{\infty} < \infty$, and for all $t \geq t_{\infty}$, $\hat{t} \geq t$. Note also that $t_{\infty}(\zeta) \to t_c$ as $\zeta \searrow 0$.

Since $\next_{\zeta}(t) \to t_c$ as $t \searrow t_c$, $\next_{\zeta}$ is a bijection from $(t_c, \infty)$ onto itself. We define the \emph{exceptional times} $t_k = t_k(\zeta)$ ($k \geq 0$) by induction as follows. We take
\begin{equation}
t_0 := 2 t_c, \quad \text{and for all $k \geq 0$, } t_{k+1}(\zeta) := \next_{\zeta}^{-1}\big( t_k(\zeta) \big)
\end{equation}
(the choice of $t_0$ is completely arbitrary, and any fixed value $> t_c$ would work). In the following, we always assume that $\zeta$ is small enough so that $t_{\infty}(\zeta) < t_0$, which implies that $t_k(\zeta) > t_{\infty}(\zeta)$ for all $k \geq 0$, and that for a fixed value $\zeta$, $(t_k(\zeta))_{k \geq 0}$ is strictly decreasing.

We also consider the corresponding \emph{exceptional scales} $m_k(\zeta) := L(t_k(\zeta))$ ($k \geq 0$), which satisfy (from \eqref{eq:rel_t_t_hat})
\begin{equation} \label{eq:rel_except_scales}
m_{k+1}^2 \asymp \zeta^{-1} m_k^2 \frac{\pi_4(m_k)}{\pi_1(m_k)}.
\end{equation}
For future use, note that $m_0(\zeta) = L(t_0) = L(2 t_c)$ is a constant, and (from \eqref{eq:rel_except_scales}) $m_1(\zeta) \asymp \frac{1}{\sqrt{\zeta}}$.

By combining \eqref{eq:equiv_theta} and \eqref{eq:equiv_L} with the fact that $\pi_i(n) = n^{-\alpha_i + o(1)}$ (as $n \to \infty$) for $i = 1, 4$, with $\alpha_1 = \frac{5}{48}$ and $\alpha_4 = \frac{5}{4}$ (see \eqref{eq:arm_exponent} and the paragraph below), we can obtain $L(t) = |t - t_c|^{-\frac{4}{3} + o(1)}$ as $t \to t_c$, and $\theta(t) = (t - t_c)^{\frac{5}{36} + o(1)}$ as $t \searrow t_c$. Hence, $t_{\infty} = t_c + \zeta^{\delta_{\infty} + o(1)}$ and $t_k = t_c + \zeta^{\delta_k + o(1)}$ as $\zeta \searrow 0$, with
\begin{equation} \label{eq:exp_delta}
\delta_k = \frac{36}{55} \cdot \Big( 1 - \Big( \frac{41}{96}\Big)^k \Big) \stackrel[k \to \infty]{\longrightarrow}{} \frac{36}{55} = \delta_{\infty}.
\end{equation}
The corresponding exponents for $m_k$ and $m_{\infty}$ then follow readily:
\begin{equation} \label{eq:asymp_m}
m_k(\zeta) = \zeta^{-\frac{4}{3} \delta_k + o(1)} \quad \text{and} \quad m_{\infty}(\zeta) = \zeta^{-\frac{4}{3} \delta_{\infty} + o(1)} \quad \text{as $\zeta \searrow 0$}
\end{equation}
(in particular, for all $k \geq 0$, $m_{k+1}(\zeta) \gg m_k(\zeta)$).

The following lemma ensures that $L(\hat{t}) \ll L(t)$ for $t$ ``much larger'' than $t_{\infty}$ (more precisely, for $t$ such that $|t - t_c| \gg |t_{\infty} - t_c|$).

\begin{lemma} \label{lem:est_psi}
There exist universal constants $C, \beta > 0$ such that: for all $\zeta \leq 1$ and $t \geq t_{\infty}(\zeta)$,
\begin{equation}
\frac{L(\hat{t})}{L(t)} \leq C \bigg( \frac{L(t)}{L(t_{\infty})} \bigg)^{\beta}.
\end{equation}
\end{lemma}

\begin{proof}[Proof of Lemma \ref{lem:est_psi}]
We can assume \emph{wlog} that $t \leq 10 t_c$. First, we have $L(t)^2 \theta(\hat{t}) \big| \hat{t} - t_c \big| = \zeta^{-1} = L(t_{\infty})^2 \theta(t_{\infty}) \big| t_{\infty} - t_c \big|$ (from \eqref{eq:def_t_hat} and \eqref{eq:def_fixed_pt}), so
\begin{equation} \label{eq:est_psi_pf1}
\frac{L(t)^2}{L(t_{\infty})^2} = \frac{\theta(t_{\infty})}{\theta(\hat{t})} \cdot \frac{\big| t_{\infty} - t_c \big|}{\big| \hat{t} - t_c \big|}.
\end{equation}
On the one hand,
\begin{equation} \label{eq:est_psi_pf2}
\frac{\theta(t_{\infty})}{\theta(\hat{t})} \geq C_1 \frac{\pi_1(L(t_{\infty}))}{\pi_1(L(\hat{t}))} \geq C_2 \pi_1(L(\hat{t}), L(t_{\infty})) \geq C_3 \bigg( \frac{L(\hat{t})}{L(t_{\infty})} \bigg)^{1/2}
\end{equation}
(using successively \eqref{eq:equiv_theta}, \eqref{eq:quasi_mult} and \eqref{eq:1arm}). On the other hand,
\begin{equation} \label{eq:est_psi_pf3}
\frac{\big| t_{\infty} - t_c \big|}{\big| \hat{t} - t_c \big|} \geq C'_1 \frac{L(\hat{t})^2 \pi_4(L(\hat{t}))}{L(t_{\infty})^2 \pi_4(L(t_{\infty}))} \geq C'_2 \frac{L(\hat{t})^2}{L(t_{\infty})^2} \cdot \pi_4(L(\hat{t}), L(t_{\infty}))^{-1} \geq C'_3 \frac{L(\hat{t})^2}{L(t_{\infty})^2} \cdot \bigg( \frac{L(\hat{t})}{L(t_{\infty})} \bigg)^{-1}
\end{equation}
(using \eqref{eq:equiv_L}, \eqref{eq:quasi_mult} and \eqref{eq:4arms}). The desired result then follows (with $\beta = \frac{1}{3}$) by combining \eqref{eq:est_psi_pf1}, \eqref{eq:est_psi_pf2} and \eqref{eq:est_psi_pf3}.
\end{proof}

\subsection{Comparison to percolation with holes} \label{sec:comparison_holes}

In this section, we consider the particular case of the forest fire processes on finite subsets of $\TT$, with homogeneous birth and ignition rates, respectively equal to $1$ and $\zeta$.

The main difficulty in analyzing the behavior of these processes as $t$ approaches $t_c$ is to get a good grip on how fast large connected components appear, and then ``disappear'' due to the fires. For that, in Section \ref{sec:existence_excep_scales}, we first consider the forest fire process where we stop ignition at time $t_c - \ve$ (for some $\ve = \ve(\zeta) \searrow 0$), in boxes with side length $n = n(\zeta) \gg L(t_c - \ve)$. At all later times $t \geq t_c - \ve$, a lower bound for this process is provided by the percolation process with holes (with parameter $p = p(t)$), for some well-chosen $\rho^{(m)}$ and $\pi^{(m)}_v \equiv \pi^{(m)}$ ($v \in V$) that we compute now. Here the parameter is $m = m(\zeta) = L(t_c - \ve(\zeta))$. Later, in Section \ref{sec:existence_excep_scales}, this is applied to the original FFWoR process. Let us also mention that in all our applications, we have $m(\zeta) \ll m_k(\zeta)$ for some $k \geq 2$.

For a subset $C \subseteq V$, let $\rad(C) := \inf \{ n \geq 0 \: : \: C \subseteq \Ball_n\}$ be the radius of $C$ (seen from $0$). Similarly to Section \ref{sec:coupling}, we consider a marked Poisson point process: for all $v \in V$ and $\tau \in \calT_v^{t_c - \ve}$, we assign a mark $\cluster_{v, \tau}$ with distribution $\mu_{v, \tau}$, i.e. the distribution of the cluster of $v$ at time $\tau$. We denote by $\pi^{(m)}$ the probability for a vertex $v \in V$ to be ignited at least once during the interval $[0, t_c - \ve]$, and by $\rho^{(m)}$ the distribution, conditionally on $v$ being ignited in $[0, t_c - \ve]$, of $\max_{\tau \in \calT_v^{t_c - \ve}} \rad (\overline{\cluster_{v, \tau}})$ (these quantities clearly do not depend on $v$).

\begin{lemma} \label{lem:quant_holes}
Let $k \geq 2$, and assume that $m(\zeta) \ll m_k(\zeta)$ as $\zeta \searrow 0$. For any $\upsilon > 0$, let
\begin{equation}
\alpha := \frac{3}{4} \cdot \frac{1}{\delta_{\infty}} + \upsilon \quad \text{and} \quad \beta := \frac{3}{4} \cdot \frac{1}{\delta_k} - \upsilon.
\end{equation}
Then there exist constants $c_1, c_2, c_3 \in (0,+\infty)$ such that for all $m$ sufficiently large, the conditional distribution $\rho^{(m)}$ and the probability $\pi^{(m)}$ satisfy \eqref{eq:assump_holes}:
\begin{equation}
\rho^{(m)} \big( [r,+\infty) \big) \leq c_1 r^{\alpha-2} e^{-c_2 r / m} \: (r \geq 1) \quad \text{and} \quad \pi^{(m)} \leq c_3 m^{- \beta}.
\end{equation}
\end{lemma}

\begin{remark}
Note that $\alpha = \frac{55}{48} + \upsilon \in \big( \frac{3}{4}, 2 \big)$. Moreover, since $\delta_k < \delta_{\infty}$ (the sequence $(\delta_i)_{i \geq 1}$ being strictly increasing), we also have $\beta > \alpha$ as soon as $\upsilon$ is chosen small enough (depending only on $k$).
\end{remark}

\begin{proof}[Proof of Lemma \ref{lem:quant_holes}]
First, note that since $m_k(\zeta) = \zeta^{-\frac{4}{3} \delta_k + o(1)}$ as $\zeta \searrow 0$, we have:
\begin{equation} \label{eq:ineq_N_beta}
\zeta \leq m^{- \frac{3}{4} \cdot \frac{1}{\delta_k} + \upsilon} \quad \text{for all $\zeta$ small enough.}
\end{equation}
Hence, the probability for $v \in V$ to be ignited at least once in $[0, t_c - \ve]$ satisfies
\begin{equation} \label{eq:ineq_pi_holes}
\pi^{(m)}_v = 1 - e^{- \zeta (t_c - \ve)} \asymp \zeta \leq c_3 m^{-\beta}
\end{equation}
for some $c_3 > 0$.

We now estimate the distribution of radii $\rho$. We write, for a time $s \geq 0$ and $r \geq 1$,
$$\bar{\rho}_s(r) := \PP \big( \rad (\overline{\cluster_s(0)}) \geq r \big)$$
(where $\cluster_s(0)$ denotes the occupied cluster of $0$ at time $s$). In the remainder of the proof, we forget about dependences on $m$ in the notations. We start by considering $S$ uniformly distributed in $[0, t_c - \ve]$, and computing
$$\tilde{\rho}([r,+\infty)) := \PP \big( \rad(\overline{\cluster_S(0)}) \geq r \big) \quad (r \geq 1),$$
by distinguishing the two cases $r \geq m$ and $r < m$. Let $J := \big \lfloor \log_2 \big( \frac{t_c}{2 \ve} \big) \big \rfloor \geq 0$, so that $\frac{1}{2} t_c \leq t_c - 2^J \ve < \frac{3}{4} t_c$. For later use, note also that: for all $s < t_c$ and $r \geq L(s)$,
\begin{equation} \label{eq:upper_bound_1arm}
\bar{\rho}_s(r) \leq C_1 \pi_1(L(s)) e^{-C_2 r / L(s)}
\end{equation}
for some universal constants $C_1, C_2 > 0$ (this follows from \eqref{eq:exp_decay} and \eqref{eq:near_critical_arm}).

\underline{Case $r \geq m$:} we have
\begin{align*}
\tilde{\rho}([r,+\infty)) & = \sum_{j=0}^{J-1} \PP \big( S \in (t_c - 2^{j+1} \ve, t_c - 2^j \ve], \: \rad (\overline{\cluster_S(0)}) \geq r \big)\\
& \hspace{1.5cm} + \PP \big( S \in [0, t_c - 2^J \ve], \: \rad (\overline{\cluster_S(0)}) \geq r \big)\\
& \leq \sum_{j=0}^{J-1} \frac{2^j \ve}{t_c - \ve} \cdot \bar{\rho}_{t_c - 2^j \ve}(r) + \frac{t_c - 2^J \ve}{t_c - \ve} \cdot \bar{\rho}_{t_c - 2^J \ve}(r)\\
& \leq C_3 \sum_{j=0}^J (2^j \ve) \pi_1(L(t_c - 2^j \ve)) e^{-C_2 r / L(t_c - 2^j \ve)},
\end{align*}
where we used \eqref{eq:upper_bound_1arm} for the last inequality (using that $r \geq m = L(t_c - \ve) \geq L(t_c - 2^j \ve)$). We have, for some universal constants $\gamma_1, \gamma'_1, \gamma_2, C_4, C_5, C_6 > 0$,
\begin{equation} \label{eq:ineq_rho_L}
C_4 (2^j)^{-\gamma_1} \leq \frac{L(t_c - 2^j \ve)}{m} = \frac{L(t_c - 2^j \ve)}{L(t_c - \ve)} \leq C_5 (2^j)^{-\gamma'_1}
\end{equation}
(using \eqref{eq:equiv_L}, \eqref{eq:quasi_mult} and \eqref{eq:4arms}), and
\begin{equation}
\frac{\pi_1(L(t_c - 2^j \ve))}{\pi_1(m)} \leq C_6 (2^j)^{\gamma_2}
\end{equation}
(combining \eqref{eq:ineq_rho_L} with \eqref{eq:quasi_mult} and \eqref{eq:1arm}). Hence, since $\ve m^2 \pi_4(m) \asymp 1$ (from \eqref{eq:equiv_L}),
\begin{align*}
\tilde{\rho}([r,+\infty)) & \leq C_7 \ve \sum_{j=0}^J (2^j) (2^j)^{\gamma_2} \pi_1(m) e^{-C_8 r (2^j)^{\gamma'_1} / m}\\
& \leq C_9 \frac{\pi_1(m)}{m^2 \pi_4(m)} e^{-C_8 r / m} \cdot \bigg( \sum_{j=0}^J (2^j)^{1+ \gamma_2} e^{-C_8 r ((2^j)^{\gamma'_1}-1) / m} \bigg).
\end{align*}
Since $\sum_{j=0}^J (2^j)^{1+ \gamma_2} e^{-C_8 r ((2^j)^{\gamma'_1}-1) / m} \leq \sum_{j=0}^{\infty} (2^j)^{1+ \gamma_2} e^{-C_8 ((2^j)^{\gamma'_1}-1)} < \infty$ (using $r \geq m$), we obtain: for some $\gamma_3 > 0$,
\begin{equation}
\tilde{\rho}([r,+\infty)) \leq C_{10} \frac{\pi_1(m)}{m^2 \pi_4(m)} e^{-C_8 r / m} \leq C_{11} \frac{\pi_1(r)}{r^2 \pi_4(r)} \bigg( \frac{r}{m} \bigg)^{\gamma_3} e^{-C_8 r / m}
\end{equation}
(using \eqref{eq:quasi_mult}, \eqref{eq:1arm} and \eqref{eq:4arms}), so
\begin{equation} \label{eq:comparison_case1}
\tilde{\rho}([r,+\infty)) \leq C_{12} \frac{\pi_1(r)}{r^2 \pi_4(r)} e^{-C_8 r / 2 m}.
\end{equation}

\underline{Case $r < m$:} let $i \geq 0$ be such that $L(t_c - 2^{i+1} \ve) \leq r < L(t_c - 2^i \ve)$, and assume first that $i \leq J-2$. We have
\begin{align*}
\tilde{\rho}([r,+\infty)) & = \PP \big( S \in (t_c - 2^{i+1} \ve, t_c - \ve], \: \rad (\overline{\cluster_S(0)}) \geq r \big)\\
& \hspace{1.5cm} + \sum_{j = i+1}^{J-1} \PP \big( S \in (t_c - 2^{j+1} \ve, t_c - 2^j \ve], \: \rad (\overline{\cluster_S(0)}) \geq r \big)\\
& \hspace{1.5cm} + \PP \big( S \in [0, t_c - 2^J \ve], \: \rad (\overline{\cluster_S(0)}) \geq r \big)\\
& \leq \frac{2^{i+1} \ve}{t_c - \ve} \cdot \pi_1(r) + \sum_{j = i+1}^{J-1} \frac{2^j \ve}{t_c - \ve} \cdot \bar{\rho}_{t_c - 2^j \ve}(r) + \frac{t_c - 2^J \ve}{t_c - \ve} \cdot \bar{\rho}_{t_c - 2^J \ve}(r)\\
& \leq C'_1 \frac{\pi_1(r)}{r^2 \pi_4(r)} + C'_2 \sum_{j = i+1}^J (2^j \ve) \pi_1(L(t_c - 2^j \ve)) e^{-C_2 r / L(t_c - 2^j \ve)}.
\end{align*}
Now, by a similar computation as before,
\begin{equation}
\sum_{j = i+1}^J (2^j \ve) \pi_1(L(t_c - 2^j \ve)) e^{-C_2 r / L(t_c - 2^j \ve)} \leq C'_3 (2^i \ve) \pi_1(L(t_c - 2^i \ve)) \leq C'_4 \frac{\pi_1(r)}{r^2 \pi_4(r)}
\end{equation}
(using $L(t_c - 2^{i+1} \ve) \leq r < L(t_c - 2^i \ve)$, as well as \eqref{eq:equiv_L}). Hence,
\begin{equation} \label{eq:comparison_case2}
\tilde{\rho}([r,+\infty)) \leq C'_5 \frac{\pi_1(r)}{r^2 \pi_4(r)}.
\end{equation}
Finally, if $i \geq J-1$, then $r \leq L \big( \frac{7}{8} t_c \big)$ and the same conclusion holds (after possibly increasing $C'_5$, if needed).

Hence, combining both cases \eqref{eq:comparison_case1} and \eqref{eq:comparison_case2}, we find that there exist constants $\bar{C}_1$, $\bar{C}_2$, $\bar{C}_3$ such that: for all $r \geq 1$,
\begin{equation} \label{eq:ineq_rho_holes}
\tilde{\rho}([r,+\infty)) \leq \bar{C}_1 \frac{\pi_1(r)}{r^2 \pi_4(r)} e^{-\bar{C}_2 r/m} \leq \bar{C}_3 r^{ - \frac{5}{48} - 2 + \frac{5}{4} + \upsilon} e^{-\bar{C}_2 r/m} = \bar{C}_3 r^{\alpha-2} e^{-\bar{C}_2 r / m},
\end{equation}
with $\alpha = \frac{55}{48} + \upsilon = \frac{3}{4} \cdot \frac{1}{\delta_{\infty}} + \upsilon$, as desired.

Now, for a vertex $v$ which gets ignited at least once before time $t_c - \ve$ (and possibly several times), we consider all the clusters of $v$ generated during the interval $[0, t_c - \ve]$, and denote by $\tilde{r}_v$ the maximum of their radii. We have
\begin{align*}
\PP \big( v \text{ is ignited in } [0, t_c - \ve], \: \tilde{r}_v \geq r \big) & \leq \sum_{k=1}^{\infty} \frac{1}{k!} \big( \zeta (t_c - \ve) \big)^k \cdot k \tilde{\rho}([r,+\infty))\\
& \leq \zeta \cdot t_c \tilde{\rho}([r,+\infty)) \cdot \sum_{k=0}^{\infty} \frac{1}{k!} \big( \zeta t_c \big)^k \leq \zeta \cdot 2 t_c \tilde{\rho}([r,+\infty)),
\end{align*}
which allows us to conclude (using \eqref{eq:ineq_pi_holes} and \eqref{eq:ineq_rho_holes}).
\end{proof}

\section{Existence of exceptional scales for forest fires without recovery} \label{sec:existence_excep_scales}

We now combine the results from previous sections in order to establish properties of the forest fire without recovery (FFWoR) process, run in finite boxes with side length $M = M(\zeta)$. After setting notations in Section \ref{sec:notations_FF}, we consider the two cases $M(\zeta) \asymp m_k(\zeta)$ (Section \ref{sec:case1}), and $m_k(\zeta) \ll M(\zeta) \ll m_{k+1}(\zeta)$ (Section \ref{sec:case2}), for an arbitrary $k \geq 1$. Using the results from Sections \ref{sec:four_arm_stability}, \ref{sec:other_stability} and \ref{sec:application_FF}, we prove the claim (mentioned in the Introduction) that, as $\zeta \searrow 0$, the ``impact'' of fires vanishes in the latter case, but \emph{does not vanish} in the former case. In fact, our results, Theorems \ref{thm:case1} and \ref{thm:case2}, are somewhat stronger and more general than this claim suggests: they not only hold for boxes, but also (and, in some sense, ``uniformly'') for domains whose boundary is a loop in an annulus between two boxes of comparable size. This is not just generalization for its own sake: it is needed to make the proof, which has an iterative flavor, work.

\subsection{Notations} \label{sec:notations_FF}

Recall that we focus on the FFWoR process, on subsets of $\TT = (V, E)$, with birth rate $\equiv 1$ and ignition rate $\equiv \zeta$ (in Section \ref{sec:forest_fires}, we discuss how the proofs might be adapted in the case of forest fire processes with recovery). For a given finite $V_G \subseteq V$, we denote by $\sigma_t = (\sigma_t (v))_{v \in V_G}$ the configuration at time $t \geq 0$ of the forest fire process on $V_G$. We also use the notation $\sigma^{\Lambda}_t$ for the process ``living'' on a subset $\Lambda \subseteq V_G$ (see the discussion above Remark \ref{rem:stoch_domination}). Finally, for $0 \leq s \leq t$, we will also need to consider the forest fire process on $\Lambda$ where ignitions occur only until time $s$ (i.e. nothing happens at the later ignition times $\tau > s$), and we write $\sigma^{\Lambda}_{s, t}$ for the configuration at time $t$ of this process.

For a circuit $\gamma$, we denote by $\calD(\gamma)$ the set of vertices in its interior. For all $k \geq 1$, we write $\ve_k(\zeta) := t_k(\zeta) - t_c$: for future reference, note that
\begin{equation} \label{eq:rel_eps_m}
\zeta \ve_{k-1}(\zeta) \big( m_k(\zeta) \big)^2 \theta \big( t_c + \ve_{k-1}(\zeta) \big) = 1
\end{equation}
(from \eqref{eq:def_t_hat}, since $t_{k-1}(\zeta) = \next_{\zeta}\big( t_k(\zeta) \big)$ and $L(t_k(\zeta)) = m_k(\zeta)$). In this section, we often drop the dependence on $\zeta$ for notational convenience, writing simply $m_k$, $t_k$, $\ve_k$, and so on.

We use later that: for any fixed $\eta \neq 0$,
\begin{equation} \label{eq:L_epsilon}
L( t_c + \eta \ve ) \asymp L( t_c + \ve )
\end{equation}
as $\ve \to 0$ (from \eqref{eq:equiv_L}, \eqref{eq:quasi_mult}, and \eqref{eq:4arms}). In particular, it follows from the definition of $m_k$ that
\begin{equation} \label{eq:L_epsilonk}
L( t_c + \eta \ve_k ) \asymp L( t_c + \ve_k ) = L( t_k ) = m_k \quad \text{as $\zeta \searrow 0$,}
\end{equation}
uniformly in $k$ (i.e. the constants in this asymptotic equivalence only depend on $\eta$).

The first key step of the proof strategy in Section \ref{sec:case1} can be informally described as follows. If we consider the process in a box with side length $M = M(\zeta) \gg \frac{1}{\sqrt{\zeta}}$, we introduce $\ve = \ve(\zeta)$ ($\searrow 0$ as $\zeta \searrow 0$) such that for the underlying percolation process (i.e. without any ignitions at all), we have: at time $t_c + \ve$, with high probability, the box $\Ball_M$ contains a net $\net$ whose cluster $\cluster_{\net}$ has a volume of order $\frac{1}{\ve \zeta}$. We then consider the forest fire process with ignitions ignored after time $t_c - \ve$, for which a lower bound is provided by the results of Sections \ref{sec:BCKS}, \ref{sec:coupling} and \ref{sec:comparison_holes}, so that for this process as well, at time $t_c + \ve$, there exists a net $\net'$ with $|\cluster_{\net'}| \asymp \frac{1}{\ve \zeta}$. Hence, there is a reasonable probability that no vertex of $\cluster_{\net}$ is ignited during $(t_c - \ve, t_c + \ve)$, but some vertex of $\cluster_{\net'}$ is ignited during $(t_c + \ve, t_c + 2 \ve)$. Moreover, we have sufficient control on the size of the island containing $0$ after this burning, which allows us to repeat this step iteratively.

Section \ref{sec:case2} uses similar ideas, but the situation is somewhat more complicated. In particular, it requires the use of time intervals of the form $(t_c - \eta \ve, t_c + \eta \ve)$ and $(t_c + \eta \ve, t_c + \lambda \ve)$, for some suitable $\eta, \lambda > 0$.

We want to stress that the proofs of Theorems \ref{thm:case1} and \ref{thm:case2} also yield some information about the size of the final cluster of the origin. This cluster has typically a diameter of order $1$ or $\frac{1}{\sqrt{\zeta}}$ in the first case ($M(\zeta) \asymp m_k(\zeta)$), and a diameter $\gg 1$ but $\ll \frac{1}{\sqrt{\zeta}}$ in the second case ($m_k(\zeta) \ll M(\zeta) \ll m_{k+1}(\zeta)$).

\subsection{Case $M(\zeta) \asymp m_k(\zeta)$} \label{sec:case1}

For $\zeta \leq 1$ and $0 \leq \ul t < \ol t \leq \infty$, we introduce the event $\Gamma_{\zeta, \ul t, \ol t}(n_1, n_2) := \{$for all circuits $\gamma$ in the annulus $\Ann_{n_1, n_2}$, in the forest fire process with ignition rate $\zeta$ in the domain $\calD(\gamma)$, $0$ burns during the time interval $[\ul t, \ol t] \}$ ($0 \leq n_1 < n_2$). The goal of this section is to establish the following result for forest fires in domains with ``size'' comparable to some exceptional scale $m_k$ ($k \geq 1$).

\begin{theorem} \label{thm:case1}
Let $t_c < \ul t < \ol t < \infty$. For all $k \geq 1$ and all $0 < C_1 < C_2$,
\begin{equation}
\liminf_{\zeta \searrow 0} \PP \big( \Gamma_{\zeta, \ul t, \ol t}(C_1 m_k(\zeta), C_2 m_k(\zeta)) \big) > 0.
\end{equation}
\end{theorem}

\begin{proof}[Proof of Theorem \ref{thm:case1}]
The constructions that we use turn out to be quite convoluted, due to dependences between successive scales, that need to be taken care of. We first give a proof for the case $k = 2$, after which we point out how to handle a general $k \geq 3$. We define the following six events (some of them depicted on Figure \ref{fig:case1}), for a well-chosen constant $C_3$ to be determined at the end of the proof (see \eqref{eq:case1_end3}). The superscript ``$(2)$'' in the notation of these events refers to the fact that we are considering the case $k=2$. For simplicity, we assume that $C_1 \geq 1$ (trivial adaptations of the argument are needed if $C_1 < 1)$.

\begin{figure}[t]
\begin{center}

\includegraphics[width=0.65\textwidth]{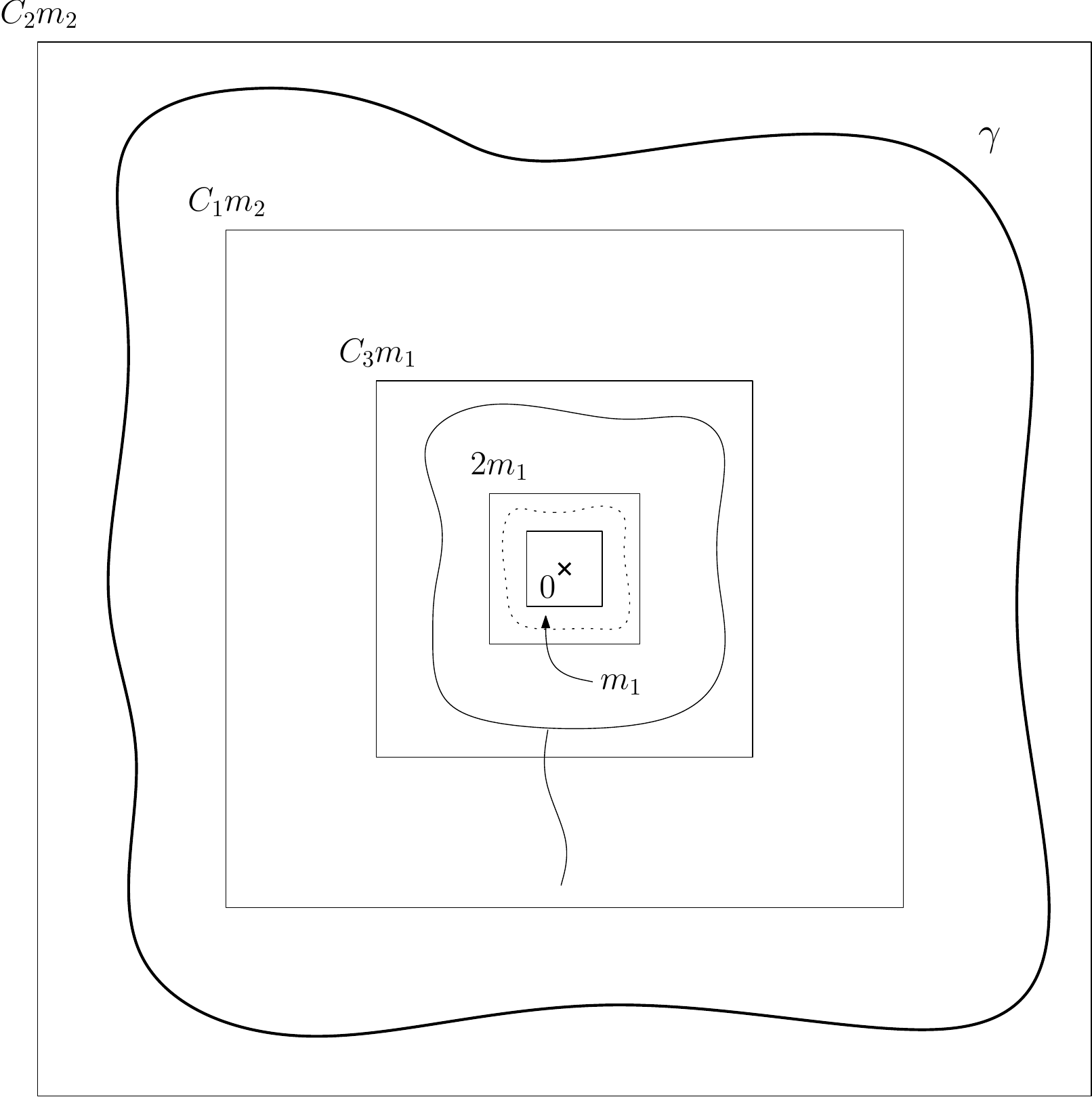}
\caption{\label{fig:case1} This figure depicts the events used to prove Theorem \ref{thm:case1} (in the case $k=2$). The dotted circuit is $(t_c + 2 \ve_1)$-vacant, while the circuit and the path in solid lines are occupied in the configuration $\min_{\gamma} \sigma^{\calD(\gamma)}_{t_c - \ve_1, t_c + \ve_1}$.}

\end{center}
\end{figure}

\begin{itemize}

\item[(i)] $\NET^{(2)} = \NET^{(2)}_{C_2 m_2} (C_1 m_2)$ is the event that the configuration $\big( \min_{\gamma} \sigma^{\calD(\gamma)}_{t_c - \ve_1, t_c + \ve_1} (v) \big)_{v \in \Ball_{C_1 m_2}}$, where the minimum is taken over all circuits $\gamma$ in the annulus $\Ann_{C_1 m_2, C_2 m_2}$,
\begin{itemize}
\item has a net $\net'$ with mesh $(C_1 m_2 \cdot m_1)^{1/2}$,
\item and $\big| \cluster_{\net'} \cap \Ann_{\frac{C_1}{2} m_2, C_1 m_2} \big| \geq \frac{1}{2} \cdot \big| \Ann_{\frac{C_1}{2} m_2, C_1 m_2} \big| \theta(t_c + \ve_1)$.
\end{itemize}
Using the comparison to percolation with holes provided by Lemmas \ref{lem:stoch_domination} and \ref{lem:quant_holes}, it follows from Proposition \ref{prop:largest_cluster} (and Remark \ref{rem:BCKS_annulus}), with $m = m_1$, $n_m = C_1 m_2$, and $p_m = p(t_c + \ve_1)$, that
\begin{equation} \label{eq:claim_NET2}
\PP \big( \NET^{(2)} \big) \stackrel[\zeta \searrow 0]{}{\longrightarrow} 1.
\end{equation}
Indeed, note that $L(p_m) = L(t_c + \ve_1) = m_1$, and also that, by \eqref{eq:exp_delta} and \eqref{eq:asymp_m}, $n_m = C_1 (m_1)^{\delta_2/\delta_1 + o(1)}$, where $\delta_2/\delta_1 > 1$. Finally, we want to emphasize that our application of Lemma \ref{lem:stoch_domination} above (and later in this section, though it will not be mentioned explicitly) involves a more general version of this lemma, pointed out in Remarks \ref{rem:stoch_domination_inhomogeneous} and \ref{rem:stoch_domination}.

\item[(ii)] $\NETB^{(2)} = \NETB^{(2)}(\ve_1; C_2 m_2)$ (where B stands for ``Bernoulli'') is the event that the largest $(t_c + \ve_1)$-occupied cluster (i.e. for the underlying percolation process) $\cluster^B$ in $\Ball_{C_2 m_2}$ has a volume $|\cluster^B| \leq 2 \cdot \big| \Ball_{C_2 m_2} \big| \theta(t_c + \ve_1)$, and contains a net $\net^B$ with mesh $(C_2 m_2 \cdot m_1)^{1/2}$. Note that the cluster $\cluster^B$ in this definition automatically contains the net $\net'$ in the definition of $\NET^{(2)}$. Since $L(t_c + \ve_1) = m_1 \ll C_2 m_2$, the standard volume estimates \eqref{eq:largest_cluster} for ordinary Bernoulli percolation (see also Remark \ref{rem:BCKS}) give
\begin{equation} \label{eq:claim_NETB2}
\PP \big( \NETB^{(2)} \big) \stackrel[\zeta \searrow 0]{}{\longrightarrow} 1.
\end{equation}

\item[(iii)] $\OCP^{(2)} = \OCP^{(2)}(2 m_1, C_3 m_1; m_2)$ (where the name stands for ``Occupied Circuit and Path'') is the event that the configuration $\big( \min_{\gamma} \sigma^{\calD(\gamma)}_{t_c - \ve_1, t_c + \ve_1} (v) \big)_{v \in \Ball_{C_1 m_2}}$, where the minimum is taken over all $\gamma$ in $\Ann_{C_1 m_2, C_2 m_2}$ (as in the definition of $\NET^{(2)}$),
\begin{itemize}
\item has an occupied circuit in $\Ann_{2 m_1, C_3 m_1}$,
\item which is connected by an occupied path to $\partial \Ball_{m_2}$.
\end{itemize}
Note that the occupied path in this definition has to intersect the net $\net'$ in the definition of $\NET^{(2)}$: indeed, $\net'$ has a mesh $\asymp (m_1 m_2)^{1/2}$, which is $\gg m_1$ and $\ll m_2$. We claim the following:
\begin{equation} \label{eq:claim_OCP2}
\text{for all $\delta > 0$, we have that for all $C_3$ large enough, } \liminf_{\zeta \searrow 0} \PP \big(\OCP^{(2)} \big) \geq 1 - \delta.
\end{equation}
Indeed, this follows from Proposition \ref{prop:crossing} combined with \eqref{eq:RSW}, and Proposition \ref{prop:exp_decay} (together with the construction of Figure \ref{fig:overlapping_rectangles}), using again $L(t_c + \ve_1) = m_1$, and the lower bound produced by Lemma \ref{lem:stoch_domination} and Lemma \ref{lem:quant_holes}.

\item[(iv)] $\VC^{(2)} = \VC^{(2)}(2 \ve_1; m_1, 2 m_1) := \big\{$there exists a $(t_c + 2 \ve_1)$-vacant circuit in $\Ann_{m_1, 2 m_1} \big\}$ (where the name stands for ``Vacant Circuit''). We claim that
\begin{equation} \label{eq:claim_VC2}
\PP \big( \VC^{(2)} \big) \geq C > 0,
\end{equation}
for some ``universal'' constant $C$, which does not depend on $C_1$, $C_2$ or $C_3$. Indeed, this is an immediate consequence of \eqref{eq:RSW}, and the fact that $L( t_c + 2 \ve_1 ) \asymp m_1$ (from \eqref{eq:L_epsilonk}).

\item[(v)] $\I^{(2)} = \I^{(2)} \big( (\ve_1, 2 \ve_1); \frac{C_1}{2} m_2, C_1 m_2 \big) := \big\{$some vertex in $\cluster_{\net'} \cap \Ann_{\frac{C_1}{2} m_2, C_1 m_2}$ is ignited during the time interval $(t_c + \ve_1, t_c + 2 \ve_1) \big\}$, where $\net'$ is as in the definition of $\NET^{(2)}$ (the name stands for ``Ignition''). Note that
\begin{equation} \label{eq:claim_I2}
\PP \big( \I^{(2)} \: | \: \NET^{(2)} \big) \geq 1 - e^{- \zeta \cdot \ve_1 \cdot \frac{1}{2} \big| \Ann_{\frac{C_1}{2} m_2, C_1 m_2} \big| \theta(t_c + \ve_1)} \geq C' > 0,
\end{equation}
for some constant $C' = C'(C_1)$ which depends only on $C_1$, using \eqref{eq:rel_eps_m}.

\item[(vi)] $\NI^{(2)} = \NI^{(2)}( (- \ve_1, \ve_1); C_2 m_2) := \big\{$no vertex of $\cluster^B$ gets ignited in the time interval $(t_c - \ve_1, t_c + \ve_1) \big\}$, where $\cluster^B$ is from the definition of $\NETB^{(2)}$ (the name stands for ``No Ignition''). We have
\begin{equation} \label{eq:claim_NI2}
\PP \big( \NI^{(2)} \: | \: \NETB^{(2)} \big) \geq e^{- \zeta \cdot 2 \ve_1 \cdot 2 \big| \Ball_{C_2 m_2} \big| \theta(t_c + \ve_1)} \geq C'' > 0,
\end{equation}
for some constant $C'' = C''(C_2)$ depending only on $C_2$ (using again \eqref{eq:rel_eps_m}).

\end{itemize}

Now, note that if all the six events (i)-(vi) above hold, then, no matter where $\gamma$ is located exactly, the forest fire process in $\calD(\gamma)$ has the property that $\cluster_{\net'}$ burns in the time interval $(t_c + \ve_1, t_c + 2 \ve_1)$, and leaves $0$ in an ``island'', whose boundary is some circuit in $\Ann_{m_1, C_3 m_1}$. Hence,
\begin{align}
\Gamma_{\zeta, \ul t, \ol t} & (C_1 m_2, C_2 m_2) \nonumber\\
& \supseteq \NET^{(2)}_{C_2 m_2}(C_1 m_2) \cap \NETB^{(2)}(\ve_1; C_2 m_2) \cap \OCP^{(2)}(2 m_1, C_3 m_1; m_2) \cap \VC^{(2)}(2 \ve_1; m_1, 2 m_1) \nonumber \\
& \hspace{1cm} \cap \NI^{(2)}( (- \ve_1, \ve_1); C_2 m_2) \cap \I^{(2)} \Big( (\ve_1, 2 \ve_1); \frac{C_1}{2} m_2, C_1 m_2 \Big) \cap \Gamma_{\zeta, \ul t, \ol t}(m_1, C_3 m_1). \label{eq:case1_lemma1}
\end{align}
In order to avoid ``interferences'' with events at level $m_1$ (i.e. a certain dependence between the two successive scales), we will later write
\begin{equation} \label{eq:interference_NI}
\NI^{(2)}( (- \ve_1, \ve_1); C_2 m_2) = \NI^{(2)}( (- \ve_1, \ve_1); C_3 m_1, C_2 m_2) \cap \NI^{(2)}( (- \ve_1, \ve_1); C_3 m_1),
\end{equation}
where the first event in the right-hand side involves only vertices in the annulus $\Ann_{C_3 m_1, C_2 m_2}$. For future use, note that the second event satisfies
\begin{equation}
\PP \big( \NI^{(2)}( (- \ve_1, \ve_1); C_3 m_1) \big) \stackrel[\zeta \searrow 0]{}{\longrightarrow} 1.
\end{equation}

We now investigate the event $\Gamma_{\zeta, \ul t, \ol t}(m_1, C_3 m_1)$, for which we again need to define several events, with similar names as before, but now with superscript ``$(1)$'' (some of them are illustrated in Figure \ref{fig:case1_2}).

\begin{figure}
\begin{center}

\includegraphics[width=0.45\textwidth]{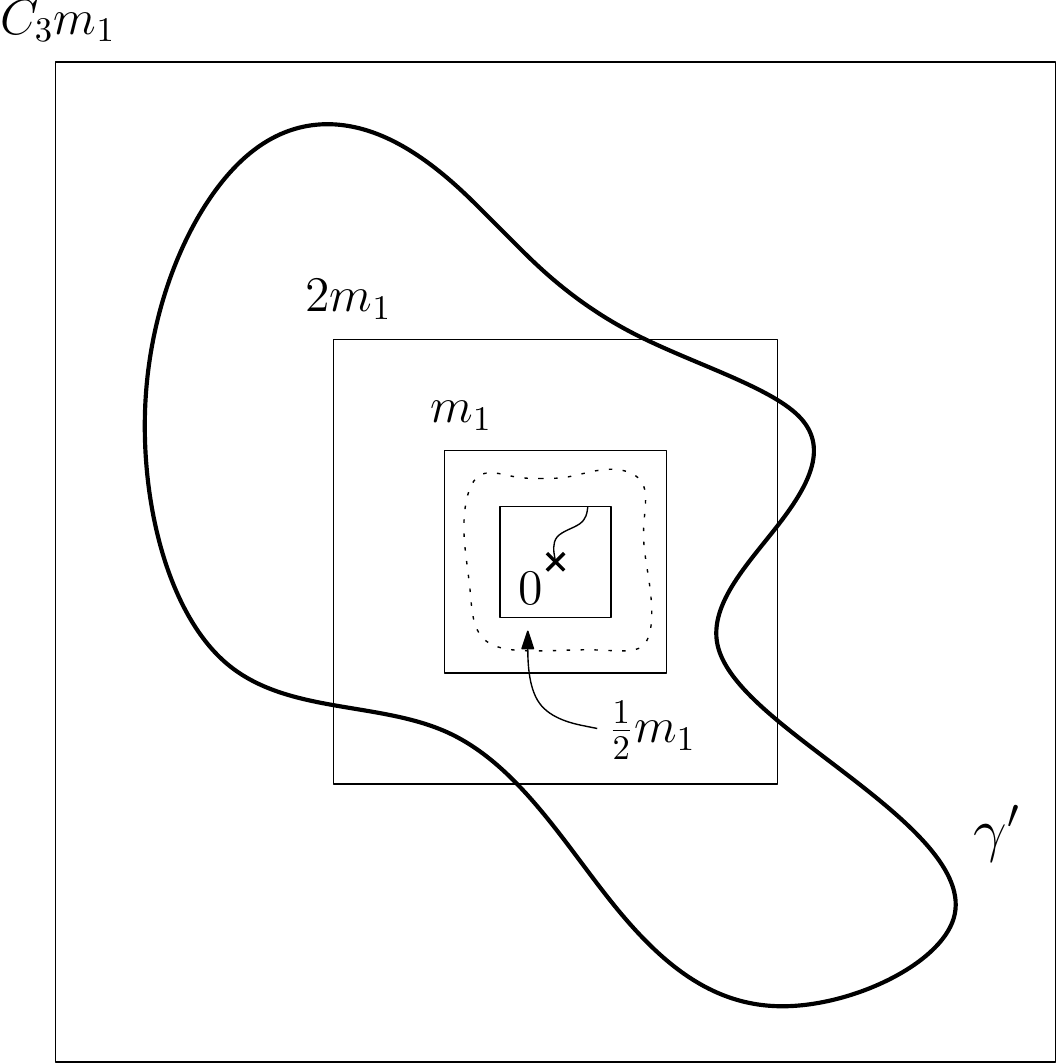}
\caption{\label{fig:case1_2} This figure depicts the events used to prove Theorem \ref{thm:case1} at scale $m_1$. The dotted circuit is $(t_c + 2 \ve_1)$-vacant, while the solid path starting from $0$ is $\ul t$-occupied.}

\end{center}
\end{figure}

\begin{itemize}

\item[(i)'] $\VC^{(1)} = \VC^{(1)} \big( 2 \ve_1; \frac{1}{2} m_1, m_1 \big) := \big\{$there exists a $(t_c + 2 \ve_1)$-vacant circuit in $\Ann_{\frac{1}{2} m_1, m_1} \big\}$. For the same reasons as for \eqref{eq:claim_VC2}, we have
\begin{equation} \label{eq:VC1}
\PP \big( \VC^{(1)} \big) \geq \tilde{C} > 0,
\end{equation}
where $\tilde{C}$ is a universal constant, which does not depend on $C_1$, $C_2$ or $C_3$.

\item[(ii)'] $\NI^{(1)} = \NI^{(1)}( 2 \ve_1; m_1) := \big\{$no vertex of $\Ball_{m_1}$ is hit by lightning before time $t_c + 2 \ve_1 \big\}$. Since $m_1 \asymp \frac{1}{\sqrt{\zeta}}$ as $\zeta \searrow 0$ (see below \eqref{eq:rel_except_scales}), and $\ve_1 \to 0$, we have
\begin{equation} \label{eq:claim_NI1}
\PP \big( \NI^{(1)} \big) \geq \tilde{C}' > 0,
\end{equation}
for some universal constant $\tilde{C}'$, which does not depend on $C_1$, $C_2$ or $C_3$.

\item[(iii)'] $\ol{\NI}^{(1)} = \ol{\NI}^{(1)} ( ( 2 \ve_1, \ul t - t_c ); C_3 m_1 ) := \big\{$no vertex of $\Ball_{C_3 m_1}$ is hit by lightning in the time interval $( t_c + 2 \ve_1, \ul t ) \big\}$. Note that, using again $m_1 \asymp \frac{1}{\sqrt{\zeta}}$,
\begin{equation} \label{eq:claim_barNI1}
\PP \big( \ol{\NI}^{(1)} \big) \geq e^{- \zeta C_3^2 m_1^2 \cdot \ul t} \geq e^{- \lambda C_3^2 \cdot \ul t}
\end{equation}
for some universal constant $\lambda > 0$ (this lower bound becomes very small when $C_3$ is large, but it is not a problem since we will later ``factorize it out'').

\end{itemize}
These first three events together, i.e. $\VC^{(1)} \cap \NI^{(1)} \cap \ol{\NI}^{(1)}$, ensure that the configuration in $\Ball_{\frac{1}{2} m_1}$ ``looks like'' ordinary Bernoulli percolation at time $\ul t$. We define further events.

\begin{itemize}

\item[(iv)'] $\NETB^{(1)} = \NETB^{(1)} \big( \ul t - t_c ; \frac{1}{2} m_1 \big) := \big\{$there exists a $\ul t$-occupied cluster in $\Ball_{\frac{1}{2} m_1}$ with volume $\geq \frac{1}{2} \cdot \big| \Ball_{\frac{1}{2} m_1} \big| \theta ( \ul t )$, and containing a net $\net^{(1)}$ with mesh $\asymp \sqrt{m_1} \big\}$. Again, it follows from the volume estimates \eqref{eq:largest_cluster}, combined with the fact that $m_1 \asymp \frac{1}{\sqrt{\zeta}} \to \infty$ as $\zeta \searrow 0$, that
\begin{equation}
\PP \big( \NETB^{(1)} \big) \stackrel[\zeta \searrow 0]{}{\longrightarrow} 1.
\end{equation}

\item[(v)'] $\OP^{(1)} = \OP^{(1)} \big( \ul t - t_c; \frac{1}{2} m_1 \big) := \big\{$there is a $\ul t$-occupied path from $0$ to $\partial \Ball_{\frac{1}{2} m_1} \big\}$. Clearly,
\begin{equation}
\PP \big( \OP^{(1)} \big) \geq \theta ( \ul t ).
\end{equation}

\item[(vi)'] $\I^{(1)} = \I^{(1)} \big( ( \ul t - t_c, \ol t - t_c ); \frac{1}{2} m_1 \big) := \big\{$some vertex in the cluster of $\net^{(1)}$ gets ignited in the time interval $( \ul t, \ol t ) \big\}$ (where $\net^{(1)}$ is from the definition of $\NETB^{(1)}$). Note that (for some universal constants $\lambda', \lambda'' > 0$)
\begin{equation} \label{eq:claim_I1}
\PP \big( \I^{(1)} \: | \: \NETB^{(1)} \big) \geq 1 - e^{- \zeta \cdot \lambda' m_1^2 \cdot \theta ( \ul t ) \cdot (\ol t - \ul t)} \geq 1 - e^{- \lambda'' \cdot \theta ( \ul t ) \cdot (\ol t - \ul t)} > 0,
\end{equation}
which does not depend on $C_1$, $C_2$ or $C_3$ (we used again $m_1 \asymp \frac{1}{\sqrt{\zeta}}$).

\end{itemize}

If these events (i)'-(vi)' hold, then, in the forest fire process in $\calD(\gamma')$, $0$ burns in the time interval $[\ul t, \ol t]$ (no matter where $\gamma'$ precisely is). We deduce that
\begin{align}
\Gamma_{\zeta, \ul t, \ol t} & (m_1, C_3 m_1) \nonumber\\
& \supseteq \VC^{(1)} \Big( 2 \ve_1; \frac{1}{2} m_1, m_1 \Big) \cap \NI^{(1)}( 2 \ve_1; m_1) \cap \ol{\NI}^{(1)} ( ( 2 \ve_1, \ul t - t_c ); C_3 m_1 ) \nonumber \\
& \hspace{1cm} \cap \NETB^{(1)} \Big( \ul t - t_c ; \frac{1}{2} m_1 \Big) \cap \OP^{(1)} \Big( \ul t - t_c; \frac{1}{2} m_1 \Big) \cap \I^{(1)} \Big( ( \ul t - t_c, \ol t - t_c ); \frac{1}{2} m_1 \Big). \label{eq:case1_lemma2}
\end{align}

Note that $\NI^{(1)}$ ``interferes'' with some of the events at scale $m_2$: with $\NET^{(2)}$, which does not matter in the computation below, since $\PP \big( \NET^{(2)} \big) \to 1$ as $\zeta \searrow 0$ (from \eqref{eq:claim_NET2}), but also with $\OCP^{(2)}$. We take care of this issue by writing
\begin{equation} \label{eq:interference_OCP}
\OCP^{(2)} \cap \NI^{(1)} = \ol{\OCP}^{(2)} \cap \NI^{(1)},
\end{equation}
for a modified event $\ol{\OCP}^{(2)}$ defined exactly as $\OCP^{(2)}$, but with respect to the forest fire process where no ignitions occur in the sub-region $\Ball_{m_1}$. Note that from Lemma \ref{lem:stoch_domination}, \eqref{eq:claim_OCP2} holds for this event $\ol{\OCP}^{(2)}$ as well.

We now combine the two inclusions \eqref{eq:case1_lemma1} and \eqref{eq:case1_lemma2}, and to take care of dependences between scales, we modify some of the events as explained (see \eqref{eq:interference_NI} and \eqref{eq:interference_OCP}). By using that the probabilities of several events tend to $1$ as $\zeta \searrow 0$, we obtain
\begin{align}
& \liminf_{\zeta \searrow 0} \PP \big( \Gamma_{\zeta, \ul t, \ol t} (C_1 m_2, C_2 m_2) \big) \nonumber \\
& \quad \geq \liminf_{\zeta \searrow 0} \PP \bigg[ \ol{\OCP}^{(2)}(2 m_1, C_3 m_1; m_2) \cap \VC^{(2)}(2 \ve_1; m_1, 2 m_1) \cap \NI^{(2)}( (- \ve_1, \ve_1); C_3 m_1, C_2 m_2) \nonumber \\
& \qquad \qquad \cap \I^{(2)} \Big( (\ve_1, 2 \ve_1); \frac{C_1}{2} m_2, C_1 m_2 \Big) \cap \VC^{(1)} \Big( 2 \ve_1; \frac{1}{2} m_1, m_1 \Big) \cap \NI^{(1)}( 2 \ve_1; m_1) \nonumber \\
& \qquad \qquad \cap \ol{\NI}^{(1)} ( ( 2 \ve_1, \ul t - t_c ); C_3 m_1 ) \cap \OP^{(1)} \Big( \ul t - t_c; \frac{1}{2} m_1 \Big) \cap \I^{(1)} \Big( ( \ul t - t_c, \ol t - t_c ); \frac{1}{2} m_1 \Big) \bigg]. \label{eq:liminf_case1}
\end{align}
Now, we use that although the $\I$ and $\NI$ events are, strictly speaking, not independent of each other (nor on the other events), their conditional probabilities, given other events, are bounded from below by some positive constants depending only on (at most) $C_1$, $C_2$ and $C_3$. This gives, by applying \eqref{eq:claim_NI2}, \eqref{eq:claim_I2}, \eqref{eq:claim_NI1}, \eqref{eq:claim_barNI1}, and \eqref{eq:claim_I1} to \eqref{eq:liminf_case1}, that for some $\bar{C} = \bar{C}(C_1, C_2, C_3) > 0$,
\begin{align}
\liminf_{\zeta \searrow 0} & \PP \big( \Gamma_{\zeta, \ul t, \ol t} (C_1 m_2, C_2 m_2) \big) \nonumber \\
& \geq \bar{C}(C_1, C_2, C_3) \cdot \liminf_{\zeta \searrow 0} \PP \bigg[ \ol{\OCP}^{(2)}(2 m_1, C_3 m_1; m_2) \cap \VC^{(2)}(2 \ve_1; m_1, 2 m_1) \nonumber \\
& \hspace{5cm} \cap \VC^{(1)} \Big( 2 \ve_1; \frac{1}{2} m_1, m_1 \Big) \cap \OP^{(1)} \Big( \ul t - t_c; \frac{1}{2} m_1 \Big) \bigg]. \label{eq:case1_end1}
\end{align}
Finally, we note that the events $\VC^{(2)}$, $\VC^{(1)}$, and $\OP^{(1)}$ are independent of each other, and that their probabilities do not depend on $C_3$. Hence,
\begin{equation} \label{eq:case1_end2}
\liminf_{\zeta \searrow 0} \PP \bigg[ \VC^{(2)}(2 \ve_1; m_1, 2 m_1) \cap \VC^{(1)} \Big( 2 \ve_1; \frac{1}{2} m_1, m_1 \Big) \cap \OP^{(1)} \Big( \ul t - t_c; \frac{1}{2} m_1 \Big) \bigg] \geq \bar{C}' > 0,
\end{equation}
for some constant $\bar{C}'$ that does not depend on $C_3$. From \eqref{eq:claim_OCP2} (and the remark following the definition of $\ol{\OCP}^{(2)}$, below \eqref{eq:interference_OCP}), we can take $C_3$ large enough so that
\begin{equation} \label{eq:case1_end3}
\liminf_{\zeta \searrow 0} \PP \big( \ol{\OCP}^{(2)}(2 m_1, C_3 m_1; m_2) \big) \geq 1 - \frac{\bar{C}'}{2}.
\end{equation}
By combining \eqref{eq:case1_end2} and \eqref{eq:case1_end3} with \eqref{eq:case1_end1}, we get
\begin{equation}
\liminf_{\zeta \searrow 0} \PP \big( \Gamma_{\zeta, \ul t, \ol t} (C_1 m_2, C_2 m_2) \big) \geq \bar{C}(C_1, C_2, C_3) \cdot \frac{\bar{C}'}{2}.
\end{equation}
This completes the proof of Theorem \ref{thm:case1} in the case $k = 2$.

We now give an outline of the proof for a general $k$. For $k \geq 3$, we define analogous events $\NET^{(k)}$, $\NETB^{(k)}$, $\OCP^{(k)}$, $\VC^{(k)}$, $\I^{(k)}$ and $\NI^{(k)}$. For the same reasons as in the case $k=2$, the corresponding claims \eqref{eq:claim_NET2}, \eqref{eq:claim_NETB2}, \eqref{eq:claim_OCP2}, \eqref{eq:claim_VC2}, \eqref{eq:claim_I2} and \eqref{eq:claim_NI2} are also satisfied by these events. Again by the same arguments as for $k=2$, an inclusion similar to \eqref{eq:case1_lemma1} holds:
\begin{align}
\Gamma_{\zeta, \ul t, \ol t} (C_1 m_k, C_2 m_k) 
& \supseteq \NET^{(k)}_{C_2 m_k}(C_1 m_k) \cap \NETB^{(k)}(\ve_{k-1}; C_2 m_k) \cap \OCP^{(k)}(2 m_{k-1}, C_3 m_{k-1}; m_k) \nonumber \\
& \hspace{1cm} \cap \VC^{(k)}(2 \ve_{k-1}; m_{k-1}, 2 m_{k-1}) \cap \NI^{(k)}( (- \ve_{k-1}, \ve_{k-1}); C_2 m_k) \nonumber \\
& \hspace{1cm} \cap \I^{(k)} \Big( (\ve_{k-1}, 2 \ve_{k-1}); \frac{C_1}{2} m_k, C_1 m_k \Big) \cap \Gamma_{\zeta, \ul t, \ol t}(m_{k-1}, C_3 m_{k-1}). \label{eq:case1_lemma1_bis}
\end{align}
We then iterate this, using \eqref{eq:case1_lemma2} for the case $k=1$ as before. We also rewrite the events $\NI^{(j)}$ ($2 \leq j \leq k$) similarly as in \eqref{eq:interference_NI}, to avoid dependences. Using that the events $\NET^{(j)}$ and $\NETB^{(j)}$ have probabilities tending to $1$, as $\zeta \searrow 0$, yields
\begin{align}
\liminf_{\zeta \searrow 0} \PP \big( \Gamma_{\zeta, \ul t, \ol t} & (C_1 m_k, C_2 m_k) \big) \geq \liminf_{\zeta \searrow 0} \PP \bigg[ \bigcap_{j=1}^k \big( \text{``} \I^{(j)} \text{ and (rewritten) } \NI^{(j)} \text{ events''} \big) \nonumber \\
& \cap \bigcap_{j=2}^k \big( \text{``} \OCP^{(j)} \text{ event''} \big) \cap \bigcap_{j=1}^k \big( \text{``} \VC^{(j)} \text{ event''} \big) \cap \OP^{(1)} \Big( \ul t - t_c; \frac{1}{2} m_1 \Big) \bigg]. \label{eq:liminf_case1_bis}
\end{align}
We now slightly modify the $\OCP^{(j)}$ events, i.e. we replace them by $\ol{\OCP}^{(j)}$ events as in \eqref{eq:interference_OCP}. This allows us, as before (see the explanation just after \eqref{eq:liminf_case1}), to ``split out'' the product of (the lower bounds for) the $\I^{(j)}$ and $\NI^{(j)}$ events. This product is again bounded from below by $\bar{C}$, for some $\bar{C} = \bar{C}(C_1, C_2, C_3) > 0$, so we get, similarly to \eqref{eq:case1_end1}:
\begin{align}
\liminf_{\zeta \searrow 0} \PP \big( \Gamma_{\zeta, \ul t, \ol t} (C_1 m_k, C_2 m_k) \big) \geq \bar{C}(C_1, & C_2, C_3) \cdot \liminf_{\zeta \searrow 0} \PP \bigg[ \bigcap_{j=2}^k \big( \text{``} \ol{\OCP}^{(j)} \text{ event''} \big) \nonumber \\
& \cap \bigcap_{j=1}^k \big( \text{``} \VC^{(j)} \text{ event''} \big) \cap \OP^{(1)} \Big( \ul t - t_c; \frac{1}{2} m_1 \Big) \bigg]. \label{eq:case1_end1_bis}
\end{align}
Since the $\VC^{(j)}$ events and the event $\OP^{(1)} \big( \ul t - t_c; \frac{1}{2} m_1 \big)$ are independent, and have probabilities independent of $C_3$, we can write
\begin{equation} \label{eq:case1_end2_bis}
\liminf_{\zeta \searrow 0} \PP \bigg[ \bigcap_{j=1}^k \big( \text{``} \VC^{(j)} \text{ event''} \big) \cap \OP^{(1)} \Big( \ul t - t_c; \frac{1}{2} m_1 \Big) \bigg] \geq \bar{C}' > 0,
\end{equation}
for some constant $\bar{C}'$ independent of $C_3$. Finally, we take $C_3$ so large that
\begin{equation} \label{eq:case1_end3_bis}
\liminf_{\zeta \searrow 0} \PP \bigg[ \bigcap_{j=2}^k \big( \text{``} \ol{\OCP}^{(j)} \text{ event''} \big) \bigg] \geq 1 - \frac{\bar{C}'}{2}.
\end{equation}
We thus obtain from \eqref{eq:case1_end1_bis}, \eqref{eq:case1_end2_bis}, and \eqref{eq:case1_end3_bis} that
\begin{equation}
\liminf_{\zeta \searrow 0} \PP \big( \Gamma_{\zeta, \ul t, \ol t} (C_1 m_k, C_2 m_k) \big) \geq \bar{C}(C_1, C_2, C_3) \cdot \frac{\bar{C}'}{2},
\end{equation}
which completes the proof of Theorem \ref{thm:case1}.
\end{proof}

\subsection{Case $m_k(\zeta) \ll M(\zeta) \ll m_{k+1}(\zeta)$} \label{sec:case2}

For $\zeta \leq 1$ and $0 < T < \infty$, we introduce the event $\tilde{\Gamma}_{\zeta, T}(n_1, n_2) := \{$there exists a circuit $\gamma$ in the annulus $\Ann_{n_1, n_2}$ such that in the forest fire process with ignition rate $\zeta$ in the domain $\calD(\gamma)$, $0$ burns before time $T \}$ ($0 \leq n_1 < n_2$). We now prove the result below, for the process in domains with ``size'' far away from the exceptional scales (more precisely, between two successive exceptional scales, but far away from both, asymptotically).

\begin{theorem} \label{thm:case2}
Let $k \geq 0$, $\delta > 0$, $0 < C_1 < C_2$, and $T > 0$. There exists $C = C(k, \delta, C_1, C_2, T) \geq 1$ such that: for every function $M(\zeta)$ satisfying
\begin{equation}
C m_k(\zeta) \leq C_1 M(\zeta) < C_2 M(\zeta) \leq C^{-1} m_{k+1}(\zeta)
\end{equation}
for all sufficiently small $\zeta$, we have
\begin{equation}
\limsup_{\zeta \searrow 0} \PP \big( \tilde{\Gamma}_{\zeta, T}(C_1 M(\zeta), C_2 M(\zeta)) \big) \leq \delta.
\end{equation}
\end{theorem}

In this section, we adopt the following notation. For a given $M \geq 1$, we define $t_c + \tilde{\ve}$ to be the ``typical'' time of the first macroscopic burning in $\Ball_{M/2}$, and $\tilde{M}$ to be the characteristic length at this time. More precisely, with the notations from Section \ref{sec:def_exceptional_scales} (seeing $L$ as a bijection on $(t_c,\infty)$),
\begin{equation} \label{eq:def_epsilon_tilde}
t_c + \tilde{\ve} := \next_{\zeta} ( L^{-1}(M) ) \quad \text{and} \quad \tilde{M} := L( t_c + \tilde{\ve} ).
\end{equation}
Similarly to \eqref{eq:rel_except_scales}, we deduce from \eqref{eq:rel_t_t_hat} that
\begin{equation} \label{eq:equiv_mtilde}
M^2 \asymp \zeta^{-1} \tilde{M}^2 \frac{\pi_4(\tilde{M})}{\pi_1(\tilde{M})}.
\end{equation}

\begin{proof}[Proof of Theorem \ref{thm:case2}]
We proceed by induction over $k$. We first consider the case $k=0$. Recall that $m_0 = L(t_0) = L(2 t_c)$ is a constant, and $m_1 = L(t_1) \asymp \frac{1}{\sqrt{\zeta}}$ (see the sentence below \eqref{eq:rel_except_scales}). Let $\delta$, $C_1$, $C_2$, and $T$ be given as in the statement. Let $M(\zeta)$ satisfying
\begin{equation}
C \leq C_1 M(\zeta) < C_2 M(\zeta) \leq C^{-1} \frac{1}{\sqrt{\zeta}}
\end{equation}
for some $C > 0$ and all sufficiently small $\zeta$. We have, clearly,
\begin{equation}
\tilde{\Gamma}_{\zeta, T}(C_1 M(\zeta), C_2 M(\zeta)) \subseteq \big\{ \text{some } v \in \Ball_{C^{-1} \frac{1}{\sqrt{\zeta}}} \text{ is hit by lightning before time } T \big\},
\end{equation}
which has a probability at most $C_0 \big( C^{-1} \frac{1}{\sqrt{\zeta}} \big)^2 \cdot \zeta \cdot T = C_0 \frac{T}{C^2}$ (for some universal constant $C_0$). This can be made $\leq \delta$ by taking $C$ large enough, which establishes the case $k=0$ (we can choose $C(0, \delta, C_1, C_2, T) = \max(1, \sqrt{C_0 T \delta^{-1}})$).

Now, we assume that the result holds for a certain $k \geq 0$, and we show that it also holds for $k+1$. So let $\delta$, $C_1$, $C_2$, and $T$ be given, and let $C = C(k+1, \delta, C_1, C_2, T) \geq 1$ be a constant that we will fix later. We consider $M(\zeta)$ such that, for all sufficiently small $\zeta$,
\begin{equation} \label{eq:case2_hyp_m}
C m_{k+1}(\zeta) \leq C_1 M(\zeta) < C_2 M(\zeta) \leq C^{-1} m_{k+2}(\zeta).
\end{equation}
To simplify notation, we just write $M$ instead of $M(\zeta)$. We now examine the event $\tilde{\Gamma}_{\zeta, T}(C_1 M, C_2 M)$. We take a small $\eta > 0$ (depending on $\delta$, $C_1$, $C_2$, and $T$, but not on $\zeta$), whose precise value will be specified later, and we define the following events (similar to, but a bit different from the events in the proof of Theorem \ref{thm:case1}).

\begin{itemize}

\item[(i)] $\NET := \big\{$the configuration $\big( \min_{\gamma} \sigma^{\calD(\gamma)}_{t_c - \eta \tilde{\ve}, t_c + \eta \tilde{\ve}} (v) \big)_{v \in \Ball_{C_1 M}}$ has a net $\net$ with mesh $\asymp (M \tilde{M})^{1/2}$, and $\big| \cluster_{\net} \cap \Ann_{\frac{C_1}{2} M, C_1 M} \big| \geq \frac{1}{2} \cdot \big| \Ann_{\frac{C_1}{2} M, C_1 M} \big| \theta(t_c + \eta \tilde{\ve}) \big\}$ (where the minimum is over circuits $\gamma$ in $\Ann_{C_1 M, C_2 M}$). Using again the comparison to percolation with holes (Lemma \ref{lem:stoch_domination} and Lemma \ref{lem:quant_holes}), we observe that, from Proposition \ref{prop:largest_cluster} (and Remark \ref{rem:BCKS_annulus}),
\begin{equation} \label{eq:case2_claim1}
\text{for all $\eta > 0$,} \quad \PP \big( \NET \big) \stackrel[\zeta \searrow 0]{}{\longrightarrow} 1.
\end{equation}
Indeed, we know from \eqref{eq:L_epsilon} and Lemma \ref{lem:est_psi} (combined with \eqref{eq:case2_hyp_m}) that
\begin{equation} \label{eq:L_t_c_eta}
L( t_c + \eta \tilde{\ve} ) \asymp L( t_c + \tilde{\ve} ) = \tilde{M} \ll M^{1 - \upsilon},
\end{equation}
for some $\upsilon > 0$.

\item[(ii)] $\NETB := \big\{$the largest $(t_c + \eta \tilde{\ve})$-occupied cluster in $\Ball_{C_2 M}$ has a volume $\leq 2 \cdot \big| \Ball_{C_2 M} \big| \theta(t_c + \eta \tilde{\ve})$, and contains a net $\net^B$ with mesh $\asymp (M \tilde{M})^{1/2} \big\}$. Observe that $\cluster_{\net}$ in the definition of $\NET$ is contained in the cluster of $\net^B$. It follows immediately from \eqref{eq:largest_cluster} (and \eqref{eq:L_t_c_eta}) that
\begin{equation} \label{eq:case2_claim2}
\text{for all $\eta > 0$,} \quad \PP \big( \NETB \big) \stackrel[\zeta \searrow 0]{}{\longrightarrow} 1.
\end{equation}

\item[(iii)] $\NI := \big\{$no vertex of $\cluster^B$ is ignited during the interval $(t_c - \eta \tilde{\ve}, t_c + \eta \tilde{\ve}) \big\}$, where $\cluster^B$ is the cluster of the net $\net^B$ in the definition of $\NETB$. It follows from \eqref{eq:def_epsilon_tilde} that for $\eta \in (0,1)$,
$$\zeta \cdot 2 \eta \tilde{\ve} \cdot M^2 \theta(t_c + \eta \tilde{\ve}) \leq 2 \eta \cdot \zeta \tilde{\ve} M^2 \theta(t_c + \tilde{\ve}) \leq C'_0 \eta$$
for some $C'_0 > 0$. Hence, $\eta > 0$ can be chosen so small that: for all sufficiently small $\zeta$,
\begin{equation} \label{eq:case2_claim3}
\PP \big( \NI \big) \geq 1 - \frac{\delta}{10}.
\end{equation}
We now fix such an $\eta$.

\item[(iv)] $\OCP := \big\{$in the configuration $\big( \min_{\gamma} \sigma^{\calD(\gamma)}_{t_c - \eta \tilde{\ve}, t_c + \eta \tilde{\ve}} (v) \big)_{v \in \Ball_{C_1 M}}$, the cluster $\cluster_{\net}$ contains a circuit in $\Ann_{\tilde{M}, K \tilde{M}} \big\}$, for some $K > 1$. We claim that we can choose $K$ large enough so that: for all sufficiently small $\zeta$,
\begin{equation} \label{eq:case2_claim4}
\PP \big( \OCP \big) \geq 1 - \frac{\delta}{10}.
\end{equation}
Indeed, this follows from similar reasons as \eqref{eq:claim_OCP2}, since $L(t_c + \eta \tilde{\ve}) \asymp \tilde{M}$ ($\tilde{M}$ and $M$ playing the roles of $m_1$ and $m_2$, respectively). We now fix $K$ such that the above is satisfied.

\item[(v)] $\I := \big\{$there exists a vertex in $\cluster_{\net}$ which gets ignited in the time interval $(t_c + \eta \tilde{\ve}, t_c + \lambda \tilde{\ve}) \big\}$, for some $\lambda > \eta$. We observe that for $\lambda \to \infty$ ($\eta$ being fixed),
$$\zeta \cdot (\lambda - \eta) \tilde{\ve} \cdot M^2 \theta(t_c + \eta \tilde{\ve}) \asymp \lambda$$
(using \eqref{eq:def_epsilon_tilde}). This implies that $\lambda$ can be taken so large that
\begin{equation} \label{eq:case2_claim5}
\PP \big( \I \big) \geq 1 - \frac{\delta}{10},
\end{equation}
and we fix such a $\lambda$.

\item[(vi)] $\VC := \big\{$there exists a $(t_c + \lambda \tilde{\ve})$-vacant circuit in $\Ann_{\frac{\tilde{M}}{K'}, \tilde{M}} \big\}$, for some $K' > 1$. Since $L( t_c + \lambda \tilde{\ve} ) \asymp L( t_c + \tilde{\ve} ) = \tilde{M}$ (using \eqref{eq:L_epsilon}), \eqref{eq:RSW} implies that we can take $K'$ so large that
\begin{equation} \label{eq:case2_claim6}
\PP \big( \VC \big) \geq 1 - \frac{\delta}{10}
\end{equation}
(and we fix such a $K'$).

\end{itemize}

Now, denote by $E$ the intersection of the six events (i)-(vi). If $E$ occurs, then, no matter which circuit $\gamma$ in $\Ann_{C_1 M, C_2 M}$ we choose, the forest fire process in $\calD(\gamma)$ has a burning event which leaves $0$ in an island, whose boundary is a circuit in $\Ann_{\frac{\tilde{M}}{K'}, K \tilde{M}}$. Hence, $\tilde{\Gamma}_{\zeta, T}(C_1 M, C_2 M) \cap E \subseteq \tilde{\Gamma}_{\zeta, T} \big( \frac{\tilde{M}}{K'}, K \tilde{M} \big)$, so
\begin{equation}
\PP \big( \tilde{\Gamma}_{\zeta, T}(C_1 M, C_2 M) \big) \leq \PP \Big( \tilde{\Gamma}_{\zeta, T} \Big( \frac{\tilde{M}}{K'}, K \tilde{M} \Big) \Big) + \PP \big( E^c \big).
\end{equation}
Using \eqref{eq:case2_claim1}, \eqref{eq:case2_claim2}, \eqref{eq:case2_claim3}, \eqref{eq:case2_claim4}, \eqref{eq:case2_claim5} and \eqref{eq:case2_claim6}, we obtain
\begin{equation} \label{eq:case2_induction_step}
\limsup_{\zeta \searrow 0} \PP \big( \tilde{\Gamma}_{\zeta, T}(C_1 M, C_2 M) \big) \leq \limsup_{\zeta \searrow 0} \PP \Big( \tilde{\Gamma}_{\zeta, T} \Big( \frac{\tilde{M}}{K'}, K \tilde{M} \Big) \Big) + 4 \cdot \frac{\delta}{10}.
\end{equation}

Finally, we explain how to choose the constant $C = C(k+1, \delta, C_1, C_2, T)$ mentioned in the beginning of the induction step. For that, we use the induction hypothesis. Note that all the ``auxiliary'' numbers introduced along the way ($\eta$, $\lambda$, $K$ and $K'$) do not depend on this constant $C$. First, we take $C' = C(k, \frac{\delta}{10}, \frac{1}{K'}, K, T)$ produced by the induction hypothesis. Then, we can take $C$ so large that, for all sufficiently small $\zeta$,
\begin{equation}
\bigg[ M \leq \frac{m_{k+2}}{C_2 C} \bigg] \Rightarrow \bigg[ \tilde{M} \leq \frac{m_{k+1}}{K C'} \bigg] \quad \text{and} \quad \bigg[ M \geq \frac{C m_{k+1}}{C_1} \bigg] \Rightarrow \bigg[ \tilde{M} \geq K' C' m_k \bigg].
\end{equation}
Indeed, such a $C$ exists since we have, from \eqref{eq:equiv_mtilde} and \eqref{eq:rel_except_scales} (for $k$ and $k+1$), combined with \eqref{eq:quasi_mult}, \eqref{eq:1arm}, and \eqref{eq:4arms}:
\begin{equation}
\frac{\tilde{M}}{m_{k+1}} \leq \bar{C}_1 \bigg( \frac{M}{m_{k+2}} \bigg)^{\beta_1} \quad \text{and} \quad \frac{\tilde{M}}{m_k} \geq \bar{C}_2 \bigg( \frac{M}{m_{k+1}} \bigg)^{\beta_2}
\end{equation}
for some universal constants $\bar{C}_1, \bar{C}_2, \beta_1, \beta_2 > 0$. We thus obtain that if $M(\zeta)$ satisfies \eqref{eq:case2_hyp_m}, by combining \eqref{eq:case2_induction_step} and the induction hypothesis,
\begin{equation}
\limsup_{\zeta \searrow 0} \PP \big( \tilde{\Gamma}_{\zeta, T}(C_1 M, C_2 M) \big) \leq \frac{\delta}{10} + 4 \cdot \frac{\delta}{10} = \frac{\delta}{2}.
\end{equation}
This completes the proof of Theorem \ref{thm:case2}.
\end{proof}

\section{Discussion: forest fires with recovery} \label{sec:forest_fires}

For the applications in Section \ref{sec:existence_excep_scales}, we focused on forest fires \emph{without} recovery: once a tree is burnt, the vertex where it is located remains vacant forever. However, we want to emphasize that several crucial intermediate results do hold for forest fires \emph{with} recovery as well, due to the quite general coupling result in Section \ref{sec:coupling} (Lemma \ref{lem:stoch_domination}). This raises the hope that the main results in Section \ref{sec:existence_excep_scales}, Theorems \ref{thm:case1} and \ref{thm:case2}, about exceptional scales could be extended, at least partially, to forest fires with recovery.

Let us discuss a natural strategy to prove such an extension, and the main difficulty that needs to be overcome in order to carry it out. To do this, let us return to the heuristic discussion in Section \ref{sec:intro_heuristics}, a few lines below \eqref{eq:intro_heuristics3}. There, we pointed out that the first exceptional scale is of order $\frac{1}{\sqrt{\zeta}}$, where $\zeta$ is the ignition rate. This can easily be extended to forest fires with recovery. However, already in the argument for the next exceptional scale, an obstacle occurs.

An important feature in that argument is that, after burning at time $\tau$, the problem is ``reduced'' to studying a forest fire process in a domain (``island'') with diameter of order $L(\tau)$. The difficulty that comes up in the model with recovery is that, at least theoretically, a significant part of the trees destroyed by the burning at time $\tau$ (and earlier burnings) may recover, and connect the above-mentioned ``island'' with other islands, thus producing a much bigger connected region, which would make the arguments invalid.

In the literature, there is a result, about a process called ``self-destructive percolation'', which suggests, at least ``morally'', that such a substantial recovery does not happen. A variant of this result was conjectured by van den Berg and Brouwer \cite{BB2004} in 2004 (see also \cite{BB2006}). Kiss, Manolescu and Sidoravicius \cite{KMS2015} established the following result around ten years later (it is the main result, Theorem 4, in \cite{KMS2015}).

Consider site percolation on the square lattice $\ZZ^2$. For $n \geq 1$, let $R_n$ be the box $[-2n, 2n] \times [0, n]$, and $S_n$ be the bigger box $[-3n, 3n] \times [0, n]$. First, we consider (ordinary) site percolation with parameter $p_c = p_c^{\textrm{site}}(\ZZ^2)$: this is the \emph{initial configuration} $\omega$. Now, let $\chi \subseteq S_n$ be the set of vertices connected in $S_n$ to both the left and right sides of $S_n$ (i.e. the union of the connected components of horizontal crossings in $S_n$). We denote by $\tilde{\omega}$ the configuration obtained from $\omega$ by setting
\begin{itemize}
\item $\tilde{\omega}_v = 0$ if $v \in \ol \chi = \chi \cup \dout \chi$,
\item and $\tilde{\omega}_v = 1$ otherwise, i.e. if $v \in S_n \setminus \ol \chi$.
\end{itemize}
In other words, if a horizontal crossing of $S_n$ occurs, we destroy (i.e we make vacant) its entire occupied cluster in $S_n$. We also keep vacant all the vertices along the outer boundary of the clusters of such crossings, and we set occupied all the other vertices in $S_n$. Lastly, each vertex vacant at this stage is ``enhanced'', i.e. it becomes, independently of the other ones, occupied with probability $\delta > 0$: this produces the \emph{final configuration} $\tilde{\omega}^{\sigma} := \tilde{\omega} \vee \sigma$, where $\sigma$ is independent of $\omega$ and has distribution $\PP_{\delta}$. We have:

\begin{theorem}[Theorem 4 of \cite{KMS2015}] \label{thm:KMS}
There exist constants $\delta, \lambda, C > 0$ such that: for all $n \geq 1$,
\begin{equation}
\PP \big( \omega \in \Ch(S_n) \text{ and } \tilde{\omega}^{\sigma} \in \Cv(R_n) \big) \leq C n^{-\lambda}.
\end{equation}
\end{theorem}

An exact analog holds on the triangular lattice $\TT$ (starting instead with the parameter $p_c = p_c^{\textrm{site}}(\TT)$). To extend the results of Section \ref{sec:existence_excep_scales} to forest fires with recovery, what we would need is a suitable analog of Theorem \ref{thm:KMS} where, roughly speaking, the initial configuration is replaced by a typical configuration at or near $t_c$ of the forest fire process with recovery. Such a result does not simply follow by using the kind of domination arguments introduced in Section \ref{sec:coupling}, and used in Section \ref{sec:existence_excep_scales} (e.g. for \eqref{eq:claim_NET2} and \eqref{eq:claim_OCP2}): the potential recoveries cause a delicate problem. This forms (part of) the work of a subsequent paper.

\bibliographystyle{plain}
\bibliography{FP_fires}

\end{document}